\documentclass[]{amsart}

\AtBeginDocument{%
   \def\MR#1{}
}

\usepackage[english]{babel}
\usepackage[T1]{fontenc}
\usepackage[utf8]{inputenc}
\usepackage[a4paper,top=3cm,bottom=3cm,left=3cm,right=3cm,marginparwidth=2cm]{geometry}

\usepackage{tikz}
\usepackage{tikz-cd}
\tikzcdset{arrow style= tikz, diagrams={>=stealth'}}
\usetikzlibrary{patterns}
\usetikzlibrary{arrows}

\tikzset{
>=stealth',
coi/.style={circle, thick, draw=black, fill=white, inner sep=1.5pt, label=center:\Large{$\cdot$}},
coii/.style={circle, thick, draw=black, fill=white, inner sep=1.5pt},
generic/.style={circle, fill, inner sep=1.5pt}}

\usepackage{amssymb}
\usepackage{amsmath}
\usepackage{amsfonts}
\usepackage{mathtools}
\usepackage{mathrsfs}

\usepackage{float}
\floatstyle{plaintop}
\restylefloat{table}

\usepackage[colorlinks=true, linkcolor = blue, citecolor=blue]{hyperref}
\usepackage[]{cleveref}
\usepackage[all,pdf]{xy}
\usepackage{color}

\theoremstyle{plain}
\newtheorem{theorem}{Theorem}[section]
\newtheorem{lemma}[theorem]{Lemma}
\newtheorem{corollary}[theorem]{Corollary}

\newtheorem{proposition}[theorem]{Proposition}

\theoremstyle{definition}
\newtheorem{definition}[theorem]{Definition}

\newtheorem{remark}[theorem]{Remark}
\newtheorem{example}[theorem]{Example}
\newtheorem{problem}[theorem]{Problem}
\numberwithin{equation}{section}

\newcommand{\M}{\mathbf{M}}
\renewcommand{\S}{\mathbf{S}}

\DeclareMathOperator{\grp}{\pi}
\newcommand{\fzmu}{\mu^\textup{FZ}} 
\newcommand{\prefzmu}{\widetilde{\mu}^\textup{FZ}} 

\newcommand{\cf}{{\em cf.}\ }

\newcommand{\ncl}[1]{\langle\!\langle #1 \rangle\!\rangle}

\newcommand{\trht}{\mathbf{1}} 
\newcommand{\punc}{\mathbf{P}} 
\newcommand{\tarc}{\mathbf{A}^{\bowtie}} 
\newcommand{\tagcom}{\Delta^{\bowtie}} 
\newcommand{\coi}{\mathrm{I}} 
\newcommand{\coii}{\mathrm{II}} 

\title[Mutations of quivers with $2$-cycles]
{Mutations of quivers with $2$-cycles}

\author{Fang Li}
\address{School of Mathematical Sciences, Zhejiang University, Hangzhou, Zhejiang 310058, China}
\email{fangli@zju.edu.cn}

\author{Siyang Liu}
\address{School of Mathematics, Hangzhou Normal University, Hangzhou, Zhejiang 311121, China}
\email{siyangliu@hznu.edu.cn}

\author{Lang Mou}
\address{Department of Mathematics, Mathematical Sciences Building, One Shields Ave. University of California
Davis, CA 95616, USA}
\email{lmou.math@gmail.com}

\author{Jie Pan}
\address{Department of Mathematics, Faculty of Sciences, University of Sherbrooke, Sherbrooke, Quebec, Canada}
\email{jie.pan@usherbrooke.ca}

\thanks{\textit{Mathematics Subject Classification(2020): 13F60, 16G20, 55P10}}
\thanks{\textit{Keywords}: quiver, 2-cycles, homotopy, cluster algebra, mutation}

\begin{document}

\begin{abstract}
    We develop a mutation theory for quivers with oriented 2-cycles using a structure called a homotopy, defined as a normal subgroupoid of the quiver's fundamental groupoid. This framework extends Fomin--Zelevinsky mutations of $2$-acyclic quivers and yields involutive mutations that preserve the fundamental groupoid quotient by the homotopy. It generalizes orbit mutations arising from quiver coverings and allows for infinite mutation sequences even when orbit mutations are obstructed. We further construct quivers with homotopies from triangulations of marked surfaces with colored punctures, and prove that flips correspond to mutations, extending the Fomin--Shapiro--Thurston model to the setting with 2-cycles.
\end{abstract}

\maketitle

\tableofcontents

\section{Introduction}\label{section: introduction}

The Fomin--Zelevinsky mutation (FZ-mutation) of quivers provides the combinatorial foundation of cluster algebra theory \cite{FZI}, but it applies only to $2$-acyclic quivers. This limitation excludes natural examples arising in representation theory \cite{Palu,BO} and physics \cite{DM96,BD02}. In this paper, we propose a framework for mutations of quivers with oriented 2-cycles, generalizing the FZ-mutation.

To extend mutations to quivers with 2-cycles, we introduce an additional combinatorial structure called a \emph{homotopy}. A homotopy is an equivalence relation on paths in the underlying graph of the quiver; algebraically, it corresponds to a normal subgroupoid of the fundamental groupoid. Unlike the classical FZ-mutation, which depends only on the quiver structure, our mutation rules depend on the chosen homotopy.

We call a homotopy $H$ on $Q$ \emph{reduced} if no $2$-cycle lies in $H$. For any loop-free quiver with reduced homotopy $(Q, H)$ and a vertex $k$ of $Q$, we construct the \emph{mutation} $\mu_k(Q, H)$ as another quiver with reduced homotopy $(Q^\dag, H^\dag)$; see \Cref{subsec: mutations with homotopy} for the precise definition.

A key property of this construction, as in the 2-acyclic case, is that mutation is involutive on quivers with reduced homotopies (\Cref{prop: mutation involutive}). This involutivity enables natural extensions of classical cluster-theoretic constructions. Furthermore, the fundamental groupoid quotient $\pi(Q)/H$ is invariant under mutation (\Cref{rmk: quotient groupoid invariant}), yielding an isomorphism $\pi(Q)/H \cong \pi(Q^\dag)/H^\dag$.

The classical FZ-mutation $\fzmu_k(Q)$ proceeds via a pre-mutation $\prefzmu_k(Q)$, which may create new $2$-cycles. This step extends naturally to quivers with $2$-cycles. In our framework, the homotopy $H$ determines which of the newly created $2$-cycles in $\prefzmu_k(Q)$ are removed. For example, with the trivial homotopy, no $2$-cycles are removed; with the maximal homotopy, all newly created $2$-cycles are removed. Thus, the FZ-mutation of a $2$-acyclic quiver is recovered precisely as mutation with maximal homotopy (\Cref{prop: maximal H induces FZ mutation}). See \Cref{fig: intro mutation} for an example illustrating how different homotopies lead to different quiver mutations.

\begin{figure}[ht]
    \centering
    \begin{tikzpicture}[scale=0.5]
        \begin{scope}[shift={(0,0)}]
            \node[] (1) at ({-sqrt(3)}, -1) {$1$};
            \node[] (3) at ({sqrt(3)}, -1) {$3$};
            \node[] (2) at (0,2) {$2$};

            \draw[->] (1) to node[left]{$c$} (2);
            \draw[->] (2) to node[right]{$b$} (3);
            \draw[->] (3) to node[below]{$a$} (1);
        \end{scope}

        \draw[->] (4,{sqrt(3)}) to node[above, sloped]{$\mu_2$} (6, {sqrt(3)+1});

        \draw[->] (4, {-sqrt(3)}) to node[above, sloped]{$\mu_2$} (6, {-sqrt(3)-1});

        \begin{scope}[shift={(10,{sqrt(12)})}]
            \node[] (1) at ({-sqrt(3)}, -1) {$1$};
            \node[] (3) at ({sqrt(3)}, -1) {$3$};
            \node[] (2) at (0,2) {$2$};

            \draw[<-] (1) to node[left]{$c^\star$} (2);
            \draw[<-] (2) to node[right]{$b^\star$}(3);
            \draw[->, bend left=10] (3) to node[below]{$a$} (1);
            \draw[->, bend left=10] (1) to node[above]{$[bc]$}(3);
        \end{scope}

        \begin{scope}[shift={(10,{-sqrt(12)})}]
            \node[] (1) at ({-sqrt(3)}, -1) {$1$};
            \node[] (3) at ({sqrt(3)}, -1) {$3$};
            \node[] (2) at (0,2) {$2$};

            \draw[<-] (1) to node[left]{$c^\star$}(2);
            \draw[<-] (2) to node[right]{$b^\star$}(3);
        \end{scope}

        \node[] at (19,{sqrt(12)}) {mutation with the trivial homotopy};
        
        \node[] at (18,{-sqrt(12)}) {mutation with $H=\ncl{abc}$};
    \end{tikzpicture}
    \caption{Mutations with different homotopies result in different quivers.}
    \label{fig: intro mutation}
\end{figure}
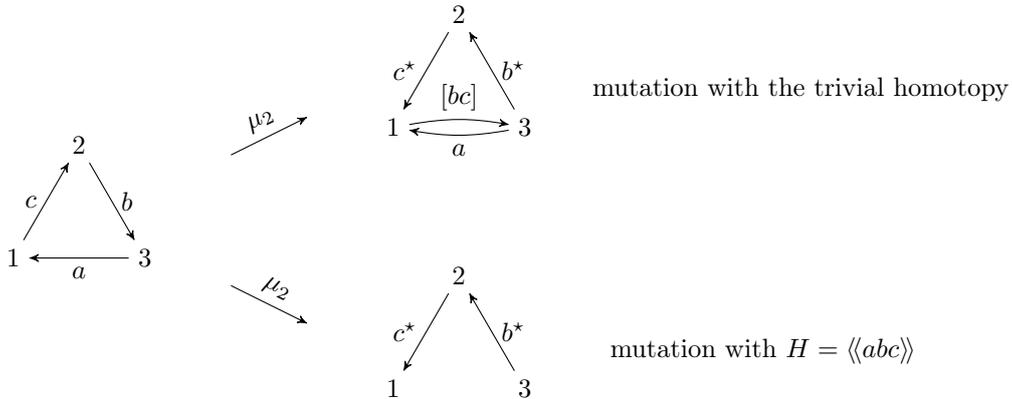

\subsection{Main results}

For orientation, we record here the main statements proved in this paper. Undefined terms are recalled later in the introduction.

First, we show that mutation with a reduced homotopy is \emph{involutive} and preserves the quotient groupoid $\pi(Q)/H$; in particular, for $2$-acyclic $Q$ the maximal choice $H(i,i)=\pi_1(Q,i)$ for all $i\in Q_0$ recovers the classical FZ-mutation (\Cref{prop: mutation involutive}, \Cref{rmk: quotient groupoid invariant}, \Cref{prop: maximal H induces FZ mutation}).

We relate our construction to coverings: for weakly admissible coverings $p\colon \widetilde Q \rightarrow Q$, orbit mutations of $\widetilde Q$ agree with mutations of $Q$ with the homotopy determined by the covering data; moreover, for acyclic $Q$ even the trivial homotopy yields indefinite FZ-mutation sequences, giving a lower bound on fundamental groups across the mutation class (\Cref{prop: orbit mutation compatible}, \Cref{thm: any homotopy induces fz mutation acyclic}, \Cref{cor: lower bound fund group}).

From triangulations of marked bordered surfaces with $\coii$-punctures we construct quivers with homotopies $(Q(T), H(T))$ associated to triangulations $T$ and establish a \emph{flip/mutation} correspondence in the 2-cycle setting (\Cref{thm: mutation and flip}). We also give a topological realization of $(Q(T), H(T))$ via $2$-complexes built from T (\Cref{def: 2-complex X(T)}, \Cref{prop: fund grp surface}).

Finally, in \Cref{section: 2-cycle cluster algebras}, we associate cluster algebras $\mathcal A(Q,H)$ to $(Q, H)$ and verify the Laurent phenomenon in the cases arising from weakly admissible coverings that admit indefinite orbit mutations (\Cref{thm: laurent}).

\subsection{Generalizing orbit mutations}

Our motivation partly comes from the theory of quiver coverings. A covering of quivers $p\colon \widetilde Q \rightarrow Q$ is a covering map of the underlying graphs which respects the orientation of arrows. Coverings closely related to cluster theory are globally admissible coverings, which are defined between 2-acyclic (valued) quivers and compatible with FZ-mutation. They are typically used both to produce new cluster patterns from existing ones \cite[Section 4.4]{FWZ-II}) and to reduce complex cluster patterns to simpler ones \cite{Dup08,Kel13,HL}. In this paper, we will see that loosening the admissibility condition leads to interesting extensions.

We say that a normal covering $p$ is \emph{weakly admissible} if $\widetilde Q$ is $2$-acyclic and no arrow connects vertices lying in the same orbit of deck transformations (so loops are excluded, but 2-cycles may remain in $Q$). In this case, a sequence of FZ-mutations applied once at each vertex in an orbit can be performed on $\widetilde Q$, and the resulting composition is called an \emph{orbit mutation}. This composition is independent of the order in which the vertices are mutated. In \Cref{subsec: homotopy and orbit mutation}, we show that the quotient quiver after an orbit mutation coincides with the mutation with homotopy associated to the covering data.

However, orbit mutations may destroy weak admissibility, thereby obstructing further orbit mutations. In contrast, mutations of quivers with homotopies can always be applied indefinitely. This suggests that our homotopy-based mutations offer a natural extension beyond orbit mutations arising from coverings.

\subsection{Quivers with homotopies arising from surface triangulations}\label{subsec: qh from surface}

In \Cref{section: marked surfaces}, we construct a class of quivers with homotopies arising from triangulations of marked bordered surfaces with colored punctures. We generalize the model of Fomin, Shapiro and Thurston \cite{FST08} by assigning two distinct colors, $\coi$ and $\coii$, to interior marked points (known as \emph{punctures}). Punctures of color $\coii$ give rise to $2$-cycles in the associated quivers for certain triangulations. 

Our main result in this direction, \Cref{thm: mutation and flip}, establishes that mutations of quivers with homotopies are compatible with flips of triangulations. We note that mutations of this class of quivers have previously appeared in the work of Paquette--Schiffler \cite{PS} and Kaufman \cite{Kaufman}. Our new contributions are: (i) constructing homotopies directly from triangulations, and (ii) showing that quiver mutations are determined entirely by the corresponding homotopies. In the case where $\coii$-punctures are absent, our flip/mutation correspondence (\Cref{thm: mutation and flip}) for quivers with homotopies is analogous to Labardini-Fragoso's correspondence on quivers with potentials \cite[Theorem 30]{LF09}. Certain quivers with potentials---emphasizing quivers with $2$-cycles arising from triangulations---have been studied recently in \cite{LiPeng}. Some quivers with $2$-cycles (and loops) and their mutations have appeared in a closely related setting of dg algebras \cite{Qiu, CHQ}. See also a comparison between homotopies and potentials in \Cref{subsec: compare homotopy potential}.

\subsection{Cluster algebras associated to quivers with $2$-cycles}

The cluster exchange relations in Fomin--Zelevinsky cluster algebras \cite{FZI} naturally extend to the case of quivers with $2$-cycles:
\begin{equation}\label{eq: exchange relation intro}
    x_k'x_k = \prod_{i} x_i^{p_{ik}} + \prod_{i} x_i^{p_{ki}},
\end{equation}
where $p_{ij}$ denotes the number of arrows from vertex $i$ to vertex $j$ in a quiver $Q$ and in both products $i$ runs through vertices in $Q$. The involutivity of mutations allows us to define a cluster algebra $\mathcal A(Q, H)$ associated to any quiver with homotopy $(Q, H)$, formalized in \Cref{section: 2-cycle cluster algebras}. The cluster variables are generated as rational functions from exchange relations \eqref{eq: exchange relation intro} associated to quivers obtained from $Q$ by mutations with homotopies.

A crucial difference with ordinary cluster algebras is that the two summands on the right-hand side of \eqref{eq: exchange relation intro} may share non-trivial common factors due to the presence of $2$-cycles. This prevents the direct extension of Fomin--Zelevinsky's proof of the \emph{Laurent phenomenon} \cite{FZI}. Nevertheless, when $(Q, H)$ arises from a weakly admissible quiver covering that admits indefinite orbit mutations, the Laurent phenomenon does hold (\Cref{thm: laurent}). We leave the general case for future research; see \Cref{section: final remarks} for final remarks.

\section{Quiver mutation and covering of quivers}

We first recall mutations of quivers without oriented $2$-cycles in \Cref{subsec: fz mutation} and then study how mutations interact with coverings of quivers in \Cref{subsec: quiver covering,subsec: orbit mutation}.

\subsection{The Fomin--Zelevinsky quiver mutation}\label{subsec: fz mutation}

A \emph{quiver} $Q = (Q_0, Q_1, s, t)$ consists of a set $Q_0$ of \emph{vertices} and a set $Q_1$ of \emph{arrows}, together with two maps $s, t \colon Q_1 \rightarrow Q_0$ assigning to each arrow $a$ its \emph{source} $s(a)$ and \emph{target} $t(a)$ respectively. A \emph{loop} in $Q$ is an arrow $a\in Q_1$ such that $s(a) = t(a)$. Throughout this paper, we will restrict attention to \emph{locally finite} quivers, meaning that both $s^{-1}(v)$ and $t^{-1}(v)$ are finite sets for any $v\in Q_0$.

A quiver is often depicted as a directed graph, where each arrow is drawn as an edge pointing from its source vertex to its target vertex.

We say that a quiver $Q$ is \emph{$2$-acyclic} if it contains no (oriented) \emph{$2$-cycles}, that is, there do not exist $\alpha, \beta\in Q_1$ such that $t(\alpha) = s(\beta)$ and $s(\alpha) = t(\beta)$. In particular, a 2-acyclic quiver has no loops, since a loop forms a $2$-cycle with itself.

In defining mutations of $2$-acyclic quivers and their associated cluster algebras, the relevant data is not the quiver itself, but the number of arrows between each pair of vertices. Accordingly, a $2$-acyclic quiver $Q$ can be represented by a skew-symmetric integer matrix $B = B(Q) = (b_{i,j}) \in \mathbb Z^{Q_0\times Q_0}$, where for any $(i, j)\in Q_0\times Q_0$,
\begin{equation}\label{eq: exchange matrices of quivers}
    b_{i,j} \coloneqq |\{a\in Q_1 \mid s(a) = i, t(a) = j\}| - |\{a\in Q_1 \mid s(a) = j, t(a) = i\}|.
\end{equation}

Fomin and Zelevinsky \cite{FZI} introduced operations called \emph{mutations} on skew-symmetric matrices (actually more generally on skew-symmetrizable matrices and even sign-skew-symmetric matrices). These operations are fundamental to their theory of cluster algebras. In terms of $2$-acyclic quivers, mutations can be described as follows.

\begin{definition}\label{def: fz mutation}
    Let $Q$ be a $2$-acyclic quiver. The \emph{Fomin--Zelevinsky mutation} (FZ-mutation) of $Q$ at direction $k\in Q_0$ is the $2$-acyclic quiver $\fzmu_k(Q)$ obtained from $Q$ by
    \begin{enumerate}
        \item for each pair $(\alpha, \beta)\in Q_1 \times Q_1$ such that $s(\alpha) = t(\beta) = k$, adding a new arrow $[\alpha\beta]$ from $s(\beta)$ to $t(\alpha)$;
        \item reversing all arrows in $Q_1$ incident to $k$, that is, replacing each $\alpha\in Q_1$ such that $s(\alpha) = k$ or $t(\alpha) = k$ with a new arrow $\alpha^\star$ such that $s(\alpha^\star) = t(\alpha)$ and $t(\alpha^\star) = s(\alpha)$;
        \item deleting a maximal collection of $2$-cycles created after the second step.
    \end{enumerate}
\end{definition}

The operation that only performs the first two steps is called a \emph{pre-mutation}. It can be extended to operate on quivers with $2$-cycles as follows.

\begin{definition}\label{def: fz premutation}
    Let $Q$ be a loop-free quiver. The \emph{pre-mutation} of $Q$ at direction $k\in Q_0$ is another loop-free quiver $\prefzmu_k(Q)$ obtained from $Q$ by
    \begin{enumerate}
        \item for each pair $(\alpha, \beta)\in Q_1 \times Q_1$ such that $s(\alpha) = t(\beta) = k$ and $s(\beta)\neq t(\alpha)$, adding a new arrow $[\alpha\beta]$ from $s(\beta)$ to $t(\alpha)$;
        \item reversing all arrows in $Q_1$ incident to $k$.
    \end{enumerate}
\end{definition}

There is a choice made in the last step $(3)$ of $\fzmu_k$. Nevertheless the quiver $\fzmu_k(Q)$ is well-defined in the sense that the matrix $\fzmu_k(B) \coloneqq B(\fzmu_k(Q))$ is uniquely determined. It is easy to verify combinatorially that the operation $\fzmu_k$ is involutive.

\subsection{Coverings of quivers}\label{subsec: quiver covering}

\begin{definition}\label{def: covering between quivers}
    A \emph{covering (map)} $p = (p_0, p_1)$ from a quiver $\widetilde Q$ to another quiver $Q$ (denoted as $p\colon \widetilde Q \rightarrow Q$) consists of a map $p_0 \colon \widetilde Q_0 \rightarrow Q_0$ and a map $p_1 \colon \widetilde Q_1 \rightarrow Q_1$ such that 
    \begin{enumerate}
        \item $(s\circ p_1) (a) = (p_0 \circ s) (a)$ and $(t \circ p_1) (a) = (p_0 \circ t)(a)$ for any arrow $a\in \widetilde Q_1$;
        \item the map $p_0$ on vertices is surjective;
        \item for any $v\in \widetilde Q_0$, the map $p_1$ induces a bijection from $\{a\in \widetilde Q_1 \mid t(a) = v \}$ to $\{a\in Q_1 \mid t(a) = p_0(v)\}$ and a bijection from $\{a\in \widetilde Q_1 \mid s(a) = v \}$ to $\{a\in Q_1 \mid s(a) = p_0(v)\}$.
    \end{enumerate}
\end{definition}

\begin{remark}
    A pair $p = (p_0, p_1)$ as in the above definition only satisfying (1) is called a \emph{map} or \emph{morphism} between quivers. It is called an \emph{isomorphism} if both $p_0$ and $p_1$ are bijections. We will simply use $p$ for both $p_0$ and $p_1$ when it causes no ambiguity.
\end{remark}

A covering $p$ between quivers induces in the obvious way a (topological) covering map between the underlying unoriented graphs (realized as $1$-complexes). Various constructions and properties in those contexts carry over to quivers. For example, we have the group $\operatorname{Deck}(p)$ of deck transformations of the topological covering $p$, which can also be identified with the group of automorphisms of $\widetilde Q$ commuting with $p\colon \widetilde Q \rightarrow Q$. In addition, for any connected quiver $Q$, there always exists a (connected) tree quiver $\widehat Q$ as the \emph{universal cover} with a covering $\hat p \colon \widehat Q \rightarrow Q$; see for example \cite[1.A]{Hatcher} and \cite[Construction 2.6]{HL}.

There is a simple description of elements in $\operatorname{Deck}(\hat p)$. Let $v$ be a vertex in $Q_0$ and $v'$ be any vertex in $\hat p^{-1}(v)$. For any $v''\in \hat p^{-1}(v)$, there is a unique reduced walk $w$ from $v'$ to $v''$ since $\widehat Q$ is a tree. Sending $v'$ to $v''$ uniquely determines a deck transformation for the universal cover. The covering map $\hat p$ sends $w$ to a reduced walk $\hat p(w)$ in $Q$ from $v$ to itself. It is called a \emph{cyclic walk} in $Q$ as in \Cref{def: cyclic walk}. In this way, the group $\operatorname{Deck}(\hat p)$ is isomorphic to the fundamental group $\pi_1(Q, v)$ based at $v\in Q_0$ where we view $Q$ as a $1$-complex.

\begin{definition}\label{def: regular cover}
    A covering $p \colon \widetilde Q \rightarrow Q$ between connected quivers $\widetilde Q$ and $Q$ is called \emph{regular (or normal)} if $\operatorname{Deck}(p)$ acts on $p^{-1}(v)$ transitively for any $v\in Q_0$. 
\end{definition}

\begin{remark}
    Any covering $p$ induces a group homomorphism $p_* \colon \pi_1(\widetilde Q, \tilde v) \rightarrow \pi_1(Q, v)$ where $p(\tilde v) = v$. Being regular is equivalent to that the image of $p_*$ is a normal subgroup of $\pi_1(Q, v)$.
\end{remark}

\begin{example}\label{ex: loop-free covering}
    The following covering is non-regular because there is no deck transformation sending the upper left vertex $2$ to the lower left vertex $2$.
    \[
        \widetilde Q = 
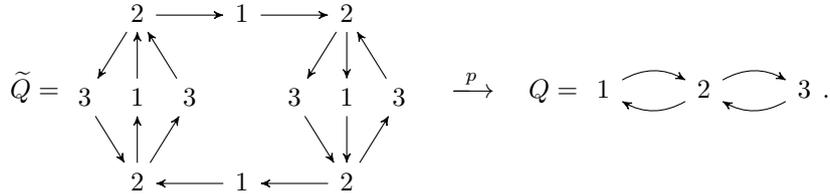
\begin{tikzcd}[column sep = tiny]
        	& 2 && 1 && 2 \\
        	3 & 1 & 3 && 3 & 1 & 3 \\
        	& 2 && 1 && 2
        	\arrow[from=1-2, to=2-1]
        	\arrow[from=2-1, to=3-2]
        	\arrow[from=3-2, to=2-2]
        	\arrow[from=2-2, to=1-2]
        	\arrow[from=3-2, to=2-3]
        	\arrow[from=2-3, to=1-2]
        	\arrow[from=1-2, to=1-4]
        	\arrow[from=1-4, to=1-6]
        	\arrow[from=1-6, to=2-6]
        	\arrow[from=1-6, to=2-5]
        	\arrow[from=2-5, to=3-6]
        	\arrow[from=2-6, to=3-6]
        	\arrow[from=3-6, to=2-7]
        	\arrow[from=2-7, to=1-6]
        	\arrow[from=3-6, to=3-4]
        	\arrow[from=3-4, to=3-2]
        \end{tikzcd}
        \quad
        \xlongrightarrow{p}
        \quad
        Q = \begin{tikzcd}
            1 \ar[r, bend left] & 2 \ar[l, bend left] \ar[r, bend left] & 3 \ar[l, bend left]
        \end{tikzcd}.
    \]
\end{example}

The following proposition is standard in topology of graphs; see for example \cite{Hatcher}.

\begin{proposition}[Galois correspondence]\label{prop: galois correspondence}
    There is a bijection between regular coverings of $Q$ (up to isomorphisms relative to $Q$) and normal subgroups of $\pi_1(Q, v)$ for any $v\in Q_0$. If $p \colon \widetilde Q \rightarrow Q$ is a regular covering, the fundamental group $\pi_1(Q, v)$ for any $v\in Q_0$ acts on $\widetilde Q$ by deck transformations. This gives rise to a surjective group homomorphism from $\pi_1(Q, v)$ to $\operatorname{Deck}(p)$ whose kernel is the normal subgroup $p_*(\pi_1(\widetilde Q, \tilde v))$ in $\pi_1(Q, v)$ for any $\tilde v \in p^{-1}(v)$. 
\end{proposition}

\begin{example}\label{ex: universal abelian cover}
    A covering $p \colon \widetilde Q \rightarrow Q$ is called \emph{abelian} if it is regular and $\operatorname{Deck}(p)$ is an abelian group. For each (connected) quiver $Q$, there is a \emph{universal abelian covering} that is a covering of any other abelian covering of $Q$. It corresponds to the normal subgroup of commutators $[\pi_1(Q,v),\pi_1(Q,v)]$ for any $v\in Q_0$.
\end{example}

Instead of starting with a covering, one may consider a free action of a group $\Gamma$ on a (connected) quiver $\widetilde Q$. This in particular implies that each element in $\Gamma$ gives an automorphism on $\widetilde Q$ and the induced group actions on $\widetilde Q_0$ and $\widetilde Q_1$ are both free. The quotient quiver $Q = \widetilde Q/\Gamma$ is defined to have vertices and arrows both as orbit sets. Then the natural quotient map $p \colon \widetilde Q \rightarrow Q$ is a normal covering with $\Gamma = \operatorname{Deck}(p)$.

\subsection{Orbit mutations}\label{subsec: orbit mutation}
From now on, all coverings of quivers are assumed to be regular.

We examine how FZ-mutations interact with quiver coverings. The main observation is that orbit mutations on the covering quiver are well-defined if the covering admits certain admissible conditions in \Cref{def: weakly admissible}. How the quotient quiver (which may contain $2$-cycles) changes is the guiding principle of their mutation rules, which will be incorporated into mutations of quivers with homotopies in \Cref{section: quivers with homotopies}.

\begin{definition}\label{def: weakly admissible}
    A covering $p\colon \widetilde Q \rightarrow Q$ is called \emph{admissible} if both $\widetilde Q$ and $Q$ are $2$-acyclic. It is called \emph{weakly admissible} if $\widetilde Q$ is $2$-acyclic and $Q$ is loop-free.
\end{definition}

\begin{definition}[Orbit mutation]\label{def: orbit mutation}
    Let $p\colon \widetilde Q \rightarrow Q$ be a weakly admissible covering. The \emph{orbit pre-mutation} $\prefzmu_{[k]}(\widetilde Q)$ of $\widetilde Q$ at direction $k\in Q_0$ is obtained from $\widetilde Q$ by
    \begin{enumerate}
        \item for each pair $(\alpha, \beta)\in \widetilde Q_1 \times \widetilde Q_1$ such that $s(\alpha) = t(\beta) \in p^{-1}(k)$, adding a new arrow $[\alpha\beta]$ from $s(\beta)$ to $t(\alpha)$;
        \item reversing all arrows of $\widetilde Q$ incident to any vertex in $p^{-1}(k)$.
    \end{enumerate}
    The \emph{orbit mutation} $\fzmu_{[k]}(\widetilde Q)$ is obtained by further
    \begin{enumerate}
        \item[(3)] deleting a maximal collection of $2$-cycles in $\prefzmu_{[k]}(\widetilde Q)$.
    \end{enumerate}
\end{definition}

\begin{remark}
    The loop-free condition on $Q$ ensures that the step (1) in the above definition is well-defined, as there is no arrow in $\widetilde Q$ between two vertices in $p^{-1}(k)$.
\end{remark}

For any quiver $Q$, denote by $Q^{\operatorname{lf}}$ the quiver obtained from $Q$ by deleting all loops. If $Q$ carries a free $\Gamma$-action, denote by $Q^{\operatorname{\Gamma-lf}}$ the quiver obtained from $Q$ by deleting any arrow between two vertices in the same $\Gamma$-orbit. Such an arrow will be referred to as a \emph{$\Gamma$-loop}. The superscripts $\operatorname{lf}$ and $\operatorname{\Gamma-lf}$ stand for loop free and $\Gamma$-loop free respectively.

\begin{lemma}\label{lemma: delete loops}
    Let $p\colon \widetilde Q \rightarrow Q$ be a weakly admissible quiver covering and $k\in Q_0$. Then the group $\Gamma = \operatorname{Deck}(p)$ acts freely on the quivers 
    \[
        \prefzmu_{[k]}(\widetilde Q) \quad \text{and} \quad \prefzmu_{[k]}(\widetilde Q)^{\operatorname{\Gamma-lf}}
    \]
    by automorphisms. We have
    \[
        \prefzmu_{[k]}(\widetilde Q)^{\operatorname{\Gamma-lf}}/\Gamma = \left(\prefzmu_{[k]}(\widetilde Q)/\Gamma\right)^{\operatorname{lf}} =   \prefzmu_k(Q).  
    \]
\end{lemma}

\begin{proof}
    The action of $\Gamma$ on the vertex set remains unchanged. For any arrow $\alpha$ incident to $p^{-1}(k)$ and $\tau\in \Gamma$, we let $\tau(\alpha^\star) = (\tau(\alpha))^\star$ where $(\cdot)^\star$ denotes the reversed arrow in the pre-mutation. For a newly created arrow $[\alpha \beta]$, let $\tau([\alpha\beta]) = [\tau(\alpha)\tau(\beta)]$. It is straightforward that in this way $\Gamma$ acts freely on $\prefzmu_{[k]}(\widetilde Q)$ by automorphisms.

    The quiver $\widetilde Q$ does not have any $\Gamma$-loop. Any $\Gamma$-loop in $\prefzmu_{[k]}(\widetilde Q)$ is of the form $[\alpha\beta]$ arising from the path
    \[
        \begin{tikzcd}
            i \ar[r, "\beta"] & \tilde k \ar[r, "\alpha"] & j
        \end{tikzcd}
    \]
    in $\widetilde Q$ where $p(\tilde k) = k$ and $i$ and $j$ are in the same $\Gamma$-orbit. The loops in $\prefzmu_{[k]}(\widetilde Q)/\Gamma$ can only come from such $[\alpha\beta]$. Therefore we have
    \[
        \prefzmu_{[k]}(\widetilde Q)^{\operatorname{\Gamma-lf}}/\Gamma = \left(\prefzmu_{[k]}(\widetilde Q)/\Gamma\right)^{\operatorname{lf}}.
    \]

    Without the composite arrows, the second equality is clear as we simply reverse arrows in $\widetilde Q$ incident to vertices in $p^{-1}(k)$ and reverse arrows in $Q$ incident to $k$. Now we consider newly created composite arrows. Let $i\in \widetilde Q_0\setminus p^{-1}(k)$. The new arrows of the form $[\alpha\beta]$ incident to $i$ are a set of representatives of $\Gamma$-orbits on the set of new arrows incident to $\Gamma\cdot i$. Delete any $\Gamma$-loop incident to $i$ for every $i$. Now sending (the $\Gamma$-orbit of) $[\alpha\beta]$ to $[p(\alpha)p(\beta)]\in \prefzmu_{k}(Q)$ induces a quiver map from $\left(\prefzmu_{[k]}(\widetilde Q)/\Gamma\right)^{\operatorname{lf}}$ to $\prefzmu_{k}(Q)$. By directly computing the number of arrows incident to $i$ and $p(i)$, this map is a local isomorphism and thus an isomorphism because of the bijection on vertices.
\end{proof}

Between two vertices $i$ and $j$ that are not in $p^{-1}(k)$, there is a (choice of) cancellation of $2$-cycles in the last step of the orbit mutation $\fzmu_{[k]}$. We would like this choice to be consistent with the action of $\Gamma$.

\begin{lemma}\label{lemma: remove 2 cycle orbit}
    There exists a choice of a maximal collection of 2-cycles 
    \[
        R_{i, j} \subseteq (\prefzmu_{[k]}(\widetilde Q))_1(i, j) \times (\prefzmu_{[k]}(\widetilde Q))_1(j, i)
    \]
    between any unordered pair $\{i, j\}\subseteq \widetilde Q_0\setminus p^{-1}(k)$ with $i\neq j$ such that for any $\tau\in \Gamma$, we have
    \[
        R_{\tau(i), \tau(j)} = \{(\tau(a), \tau(b)) \mid (a, b)\in R_{i,j}\}.
    \]
\end{lemma}

\begin{proof}
    Let $\Gamma$ act on the set of all unordered pairs of vertices in $\widetilde Q_0$. For $\{i, j\}$ with a free $\Gamma$-orbit, one can choose a maximal collection of $2$-cycles between $i$ and $j$ and that determines the choice for any other $\{\tau(i), \tau(j)\}$.
    
    Suppose that $\{i, j\}$ does not have a free $\Gamma$-orbit. This can happen only if $i$ and $j$ belong to the same fiber of the covering $p$. In this case, there is a unique element $\tau\in \Gamma$ such that $\tau(i) = j$ and $\tau(j) = i$. Since $p$ is assumed to be weakly admissible, there is no arrow between $i$ and $j$ in $\widetilde Q$. We take the collection of $2$-cycles
    \[
        \{([ab], [\tau(a)\tau(b)]) \mid t(a) = i, s(b) = j, s(a) = t(b)\in p^{-1}(k), a\in \widetilde Q_1, b\in \widetilde Q_1\}.
    \]
    In fact, these are all the newly created arrows between $i$ and $j$ in the pre-mutation. This determines in this case the collection of $2$-cycles between $\gamma(i)$ and $\gamma(j)$ for any other $\gamma\in \Gamma$.
\end{proof}

The following proposition is a direct corollary of \Cref{lemma: remove 2 cycle orbit}.

\begin{proposition}\label{prop: mutation of a covering}
    With the choices made in \Cref{lemma: remove 2 cycle orbit}, the group $\Gamma = \mathrm{Deck}(p)$ acts freely on the orbit mutation $\fzmu_{[k]}(\widetilde Q)$ by quiver automorphisms, inducing a quiver covering 
    \[
         p'\colon \fzmu_{[k]}(\widetilde Q) \rightarrow \fzmu_{[k]}(\widetilde Q)/\Gamma.
    \]
    We call the covering $p'$ the orbit mutation of $p$ at direction $[k]$ and denote it by $\fzmu_{[k]}(p)$.
\end{proposition}

\begin{proof}
    As discussed earlier $\Gamma$ acts on the orbit pre-mutation. With $2$-cycles chosen in \Cref{lemma: remove 2 cycle orbit} removed, the group $\Gamma$ now acts on the remaining arrows. In fact, what gets removed is (free) $\Gamma$-orbits of arrows. Hence $\Gamma$ also acts on $\fzmu_{[k]}(\widetilde Q)$. It remains a free action as we remove free orbits of arrows.

    The quiver $\fzmu_{[k]}(\widetilde Q)/\Gamma$ is defined to be the \emph{quotient quiver} of the action of $\Gamma$ on $\fzmu_{[k]}(\widetilde Q)$ where both vertex and arrow sets are orbit sets. The group $\Gamma$ again is identified with the deck transformation group $\mathrm{Deck}(\fzmu_{[k]}(p))$.
\end{proof}

\begin{remark}\label{rmk: covering after orbit mutation}
    In general, the covering $\fzmu_{[k]}(p)$ may no longer be weakly admissible (\Cref{def: weakly admissible}) even though $p$ is assumed to be so. Namely, the quotient quiver $\fzmu_{[k]}(\widetilde Q)/\Gamma$ may have loops.
\end{remark}

\begin{example}\label{ex: covering 4 1}
    Consider covering $p\colon \widetilde Q \rightarrow Q$ in \Cref{fig: weak admin cover}. The orbit mutations at $[1]$ and $[2]$ are both simply reversing all arrows in $\widetilde Q$. In either orbit mutation, any $\Gamma$-loop is deleted due to deletion of $2$-cycles. So both orbit mutations are still weakly admissible.
    \begin{figure}[ht]
        \centering
        \begin{tikzpicture}[->, >=stealth', scale = 1.5]
            \node[] (1) at (-2, 0) {1};
            \node[] (2) at (-1, 1) {2};
            \node[] (3) at (-1, -1) {2};
            \node[] (4) at (0, 0) {1};
            \node[] (5) at (1, 1) {2};
            \node[] (6) at (1, -1) {2};
            \node[] (7) at (2, 0) {1};
            \node[] (8) at (-2.3, 0) {$\cdots$};
            \node[] (9) at (2.3, 0) {$\cdots$};
            \node[] (10) at (2.5, 0) {};
            \node[] (11) at (3.5, 0) {};
            \node[] (12) at (4, 0) {1};
            \node[] (13) at (6, 0) {2.};
            \node[] (15) at (-1, 1.4) {$\vdots$};
            \node[] (16) at (-1, -1.3) {$\vdots$};
            \node[] (17) at (1, 1.4) {$\vdots$};
            \node[] (18) at (1, -1.3) {$\vdots$};

            \path[]
            (1) edge[bend left = 25] node[fill = white, pos = 0.5] {$\alpha_1$} (2)
            (1) edge[bend right = 25] node[fill = white, pos = 0.5] {$\alpha_2$} (2)
            (2) edge[bend left = 20] node[fill = white, pos = 0.5] {$\gamma_1$} (4)
            (3) edge[bend left = 20] node[fill = white, pos = 0.5] {$\gamma_2$} (1)
            (4) edge[bend left = 25] node[fill = white, pos = 0.5] {$\beta_1$} (3)
            (4) edge[bend right = 25] node[fill = white, pos = 0.5] {$\beta_2$} (3)
            (4) edge[bend left = 25] node[fill = white, pos = 0.5] {$\alpha_1$} (5)
            (4) edge[bend right = 25] node[fill = white, pos = 0.5] {$\alpha_2$} (5)
            (5) edge[bend left = 20] node[fill = white, pos = 0.5] {$\gamma_1$} (7)
            (6) edge[bend left = 20] node[fill = white, pos = 0.5] {$\gamma_2$} (4)
            (7) edge[bend left = 25] node[fill = white, pos = 0.5] {$\beta_1$} (6)
            (7) edge[bend right = 25] node[fill = white, pos = 0.5] {$\beta_2$} (6)
            (10) edge node[above, pos = 0.5] {$p$} (11)
            (12) edge[bend left = 37.5] node[fill = white, pos = 0.5] {$\alpha_1$} (13)
            (12) edge[bend left = 80] node[fill = white, pos = 0.5] {$\alpha_2$} (13)
            (13) edge[bend left = 10] node[fill = white, pos = 0.5] {$\gamma_1$} (12)
            (13) edge[bend right = 10] node[fill = white, pos = 0.5] {$\gamma_2$} (12)
            (12) edge[bend right = 37.5] node[fill = white, pos = 0.5] {$\beta_1$} (13)
            (12) edge[bend right = 80] node[fill = white, pos = 0.5] {$\beta_2$} (13)
            ;
        \end{tikzpicture}
    \caption{A globally weakly admissible covering.}
        \label{fig: weak admin cover}
    \end{figure}
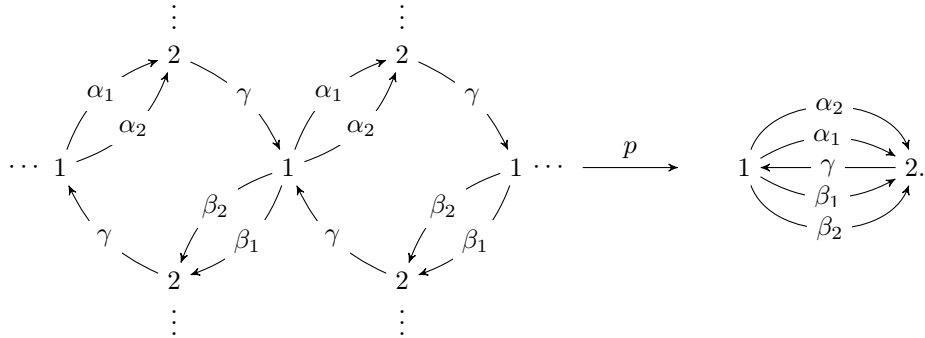
\end{example}

\begin{definition}\label{def: k mutable}
    We say that a weakly admissible covering $p \colon \widetilde Q \rightarrow Q$ is \emph{$[k]$-mutable} or \emph{mutable at} $k$ for $k\in Q_0$ if the covering (\Cref{prop: mutation of a covering})
    \[
         \fzmu_{[k]}(p) \colon \fzmu_{[k]}(\widetilde Q) \rightarrow \fzmu_{[k]}(\widetilde Q)/\Gamma
    \]
     is again weakly admissible. If the covering stays weakly admissible after any sequence of orbit mutations at any direction, we say that $p$ is \emph{global} or \emph{globally weakly admissible}.
\end{definition}

The covering in \Cref{ex: covering 4 1} is mutable at $[1]$ and $[2]$. Since both mutations are simply reversing all arrows, the covering $p$ is global.

\begin{example}
    The following covering is not mutable at either $[1]$ or $[2]$.
    
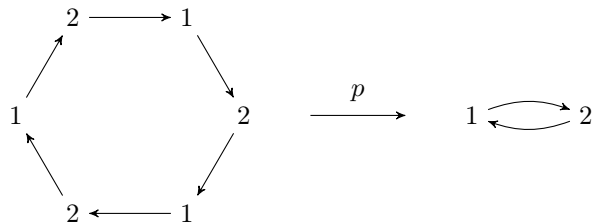
\[
        \begin{tikzpicture}[->, >=stealth', scale = 0.75]
            \node[] (1) at (-2, 0) {1};
            \node[] (2) at (-1, {sqrt(3)}) {2};
            \node[] (3) at (1, {sqrt(3)}) {1};
            \node[] (4) at (2, 0) {2};
            \node[] (5) at (1, {-sqrt(3)}) {1};
            \node[] (6) at (-1, {-sqrt(3)}) {2};
            \node[] (t) at (3, 0) {};
            \node[] (h) at (5, 0) {};
            \node[] (p) at (4, 0.4) {$p$};
            \node[] (7) at (6, 0) {1};
            \node[] (8) at (8, 0) {2};

            \path[]
            (1) edge (2)
            (2) edge (3)
            (3) edge (4)
            (4) edge (5)
            (5) edge (6)
            (6) edge (1)
            (t) edge (h)
            (7) edge[bend left = 20] (8)
            (8) edge[bend left = 20] (7)
            ;
        \end{tikzpicture}
    \]
    For example, the covering $\fzmu_{[1]}(p)$ is
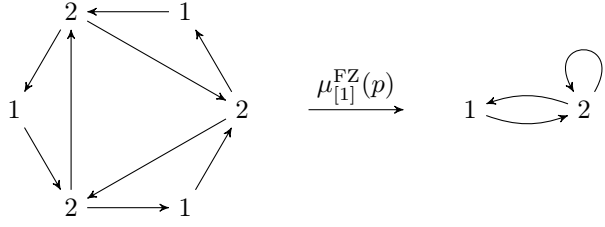
    \[
        \begin{tikzpicture}[->, >=stealth', scale = 0.75]
            \node[] (1) at (-2, 0) {1};
            \node[] (2) at (-1, {sqrt(3)}) {2};
            \node[] (3) at (1, {sqrt(3)}) {1};
            \node[] (4) at (2, 0) {2};
            \node[] (5) at (1, {-sqrt(3)}) {1};
            \node[] (6) at (-1, {-sqrt(3)}) {2};
            \node[] (t) at (3, 0) {};
            \node[] (h) at (5, 0) {};
            \node[] (p) at (4, 0.4) {$\fzmu_{[1]}(p)$};
            \node[] (7) at (6, 0) {1};
            \node[] (8) at (8, 0) {2};

            \path[]
            (1) edge (6)
            (2) edge (1)
            (3) edge (2)
            (4) edge (3)
            (5) edge (4)
            (6) edge (5)
            (t) edge (h)
            (7) edge[bend right = 20] (8)
            (8) edge[bend right = 20] (7)
            (8) edge[out = 60, in = 120, looseness = 8] (8)
            (6) edge (2)
            (2) edge (4)
            (4) edge (6)
            ;
        \end{tikzpicture}
    \]
\end{example}

We give below a sufficient condition for a covering to be $[k]$-mutable.

\begin{lemma}\label{lemma: sufficient k-mutable}
    Let $p\colon \widetilde Q \rightarrow Q$ be a weakly admissible quiver covering. Let $k\in Q_0$. Suppose that for any $2$-cycle
    
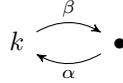
\[
        \begin{tikzcd}
            k \ar[r, bend left, "\beta"] & \bullet \ar[l, bend left, "\alpha"]
        \end{tikzcd}
    \]
    in $Q$, its square $(\beta\alpha)^2$ is in $p_*(\pi_1(\widetilde Q))$. Then $p$ is $[k]$-mutable.
\end{lemma}

\begin{proof}
    Let $\widetilde k\in p^{-1}(k)$ and consider any path
    \[
        \begin{tikzcd}
            i \ar[r, "\widetilde \alpha"] & \widetilde k \ar[r, "\widetilde \beta"] & j
        \end{tikzcd}
    \]
    such that $p(i) = p(j)\in Q_0$. Then there is a deck transformation $\tau\in \mathrm{Deck}(p)$ corresponding to $\widetilde \beta \widetilde \alpha$. By the assumption $(p(\widetilde \beta)p(\widetilde \alpha))^2 \in p_*(\pi_1(\widetilde Q))$, we have that $\tau$ is an involution swapping $i$ and $j$. Applying $\tau$ to the above path, we obtain the square
    
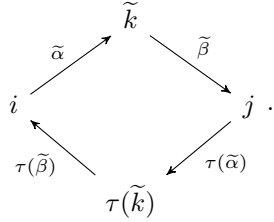
\[
        \begin{tikzcd}
            & \widetilde k \ar[rd, "\widetilde \beta"] & \\
            i \ar[ru, "\widetilde \alpha"] & & j \ar[dl, "\tau(\widetilde \alpha)"] \\
            & \tau(\widetilde k) \ar[lu, "\tau(\widetilde \beta)"] &
        \end{tikzcd}.
    \]
    Then any path $\widetilde \beta \widetilde \alpha$ from $i$ to $j$ through $\widetilde k$ is paired with the path $\tau(\widetilde \beta)\tau( \widetilde \alpha)$. Therefore any arrow created between $i$ and $j$ will be deleted in the orbit mutation $\fzmu_{[k]}$, meaning that $p$ is $[k]$-mutable.
\end{proof}

The covering in \Cref{ex: covering 4 1}, while globally weakly admissible, does not satisfy the condition in \Cref{lemma: sufficient k-mutable}.

\section{Quivers with homotopies}\label{section: quivers with homotopies}

In \Cref{subsec: quivers with homotopies} we introduce \emph{quivers with homotopies}, the central objects of study in this paper. Then we develop the theory of their \emph{mutations} in \Cref{subsec: mutations with homotopy,subsec: involution,subsec: patterns of QH,subsec: homotopy and orbit mutation}.

\subsection{Quivers with homotopies}\label{subsec: quivers with homotopies}

For any arrow $a\in Q_1$, denote by $a^{-1}$ the \emph{formal inverse} of $a$, that is, an opposite arrow with $s(a^{-1}) = t(a)$ and $t(a^{-1}) = s(a)$. Denote by $\overline Q$ the \emph{double quiver} of $Q$ where
\[
    \overline Q_1 = \{a^{-1}\mid a\in Q_1\}\cup Q_1, \quad \overline Q_0 = Q_0.
\]

\begin{definition}\label{def: cyclic walk}
    Let $Q$ be a quiver.
    \begin{enumerate}
        \item A \emph{path} $p$ in $Q$ is either \emph{non-trivial}, being a word $a_{n}a_{n-1}\cdots a_1$ of arrows $a_n,\dots, a_1$ such that $t(a_i) = s(a_{i+1})$ for $i = 1, \dots, n-1$ or a \emph{trivial path} $e_v$ for $v\in Q_0$. We define $s(p)=s(a_1)$, $t(p)=t(a_n)$ or $s(p)=t(p)=v$ respectively. When $s(p)=t(p)$, a non-trivial path $p$ is called an \emph{$n$-cycle}.

        \item When $s(p)=t(q)$ for paths $p$ and $q$, their \emph{composition} is (1) $pq$ if $p$ and $q$ are non-trivial; (2) $p$ if $q$ is trivial; (3) $q$ if $p$ is trivial.

        \item A \emph{walk} in $Q$ is a path in $\overline {Q}$. The \emph{reduction} of a walk is obtained by removing any sub-word $aa^{-1}$ or $a^{-1}a$ (the order of removal is irrelevant). A walk is called \emph{reduced} if it is identical to its reduction.

        \item A walk in $Q$ is called \emph{cyclic} if it is a cycle in $\overline{Q}$.
    \end{enumerate}
\end{definition}

A \emph{groupoid} is a category in which every morphism is an isomorphism. For a groupoid $G$, denote by $G(x, y)$ the set of morphisms from object $y$ to object $x$. We take the convention that morphisms compose from right to left, to align with the order of path concatenation of a quiver.

\begin{definition}\label{def: free groupoid of quiver}
    We define $\grp(Q)$ to be the \emph{free groupoid} generated by a quiver $Q$, that is, it has objects $Q_0$ and morphisms as reduced walks in $Q$ whose composition is the composition of paths in $\overline Q$ followed by a reduction. The identity of each object $v\in Q_0$ is the trivial path $e_v$.
\end{definition}

If $Q_0$ has only one element, then $\pi(Q)$ coincides with the free group generated by $Q_1$. 

It is worth noting that $\grp(Q)$ is isomorphic to the \emph{fundamental groupoid} $\pi_1(Q, Q_0)$ (viewing $Q$ as a 1-complex and $Q_0$ as a set of base points) by regarding any arrow $a$ as the homotopy class of paths traveling from $s(a)$ to $t(a)$ along $a$.

\begin{definition}[\cite{Brown}]
    A \emph{normal subgroupoid} (or a \emph{homotopy}) $H$ of a groupoid $G$ is a subcategory that contains all objects of $G$, and that itself is a groupoid such that 
    \begin{enumerate}
        \item $H(x, y) = \varnothing$ unless $x = y$;
        \item $g^{-1}H(x, x)g = H(y, y)$ for any $g\in G(x, y)$.
    \end{enumerate}
\end{definition}

In other words, a normal subgroupoid $H$ is simply a collection of normal subgroups $H(x, x)\subseteq G(x, x)$ for each object $x$ of $G$ that are conjugated by elements in $G(x, y)$. Thus if $G$ is connected, then any $H(x, x)$ will determine others through conjugation.

Let $S$ be a set of endomorphisms in $G$. Denote by $\ncl{S}$ the normal subgroupoid of $G$ generated by $S$, i.e., the smallest subgroupoid containing $S$.

For a morphism $f$ in $G$, by slight abuse of notation, we say that \emph{$f$ is in $H$} or write $f\in H$ if $f\in H(x,x)$ for some object $x$.

\begin{definition}
    The \emph{quotient groupoid} $G/H$ of a groupoid $G$ by a homotopy $H$ is the groupoid that has the same objects $\mathrm{Obj}(G)$ as $G$ does, and that $(G/H)(x,x) = G(x,x)/H(x,x)$ and $(G/H)(x, y)$ is the set of equivalent classes in $G(x, y)$ where $f$ is equivalent to $g$ (denoted as $f \sim_H g$) if $f^{-1}g\in H$.
\end{definition}

It is natural to consider homotopies on the groupoid $\pi(Q)$ of a quiver $Q$, see for example \cite[Section 2]{CP}. We present the following slightly more restrictive definition for the purpose of this paper.

\begin{definition}\label{def: quiver with homotopy}
    A \emph{quiver with homotopy} $(Q, H)$ is a loop-free quiver $Q$ with a normal subgroupoid (or a homotopy) $H$ of $\grp(Q)$. The homotopy $H$ is called \emph{reduced} if it does not contain any $2$-cycles in $Q$. The \emph{trivial homotopy} is denoted by $\trht$ where $\trht(v,v)=\{e_v\}$ for any $v\in Q_0$.
\end{definition}

Let $W$ be a finite set of reduced cyclic walks in $Q$ and $H = \ncl{W}$. Then the quotient groupoid $\pi(Q)/H$ has the following topological realization. Regard $Q$ as a $1$-dimensional CW-complex $|Q|$. For each cyclic walk $w = a_m\dots a_1\in W$, take an oriented $m$-gon $P_w$ whose boundary $\partial P_w$ consists of $m$ consecutive oriented edges $e_1,\dots e_m$. Define an attaching map 
\[
    \varphi_w\colon \partial P_w \to |Q|
\]
by sending $e_r$ homeomorphically onto the edge of $|Q|$ underlying $a_r$, with the orientation prescribed by $a_r$ if $a_r\in Q_1$, and with the opposite orientation if $a_r = b^{-1}$ for some $b\in Q_1$. We then define a $2$-dimensional CW-complex
\[
    \Delta(Q, H, W) \coloneqq |Q| \cup_{\{\varphi_w\}_{w\in W}} \bigcup_{w\in W} P_w.
\]
Note that the maps $\varphi_w$ are not required to be embeddings since $w$ may pass through the same vertex or edge more than once. Different $\varphi_w$'s may have overlapping images in $|Q|$. The construction simply attaches one $2$-cell to $|Q|$ for each cyclic walk in $W$.

\begin{lemma}\label{lemma: 2-complex quiver homotopy}
    The quotient groupoid $\pi(Q)/H$ is isomorphic to $\pi_1(\Delta(Q, H, W), Q_0)$.    
\end{lemma}

\begin{proof}
    We can work with each connected component of $Q$ and thus assume that $Q$ is connected. We prove by induction on $|W|$. The statement trivially holds when $W = \varnothing$. Then the induction is completed by showing the isomorphism when adding one more cyclic walk $w'$ to $W$:
    \[
        \pi_1(\Delta(Q, \ncl{W}, W), Q_0)/\ncl{w'} \cong \pi_1(\Delta(Q, \ncl{W\cup\{w'\}}, W\cup\{w'\}), Q_0).
    \]
    This follows from van Kampen's theorem; see also \cite{Brown} for a version in groupoids.
\end{proof}

\begin{example}\label{ex: 2 cycle}
    Let $Q$ be the quiver $\begin{tikzcd}
        x \ar[r, shift left, "a"] & y \ar[l, shift left, "b"]
    \end{tikzcd}$. Any non-trivial homotopy is of the form $H_k = \ncl{(ab)^k}$ for some $k\geq 1$. The $2$-complex $\Delta(Q, H_k, \{(ab)^k\})$ is homeomorphic to a disk when $k=1$ and to the real projective plane when $k=2$.
\end{example}

\subsection{Mutations of quivers with homotopies}\label{subsec: mutations with homotopy}

In this subsection we define \emph{mutations} of quivers with homotopies and prove some of their basic properties.

\begin{definition}[Pre-mutation]\label{def: pre-mutation}
    Let $(Q, H)$ be a quiver with homotopy where $H$ is reduced. The \emph{pre-mutation} $\tilde{\mu}_k(Q,H)$ of $(Q, H)$ in direction $k\in Q_0$ is defined to be the quiver with homotopy $(Q', H')$ where $Q' = \prefzmu_k(Q)$ and $H'$ is constructed via the following process. Recall notations from \Cref{def: fz mutation,def: fz premutation}.
    \begin{enumerate}

        \item Let $\varphi$ be the functor from $\grp(Q')$ to $\grp(Q)/H$ given by identity on vertices and
        \[
            \alpha^\star \mapsto \alpha^{-1},\quad \beta^\star\mapsto \beta^{-1},\quad [\beta\alpha] \mapsto \beta\alpha \quad \text{(and $\gamma \mapsto \gamma$ for all rest arrows in $Q'$)}.
        \]
        \item Define $H' \coloneqq \ker \varphi$, that is, for any $i\in Q_0$
        \[
            H'(i, i) = \ker \left(\pi(Q')(i,i) \xrightarrow{\ \varphi\ } \pi(Q)(i,i)/H(i, i) \right).
        \]
        It is a normal subgroupoid of $\grp(Q')$. Note that $\alpha^\star\beta^\star[\beta\alpha]$ is in $H'$ when $t(\alpha)=s(\beta)=k$ and $s(\alpha)\neq t(\beta)$.
    \end{enumerate}
\end{definition}

\begin{definition}[Deletion of $2$-cycles]\label{def: delete 2-cycles}
    For any two distinct vertices $i,j$ (other than $k$) in $Q_0$, consider the sets of arrows between them:
    \begin{itemize}
        \item Let $Q'_1(i, j) = \{\gamma_1, \dots, \gamma_N\}$ be the set of arrows from $j$ to $i$.
        \item Let $Q'_1(j, i) = \{\delta_1, \dots, \delta_M\}$ be the set of arrows from $i$ to $j$.
    \end{itemize}
    The arrows in these sets are listed in an arbitrary order. We now modify the quiver $Q'$ by iteratively removing certain pairs of arrows that form 2-cycles.
    
    \begin{enumerate}
        \item Start with the first arrow $\gamma_1 \in Q'_1(i, j)$. Identify the smallest index $r$ such that the composition $\gamma_1 \delta_r$ is in $H'$. If such an index exists, delete both $\gamma_1$ and $\delta_r$ from their respective sets.
        \item Proceed to the next remaining arrow $\gamma_p$ in $Q'_1(i, j)$ and repeat the process: find the smallest index $q$ such that $\gamma_p \delta_q \in H'$ and remove both $\gamma_p$ and $\delta_q$.
        \item Continue this process until all arrows in $Q'_1(i, j)$ have been considered.
    \end{enumerate}
    
    After completing this procedure for all pairs of vertices $(i, j)$, we obtain a new quiver $Q^\dagger$, which has the same vertex set as $Q'$ but with certain 2-cycles removed.
\end{definition}

\begin{definition}[Mutation]\label{def: mutation}
    The \emph{mutation} $\mu_k(Q, H)$ of $(Q, H)$ at direction $k$ is defined to be the quiver with reduced homotopy $(Q^\dagger, H^\dagger)$ where $H^\dagger \coloneqq \ker \psi$ where $\psi\colon \pi(Q^\dag) \rightarrow \pi(Q')/H'$ is the functor induced by the inclusion $Q^\dag \subseteq Q'$.
\end{definition}

Several fundamental properties of mutations are in order.

\begin{lemma}\label{lemma: surjective grp map}
    The functor $\psi\colon \grp(Q^\dagger) \rightarrow \grp(Q')/H'$ is full, i.e., the maps on morphisms are surjective.
\end{lemma}

\begin{proof}
    We show that every morphism in $\grp(Q')/H'$ has a lift in $\grp(Q^\dagger)$. Let $u$ be a reduced walk in $Q'$. It suffices to replace each arrow of $u$ deleted from $Q'$ by an equivalent walk modulo $H'$ in $Q^\dagger$.

    Suppose that an arrow $\gamma$ is deleted together with $\delta$, so that $\gamma\delta\in H'$. Then $\gamma = \delta^{-1}$ in $\grp(Q')/H'$. Since $H$ is reduced, at least one of $\gamma,\delta$ is newly created of the form $[\beta\alpha]$. If, for instance, $\delta = [\beta\alpha]$, then
    \[
        \alpha^\star\beta^\star[\beta\alpha]\in H',
    \]
    and hence
    \[
        \gamma=\delta^{-1}=[\beta\alpha]^{-1} =\alpha^\star\beta^\star \quad \text{in } \pi(Q')/H'.
    \]
    The right-hand side $\alpha^\star\beta^\star$ is a walk in $Q^\dagger$. The case $\gamma=[\beta\alpha]$ is the same.

    Thus every deleted arrow is equivalent modulo $H'$ to a walk in $Q^\dagger$. Replacing all deleted arrows in $u$ gives a walk in $Q^\dagger$ representing the same class in $\grp(Q')/H'$. Therefore the functor $\psi$ is full.
\end{proof}

The next lemma shows that the mutation $\mu_k(Q, H)$ is independent of the choice of deletion in \Cref{def: delete 2-cycles} up to quiver isomorphisms.

\begin{lemma}\label{lemma: unique reduction up to iso}
    Let $Q^{\dagger, 1}$ and $Q^{\dagger, 2}$ be two quivers obtained by deleting 2-cycles (with possibly different choices of deletion) from the pre-mutation $Q'$ as in \Cref{def: delete 2-cycles}. Then there exists an isomorphism $\eta \colon Q^{\dagger, 1} \rightarrow Q^{\dagger, 2}$ (fixing the vertex set) such that the diagram of groupoids
    \[
        \begin{tikzcd}[column sep = tiny]
            \grp(Q^{\dagger,1}) \ar[rr, "\eta_*"] \ar[dr, swap, "\psi_1"] & & \grp(Q^{\dagger, 2}) \ar[dl, "\psi_2"] \\
            & \grp(Q')/H' &
        \end{tikzcd}
    \]
    commutes, where $\psi_1$ and $\psi_2$ denote the surjective maps in \Cref{lemma: surjective grp map}.
\end{lemma}

\begin{proof}
    Fix two different vertices $i$ and $j$ that are not $k$. There are disjoint non-empty subsets $G_1, \dots, G_L$ of $Q'_1(i,j)$ and $D_1, \dots, D_L$ of $Q'_1(j,i)$ such that for any $p = 1, \dots, L$, any $\gamma\in G_p$ and any $\delta\in D_p$, we have $\gamma\delta\in H'$, and $\gamma\delta'\notin H'$ for any $\delta'\in Q'_1(j, i)\setminus D_p$, and $\gamma'\delta\notin H'$ for any $\gamma'\in Q'_1(i, j)\setminus G_p$.
    
    The deletion process described in \Cref{def: delete 2-cycles} is simply removing a same number of arrows in $G_p$ and $D_p$ such that there is no remaining arrow in one of them. We denote the sets of remaining arrows by $G_p^v$ and $D_p^v$ (one of them is empty) for $Q^{\dagger, v}$ with $v = 1, 2$. Of course $|G_p^1| = |G_p^2|$ and $|D_p^1| = |D_p^2|$ for any $p$. The restriction of $\eta$ on $G_p^1$ to $G_p^2$ (and on $D_p^1$ to $D_p^2$) is an arbitrarily chosen bijection.

    Let $\eta$ be identity on any arrow that is not in any $G_p$ or $D_p$ for any $i$ and $j$. Now by construction it is clear that $\eta$ is an isomorphism on quivers.

    Notice that $\psi_v$ in the above diagram denotes the functor in \Cref{lemma: surjective grp map} induced by the natural inclusion $\iota_v \colon Q^{\dagger, v}_1 \hookrightarrow Q'_1$ for $v = 1, 2$. To see the diagram commutes, we only need to show $\psi_1(\gamma_1) = \psi_2(\gamma_2)$ if $\eta(\gamma_1) = \gamma_2$ for $\gamma_1\in G_p^1$ and $\gamma_2\in G_p^2$ (and the same for $D_p^1$ and $D_p^2$). In fact, $\gamma_1$ and $\gamma_2$ are in the same $G_p\subseteq Q'_1$. Thus there exists some $\delta$ such that both $\gamma_1\delta$ and $\gamma_2\delta$ are in $H'$, which means $\gamma_1$ and $\gamma_2$ are equivalent with respect to $H'$ (i.e. $\gamma_1\gamma_2^{-1}\in H'$). Thus $\psi_1(\gamma_1) = \psi_2(\gamma_2)$ in $\grp(Q')/H'$.
\end{proof}

\begin{remark}\label{rmk: quotient groupoid invariant}
    Tracking through the construction of $\mu_k$, we have functors $\varphi$ and $\psi$
    \[
        \begin{tikzcd}
            \grp(Q^\dagger)/H^\dagger \ar[r, "\psi"] & \grp(Q')/H' \ar[r, "\varphi"] & \grp(Q)/H.
        \end{tikzcd}
    \]
    It follows from the construction and \Cref{lemma: surjective grp map} that both $\varphi$ and $\psi$ are isomorphisms. Therefore the quotient groupoid is preserved along mutations.
    By \Cref{lemma: unique reduction up to iso}, the quiver with homotopy $\mu_k(Q, H)$ is unique up to an isomorphism $\eta$ on quivers that translates homotopies, i.e., $\eta_*(H^{\dagger, 1}) = H^{\dagger, 2}$ where $H^{\dagger, v}$ denotes $\ker \psi_v$ for $v = 1, 2$ as in \Cref{lemma: unique reduction up to iso}.

    Throughout, equalities of quivers are understood up to isomorphisms fixing the vertex labels; for quivers with homotopies, these isomorphisms must also carry one homotopy onto the other. This convention applies in particular to \Cref{prop: mutation involutive,prop: maximal H induces FZ mutation,thm: mutation and flip}.
\end{remark}

\begin{example}\label{ex: first ex of mutation homotopy}
Consider the quiver $Q$ on the left.
    
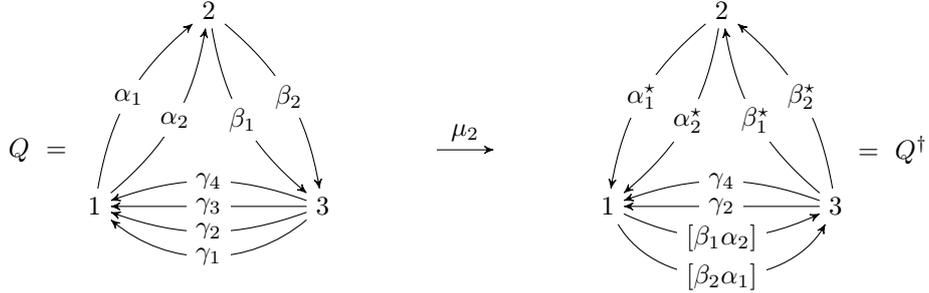
\[
        \begin{tikzpicture}[->, >=stealth', scale = 0.75]
        \begin{scope}[shift={(0,0)}]
            \node[] (Q) at (-3, 1) {$Q\ = $};
            \node[] (1) at (-2, 0) {1};
            \node[] (2) at (0, {sqrt(12)}) {2};
            \node[] (3) at (2, 0) {3};

            \path[]
            (1) edge[bend right = 20] node [fill = white, pos = 0.5] {$\alpha_2$} (2)
            (1) edge[bend left = 20] node [fill = white, pos = 0.5] {$\alpha_1$} (2)
            (2) edge[bend right = 20] node [fill = white, pos = 0.5] {$\beta_1$} (3)
            (2) edge[bend left = 20] node [fill = white, pos = 0.5] {$\beta_2$} (3)
            (3) edge[bend right = 20] node [fill = white, pos = 0.5] {$\gamma_4$} (1)
            (3) edge[] node [fill = white, pos = 0.5] {$\gamma_3$} (1)
            (3) edge[bend left = 20] node [fill = white, pos = 0.5] {$\gamma_2$} (1)
            (3) edge[bend left = 40] node [fill = white, pos = 0.5] {$\gamma_1$} (1);
        \end{scope}

        \draw[->] (4,1) to node[above]{$\mu_2$} (5,1); 

        \begin{scope}[shift={(9,0)}]
            \node[] (Qdag) at (3, 1) {$=\ Q^\dag$};
            \node[] (1) at (-2, 0) {1};
            \node[] (2) at (0, {sqrt(12)}) {2};
            \node[] (3) at (2, 0) {3};

            \path[]
            (2) edge[bend right = 20] node [fill = white, pos = 0.5] {$\alpha_1^\star$} (1)
            (2) edge[bend left = 20] node [fill = white, pos = 0.5] {$\alpha_2^\star$} (1)
            (3) edge[bend right = 20] node [fill = white, pos = 0.5] {$\beta_2^\star$} (2)
            (3) edge[bend left = 20] node [fill = white, pos = 0.5] {$\beta_1^\star$} (2)
            (3) edge[bend right = 20] node [fill = white, pos = 0.5] {$\gamma_4$} (1)
            (3) edge[] node [fill = white, pos = 0.5] {$\gamma_2$} (1)
            (1) edge[bend right = 25] node [fill = white, pos = 0.5] {$[\beta_1\alpha_2]$} (3)
            (1) edge[bend right = 60] node [fill = white, pos = 0.5] {$[\beta_2\alpha_1]$} (3);
        \end{scope}
        \end{tikzpicture} 
    \]
    Let $H = \ncl{\gamma_1 \beta_1 \alpha_1, \gamma_2 \beta_1 \alpha_1, \gamma_3 \beta_2 \alpha_2}$. After the pre-mutation $\tilde \mu_2$, we have
    \[
        Q'_1(3,1) = \{[\beta_1\alpha_1],[\beta_2\alpha_2],[\beta_1\alpha_2],[\beta_2\alpha_1]\} \quad \text{and} \quad Q'_1(1, 3) = Q_1(1, 3).
    \]
    The $2$-cycles $\gamma_1[\beta_1\alpha_1], \gamma_2[\beta_1\alpha_1], \gamma_3[\beta_2\alpha_2]$ are in $H'$ because their images are in $H$. We can first delete the pair $(\gamma_1, [\beta_1\alpha_1])$. With $[\beta_1\alpha_1]$ already removed, the edge $\gamma_2$ remains. Then $(\gamma_3, [\beta_2\alpha_2])$ gets deleted. Upon this choice, we have $Q^\dag_1(3, 1) = \{[\beta_1\alpha_2], [\beta_2\alpha_1]\}$ and $Q^\dag_1(1, 3) = \{\gamma_2, \gamma_4\}$. The homotopy $H^\dag$ is generated by
    \[
        \{\alpha_2^\star\beta_1^\star[\beta_1\alpha_2], \alpha_1^\star\beta_2^\star[\beta_2\alpha_1], \gamma_2(\beta_1^\star)^{-1}(\alpha_1^\star)^{-1}\}.
    \]
\end{example}

\begin{example}\label{ex: one mutation of Markov}
    Consider the Markov quiver $Q$ below.
    
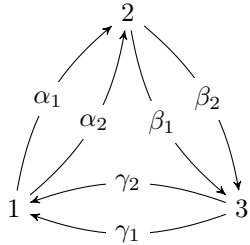
\[
        \begin{tikzpicture}[->, >=stealth', scale = 0.75]
            \node[] (1) at (-2, 0) {1};
            \node[] (2) at (0, {sqrt(12)}) {2};
            \node[] (3) at (2, 0) {3};

            \path[]
            (1) edge[bend right = 20] node [fill = white, pos = 0.5] {$\alpha_2$} (2)
            (1) edge[bend left = 20] node [fill = white, pos = 0.5] {$\alpha_1$} (2)
            (2) edge[bend right = 20] node [fill = white, pos = 0.5] {$\beta_1$} (3)
            (2) edge[bend left = 20] node [fill = white, pos = 0.5] {$\beta_2$} (3)
            (3) edge[bend left = 20] node [fill = white, pos = 0.5] {$\gamma_1$} (1)
            (3) edge[bend right = 20] node [fill = white, pos = 0.5] {$\gamma_2$} (1);
        \end{tikzpicture}
    \]
    We consider the following homotopies.
    \begin{enumerate}
        \item Let $H_1$ be any homotopy that does not contain a $3$-cycle, for example, the trivial homotopy. Clearly no $2$-cycle will be deleted in any of the three mutations $\mu_1$, $\mu_2$ and $\mu_3$.
        
        \item Let $H_2 = \ncl{\gamma_1\beta_1\alpha_1, \gamma_2\beta_1\alpha_1}$. Exactly one $2$-cycle is deleted in any of the three mutations. In $\mu_2$, the deletion is similar to the deletion of the pair $(\gamma_1, [\beta_1\alpha_1])$ in the previous \Cref{ex: first ex of mutation homotopy}.

        \item Let $H_3 = \ncl{\gamma_1\beta_1\alpha_1, \gamma_2\beta_1\alpha_2}$. Two $2$-cycles $(\gamma_1, [\beta_1\alpha_1])$ and $(\gamma_2, [\beta_1\alpha_2])$ are deleted in $\mu_2$. Similarly two $2$-cycles get deleted in $\mu_3$. One $2$-cycle $(\beta_1, [\alpha_1\gamma_1])$ or $(\beta_1, [\alpha_2\gamma_2])$ is deleted in $\mu_1$.
        
        \item Let $H_4 = \ncl{\gamma_1\beta_1\alpha_1, \gamma_2\beta_2\alpha_2}$. Two $2$-cycles are deleted in any of the three mutations.
    \end{enumerate}
\end{example}

\begin{example}\label{example of mutation 1}
    Let $(Q, H)$ be
    \begin{equation*}
        Q = \xymatrix@C=1.5cm{1\ar@/^0.2pc/[r]^{a} & 2\ar@/^0.2pc/[l]^{d}\ar@/^0.2pc/[r]^{b} & 3\ar@/^0.2pc/[l]^{c}}, \quad H = \ncl{dcba}.
    \end{equation*}
    The mutation at $1$ is
    \[
        \mu_1(Q, H) = \left( \xymatrix@C=1.5cm{1\ar@/^0.2pc/[r]^{d^*} & 2\ar@/^0.2pc/[l]^{a^*}\ar@/^0.2pc/[r]^{b} & 3\ar@/^0.2pc/[l]^{c}}, \quad \ncl{ (d^*)^{-1}cb(a^*)^{-1}} \right).
    \]
    The pre-mutation at $2$ is $\tilde \mu_2(Q, H) = (Q', H')$ with
    \begin{equation*}
        Q' = \xymatrix@C=1.5cm{1 \ar@/^0.2pc/[r]^{d^*} \ar@/^1.5pc/[rr]^{[ba]} & 2 \ar@/^0.2pc/[l]^{a^*} \ar@/^0.2pc/[r]^{c^*} & 3 \ar@/^1.5pc/[ll]^{[dc]} \ar@/^0.2pc/[l]^{b^*}
        }, \quad H' = \ncl{[ba]a^*b^*, [dc]c^*d^*, [dc][ba]}.
    \end{equation*}
   After deleting 2-cycles, we have $\mu_2(Q, H) = (Q^\dagger, H^\dagger)$ with
    \begin{equation*}
        Q^\dagger = \xymatrix@C=1.5cm{1 \ar@/^0.2pc/[r]^{d^*} & 2 \ar@/^0.2pc/[l]^{a^*} \ar@/^0.2pc/[r]^{c^*} & 3 \ar@/^0.2pc/[l]^{b^*}
        }, \quad H^\dagger = \ncl{a^*b^*c^*d^*}.
    \end{equation*}
    By the obvious isomorphism between $(Q^\dagger, H^\dagger)$ and $(Q, H)$, we see $\mu_2^2(Q, H) = (Q, H)$.
\end{example}

\begin{example}\label{example of mutation 2}
    For the same quiver $Q$ as in \Cref{example of mutation 1}, consider instead the trivial homotopy $\trht$. In this case the $2$-cycle formed by $[ba]$ and $[dc]$ in the pre-mutation $\tilde \mu_2(Q, \trht)$ will not be deleted. Then $\mu_2(Q, \trht)$ is
    \begin{equation*}
        Q' = \xymatrix@C=1.5cm{1 \ar@/^0.2pc/[r]^{d^*} \ar@/^1.5pc/[rr]^{[ba]} & 2 \ar@/^0.2pc/[l]^{a^*} \ar@/^0.2pc/[r]^{c^*} & 3 \ar@/^1.5pc/[ll]^{[dc]} \ar@/^0.2pc/[l]^{b^*}
        }, \quad H'' = \ncl{[ba]a^*b^*, [dc]c^*d^*}.
    \end{equation*}
    It is not hard to verify that $\mu_2(Q', H'') = (Q, \trht)$. Then by the rotational symmetry of $(Q', H'')$, we see that $\mu_3(Q', H'')$ and $\mu_1(Q', H'')$ are both isomorphic to $(Q, \trht)$. Therefore we have understood quivers with homotopies obtained from any mutation sequences from $(Q, \trht)$.
\end{example}

\subsection{Mutation is an involution}\label{subsec: involution}

The FZ-mutation of $2$-acyclic quivers is involutive. This simple combinatorial feature, however, is crucial for the construction of cluster algebras associated to $2$-acyclic quivers. For quivers with homotopies, we show that mutation is still involutive.

\begin{proposition}\label{prop: mutation involutive}
    Let $(Q, H)$ be a quiver with reduced homotopy. Denote $(Q^\heartsuit, H^\heartsuit) = \mu_k^2(Q, H)$ for $k\in Q_0$. Then there is a quiver isomorphism $\eta \colon Q^\heartsuit \rightarrow Q$ fixing $Q^\heartsuit_0 = Q_0$ such that $\eta_*(H^\heartsuit)=H$.
\end{proposition}

\begin{proof}
    Let $(Q', H')$ denote the pre-mutation $\tilde \mu_k(Q, H)$. Let $i$ and $j$ be two different vertices other than $k$. We have by definition that
    \[
        Q'_1(j, i) = Q_1(j, i) \sqcup (Q_1(j, k) \times Q_1(k, i)).
    \]
    Fix a (maximal) deletion of 2-cycles (as in \Cref{def: delete 2-cycles}) from $(Q', H')$ to $(Q^\dagger, H^\dagger)$: denote the deleted arrows by the superscript $(\cdot)^{\mathrm{del}}$ and the ones remaining by $(\cdot)^{\mathrm{rem}}$. So we have
    \[
        Q^\dagger_1(j, i) = Q_1(j, i)^{\mathrm{rem}} \sqcup (Q_1(j, k) \times Q_1(k, i))^{\mathrm{rem}}.
    \]
    Consider the pre-mutation $(Q^\spadesuit, H^\spadesuit) \coloneqq \tilde \mu_k(Q^\dagger, H^\dagger)$. We have
    \[
        Q^\spadesuit_1(j, i) = Q_1(j, i)^{\mathrm{rem}} \sqcup (Q_1(j, k) \times Q_1(k, i))^{\mathrm{rem}} \sqcup (Q^\dagger_1(j, k) \times Q^\dagger_1(k, i)),
    \]
    where $Q^\dagger_1(j, k)$ can be identified with $Q_1(k, j)$ and $Q^\dagger_1(k, i)$ with $Q_1(i, k)$ (via $\alpha \leftrightarrow \alpha^*$). So an arrow in the third subset is of the form $[\beta^*\alpha^*]$ for some $\alpha\in Q_1(i, k)$ and $\beta\in Q_1(k, j)$. We denote by $(Q^\dagger_1(j, k) \times Q^\dagger_1(k, i))^{\mathrm{del}}$ the subset of $Q^\dagger_1(j, k) \times Q^\dagger_1(k, i)$ consisting of arrows $[\beta^*\alpha^*]$ such that $[\alpha\beta]$ belongs to $(Q_1(i,k)\times Q_1(k,j))^{\mathrm{del}}$, that is, it gets deleted from $Q'$ to $Q^\dagger$.

    Now we perform deletion of 2-cycles on $(Q^\spadesuit, H^\spadesuit)$ to get $Q^\heartsuit$. For any arrow in $(Q_1(j, k) \times Q_1(k, i))^{\mathrm{rem}}$ (which is of the form $[\alpha\beta]$), there is $[\beta^*\alpha^*] \in Q^\dagger_1(i, k) \times Q^\dagger_1(k, j) \subseteq Q_1^\spadesuit(i, j)$. The functor $\varphi^\spadesuit \colon \grp(Q^\spadesuit) \rightarrow \grp(Q^\dagger)/H^\dagger$ (see \Cref{def: pre-mutation}) sends $[\alpha\beta][\beta^*\alpha^*]$ to $[\alpha\beta]\beta^*\alpha^* \in H^\dagger$. Then the 2-cycle $[\alpha\beta][\beta^*\alpha^*]$ belongs to $H^\spadesuit$, thus can be deleted. The same deletion applies to arrows in $(Q_1(i, k)\times Q_1(k, j))^{\mathrm{rem}} \subseteq Q_1^\spadesuit(i, j)$.

    After removing all the pairs of arrows in the above form, the remaining set of arrows from $i$ to $j$ (resp. from $j$ to $i$) is
    \[
        Q_1(j, i)^{\mathrm{rem}} \sqcup (Q^\dagger_1(j, k) \times Q^\dagger_1(k, i))^{\mathrm{del}} \quad (\text{resp.}\quad Q_1(i, j)^{\mathrm{rem}} \sqcup (Q^\dagger_1(i, k) \times Q^\dagger_1(j, k))^{\mathrm{del}}).
    \]

    In the reduction from $Q'$ to $Q^\dagger$, the set $Q_1(j, i)^{\mathrm{del}}$ must be paired with a subset $S(i, j)$ of $(Q_1(i, k) \times Q_1(k, j))^{\mathrm{del}} \subseteq Q'_1(i, j)$. (The reason that an arrow in $Q_1(j, i)$ is never deleted together with another in $Q_1(i, j)$ is that $H$ is assumed to be reduced.) Then the set of arrows
    \[
        (Q_1(i, k) \times Q_1(k, j))^{\mathrm{del}} \setminus S(i, j) \subseteq Q'_1(i, j)
    \]
    must be deleted together with arrows in $(Q_1(j, k) \times Q_1(k, i))^{\mathrm{del}} \setminus S(j, i)$ in the reduction from $Q'$ to $Q^\dagger$. Now in the current deletion of 2-cycles, we delete any pair
    \[
        [\beta_1^*\alpha_1^*] \in (Q^\dagger_1(j, k) \times Q^\dagger_1(k, i))^{\mathrm{del}} \quad \text{and} \quad  [\beta_2^*\alpha_2^*] \in (Q^\dagger_1(i, k) \times Q^\dagger_1(j, k))^{\mathrm{del}}
    \]
    such that the 2-cycle formed by $[\alpha_1\beta_1]$ and $[\alpha_2\beta_2]$ gets deleted from $Q'$ to $Q^\dagger$. We can do this because 
    \[
        \varphi^\spadesuit([\beta_1^*\alpha_1^*][\beta_2^*\alpha_2^*]) = \beta_1^*\alpha_1^*\beta_2^*\alpha_2^* \in H^\dagger
    \]
    (since the inverse $\alpha_2\beta_2\alpha_1\beta_1$ is in $H$). After deleting those, we are at the arrow sets
    \[
        Q_1^\heartsuit (j, i) = Q_1(j, i)^{\mathrm{rem}} \sqcup \{ [\beta^*\alpha^*] \mid \alpha\beta \in S(i, j)\}
        \quad \text{and} \quad
        Q_1^\heartsuit (i, j) = Q_1(i, j)^{\mathrm{rem}} \sqcup \{ [\beta^*\alpha^*] \mid \alpha\beta \in S(j, i)\}
    \]
    with no further possible deletion.

    At this point we can define a bijection 
    \[
        \eta \colon Q_1^\heartsuit(j, i) \rightarrow Q_1(j, i) = Q_1(j, i)^{\mathrm{rem}} \sqcup Q_1(j,i)^{\mathrm{del}}
    \]
    by gluing the identity on $Q_1(j, i)^{\mathrm{rem}}$ and the bijection from $S(i, j)$ to $Q_1(j, i)^{\mathrm{del}}$ (obtained in the reduction from $Q'$ to $Q^\dagger$). Then $\eta$ extends to a bijection from $Q_1^\heartsuit$ to $Q_1$ in the obvious way, which gives an isomorphism from $Q^\heartsuit$ to $Q$.

    To see $\eta_*(H^\heartsuit) = H$, one can easily check that the diagram
    \[
        \begin{tikzcd}
            \grp(Q^\heartsuit) \ar[r, "\psi^\heartsuit"] \ar[d, "\eta_*"] & \grp(Q^\spadesuit)/H^\spadesuit \ar[r, "\varphi^\spadesuit"] & \grp(Q^\dagger)/H^\dagger \ar[d, "\varphi\circ \psi"] \\
            \grp(Q) \ar[rr, two heads] & & \grp(Q)/H
        \end{tikzcd}
    \]
    commutes. The functor $\varphi\circ \psi$ is an isomorphism by \Cref{rmk: quotient groupoid invariant}. The functors $\psi^\heartsuit$ and $\varphi^\spadesuit$ are defined as $\psi$ and $\varphi$ but for the mutation from $(Q^\dagger, H^\dagger)$ to $(Q^\heartsuit, H^\heartsuit)$, where we know that the former is full and the latter is an isomorphism. Then $\eta_*(H^\heartsuit)=H$, since $H^\heartsuit \coloneqq \ker \psi^\heartsuit$.
\end{proof}

\subsection{Patterns of quivers with homotopies}\label{subsec: patterns of QH}

Let $\mathbb T_n$ be the $n$-regular (infinite) tree where the $n$ edges incident to any vertex are labeled by $\{1, \dots, n\}$. Let $Q$ be a (loop-free) quiver with $Q_0 = \{1, \dots, n\}$ and be assigned to a chosen root vertex $v_0$ of $\mathbb T_n$. Due to the involutivity of mutations (\Cref{prop: mutation involutive}), any homotopy $H$ on $Q$ induces a unique \emph{pattern of quivers with homotopies} on $\mathbb T_n$ that associates any $v\in \mathbb T_n$ a pair $(Q_v, H_v)$ such that
\[
    (Q_{v}, H_{v}) = \mu_k(Q_{v'}, H_{v'})
\]
when $v\frac{k}{\quad\quad} v'$ is a $k$-labeled edge in $\mathbb T_n$.

\begin{definition}
    Two homotopies $H$ and $H'$ on $Q$ are said to be \emph{equivalent} if the quivers of $\mu_{\underline k}(Q,H)$ and $\mu_{\underline k}(Q,H')$ are isomorphic (fixing $Q_0$) for any mutation sequence $\mu_{\underline k} \coloneqq \mu_{k_\ell}\cdots\mu_{k_1}$ with $\underline k = (k_\ell, \dots, k_1)$. In other words, they induce the same assignment of quivers via $v\mapsto Q_v$ on $\mathbb T_n$. In this case, we write $H\sim_Q H'$.
\end{definition}

\begin{example}
    We know from \Cref{ex: 2 cycle} that for $Q$ a 2-cycle, any reduced homotopy is generated by a power (at least 2) of the 2-cycle or is the trivial homotopy. Since the induced mutation on the 2-cycle quiver is simply reversing the orientation, all homotopies are equivalent.    
\end{example}

A quiver $Q$ is called \emph{2-acyclic} if it does not contain any oriented 2-cycles. It is natural to ask which homotopies $H \subseteq \grp(Q)$, for a 2-acyclic quiver $Q$, induce FZ-mutations quivers for any sequence of mutations with homotopies. In fact, this can be achieved by taking the maximal homotopy, given by $H(i,i)=\pi_1(Q,i)$ for all $i\in Q_0$.

\begin{proposition}\label{prop: maximal H induces FZ mutation}
    Let $(Q, H)$ be a quiver with homotopy where $Q$ is 2-acyclic and $H(i,i)=\pi_1(Q,i)$ for all $i\in Q_0$. For any sequence of mutations $\mu_{k_\ell}\cdots \mu_{k_1}(Q, H) = (Q', H')$, we have $Q' = \fzmu_{k_\ell}\cdots \fzmu_{k_1}(Q)$.
\end{proposition}

\begin{proof}
    It suffices to show that for any mutation sequence, the quiver $Q'$ is 2-acyclic. Assuming that every $Q'$ is 2-acyclic, they must be related by FZ-mutations since a maximal collection of $2$-cycles must be deleted in the mutation process with homotopies. In fact, by the maximal assumption of $H$, we have $(\grp(Q)/H)(i, i) = 1$ (the trivial group) and thus $(\grp(Q')/H')(i, i) = 1$ by the invariance of quotient groupoids (see \Cref{rmk: quotient groupoid invariant}). Since $H'$ is reduced, there is no $2$-cycle in $Q'$.
\end{proof}

\begin{remark}
    Nontrivial elements in $\grp(Q)(i,i)=\pi_1(Q,i)$ can be expressed as reduced cyclic walks (\Cref{def: cyclic walk}) at $i$ in $Q$. We call a reduced cyclic walk \emph{chordless} if its vertices are distinct except for the initial and final vertex, and no arrow joins nonconsecutive vertices in either direction. The group $\grp(Q)(i,i)$ is generated by $p^{-1}cp$, where $c$ is a chordless cyclic walk based at $j$ and $p$ is a walk from $i$ to $j$. Indeed, splitting a cyclic walk at a repeated vertex or along a chord expresses its class as a product of conjugates of shorter cyclic walks; induction on length gives the claim.
\end{remark}

To another extreme, one can take the trivial homotopy. For a general $2$-acyclic quiver with trivial homotopy, 2-cycles are easily created and not deleted after mutations. However, this is not the case for acyclic quivers, i.e., quivers without oriented cycles. We will see in \Cref{subsec: homotopy and orbit mutation} that for acyclic quivers, the trivial homotopy also realizes FZ-mutation patterns, hence equivalent to the maximal homotopy in \Cref{prop: maximal H induces FZ mutation}.

\subsection{Homotopies and orbit mutations}\label{subsec: homotopy and orbit mutation}

Let $(Q, H)$ be a quiver with homotopy. It is natural to consider, via the Galois correspondence (\Cref{prop: galois correspondence}), the regular covering 
\begin{equation}\label{eq: galois correspond cover}
    p = p_H \colon Q_H \rightarrow Q
\end{equation}
such that for each $x\in Q_0$ and any $\tilde x\in p^{-1}(x)$, we have 
\[
    p_*(\pi_1(Q_H, \tilde x)) = H(x, x) \subseteq \pi_1(Q, x)
\] 
where $p_*$ denotes the map between fundamental groups induced by $p$. In this case, we write $H = p_*(\pi(Q_H))$. When $Q$ is loop-free and $H$ is reduced, it is clear that $Q_H$ is $2$-acyclic. So the covering $p$ is weakly admissible (\Cref{def: weakly admissible}). However it is not necessarily $[k]$-mutable for $k\in Q_0$.

Assume that $p$ is $[k]$-mutable. This means that the (regular) covering
\[
    p^\ddag \coloneqq \fzmu_{[k]}(p) \colon \fzmu_{[k]}(Q_H) \rightarrow \fzmu_{[k]}(Q_H)/\Gamma
\]
(where $\Gamma = \mathrm{Deck}(p)$) is again weakly admissible, which means that there is no $\Gamma$-loop in $\widetilde Q^\ddag \coloneqq \fzmu_{[k]}(Q_H)$ and equivalently that there is no loop in the quotient quiver $Q^\ddag \coloneqq \fzmu_{[k]}(Q_H)/\Gamma$. The regular covering $p^\ddag$ corresponds to a normal subgroupoid $H^\ddag$ of $\pi(Q^\ddag)$, that is, by our notation
\[
    H^\ddag = p^\ddag_*(\pi(\widetilde Q^\ddag)) \subseteq \pi(Q^\ddag).    
\]

On the other hand, we can mutate the quiver with homotopy $(Q, H)$ as in \Cref{subsec: mutations with homotopy}. The following result shows the compatibility between orbit mutations and mutations of quivers with homotopies.

\begin{proposition}\label{prop: orbit mutation compatible}
    Let $\widetilde {\mu}_k(Q, H) = (Q', H')$ and $\mu_k(Q, H) = (Q^\dag, H^\dag)$. Under the assumption that $p$ is $[k]$-mutable, the choice made as in \Cref{lemma: remove 2 cycle orbit} induces a choice of deletion of $2$-cycles in \Cref{def: delete 2-cycles} from $Q'$ to $Q^\dagger$. This gives rise to a ($Q_0$-fixing) quiver isomorphism 
    \[
        f \colon Q^\ddag \rightarrow Q^\dagger \quad \text{such that} \quad  H^\dag = f_*(H^\ddag).
    \]
\end{proposition}

\begin{proof}
    We introduce an intermediate step in the orbit mutation as follows. We first delete from the orbit pre-mutation $\prefzmu_{[k]}(Q_H)$ all the 2-cycles between any two vertices in the same $\Gamma$-orbit. Since the orbit mutation is assumed to still be weakly admissible, after deleting those 2-cycles, there are no arrows in between any two vertices in the same $\Gamma$-orbit, that is, there is no $\Gamma$-loop. Denote this quiver by $\prefzmu_{[k]}(Q_H)^\mathrm{lf}$, where `lf' stands for $\Gamma$-loop free. The group $\Gamma$ acts freely on 
    \[
        Q_H' \coloneqq \prefzmu_{[k]}(Q_H)^\mathrm{lf},
    \]
    inducing a regular covering
    \[
        \prefzmu_{[k]}(p)^\mathrm{lf}\colon Q_H' \rightarrow Q_H'/\Gamma,
    \]
    where the quotient quiver is loop-free. By \Cref{lemma: delete loops}, the quotient quiver naturally identifies with $Q' = \prefzmu_k(Q)$. Since the deck transformation group does not change, we have
    \[
        \pi(Q)/H \cong \pi(Q')/\prefzmu_{[k]}(p)^\mathrm{lf}_*(\pi(Q'_H)).
    \]
    This isomorphism is induced by the same map $\varphi$ from $\pi(Q')$ to $\pi(Q)/H$ as in \Cref{def: pre-mutation}. Therefore by the definition of $H'$, we have
    \begin{equation}\label{eq: homotopy intermediate}
        H' = \prefzmu_{[k]}(p)^\mathrm{lf}_*(\pi(Q'_H)).
    \end{equation}

    Now we move on to the deletion of 2-cycles between $i$ and $j$ that are not in the same $\Gamma$-orbit. Make a choice as in \Cref{lemma: remove 2 cycle orbit}. The arrows in the collection of 2-cycles to be removed are made of $\Gamma$-orbits. Thus removing them induces removing 2-cycles in the quotient quiver $Q'$ as arrows in the quotient quiver are simply orbit sets.

    We then need to show that the above cancellation indeed qualifies for the deletion process from $Q'$ to $Q^\dagger$ in \Cref{def: delete 2-cycles}. Let $i$ be in $(Q_H)_0$ that is not in $p^{-1}(k)$. Consider a $\Gamma$-orbit $J\subseteq (Q_H)_0$ other than $p^{-1}(k)$ or $\Gamma i$. There might be 2-cycles between $i$ and any $j\in J$ in $\prefzmu_{[k]}(Q_H)^\mathrm{lf}$. In fact, such $2$-cycles, because of (\ref{eq: homotopy intermediate}), must be in the homotopy $H'$. Thus the choice made in \Cref{lemma: remove 2 cycle orbit} induces a choice in \Cref{def: delete 2-cycles}. The partition of arrows between 
    \[
        p(i)\quad \text{and} \quad p(j)
    \]
    by different $j$'s also coincides with such a partition considered in the proof of \Cref{lemma: unique reduction up to iso}.

    We now have shown that the quotient quiver $Q^\ddag = \fzmu_{[k]}(Q_H)/\Gamma$ upon the choice identifies with $Q^\dagger$ (with the identification named $f$ in the statement). Recall that $H^\dagger$ is defined so that the quotient groupoid $\grp(Q^\dagger)/H^\dagger$ is $\grp(Q)/H$ (see \Cref{rmk: quotient groupoid invariant}), which does not change under the orbit mutation (as the deck transformation group remains unchanged). Hence we have finally $H^\dag = f_*(H^\ddag)$.
\end{proof}

The following is a direct corollary of \Cref{prop: orbit mutation compatible}.

\begin{corollary}\label{cor: global pattern}
    Let $p_v\colon \widetilde Q_v \rightarrow Q_v$ be a weakly admissible covering associated to each $v\in \mathbb T_n$ such that for any $v \frac{k}{\quad\quad} v'$ in $\mathbb T_n$, we have $p_v = \fzmu_{[k]}(p_{v'})$. In other words, each $p_v$ is globally weakly admissible and they are related by orbit mutations. Then the assignment $v\mapsto (Q_v, H_v)$, $v\in \mathbb T_n$ where
    \[
        H_v\coloneqq (p_v)_*(\pi(\widetilde Q_v))
    \]
    is a pattern of quivers with homotopies on $\mathbb T_n$. In other words, for any $v \frac{k}{\quad\quad} v'$ in $\mathbb T_n$, we have
    \[
        \mu_{k}(Q_v, H_v) = (Q_{v'}, H_{v'}).
    \]
\end{corollary}

\begin{example}
    Consider the covering $p\colon \widetilde Q \rightarrow Q$ in \Cref{fig: glob weak admin covering}. The homotopy $H = p_*(\pi(\widetilde Q))$ is generated by $\{abab, cdcd, efef, fca, ebd\}$.  It is easy to verify that $p$ is globally weakly admissible. By \Cref{cor: global pattern}, mutations of quivers with homotopies can be computed by orbit mutations.
\end{example}

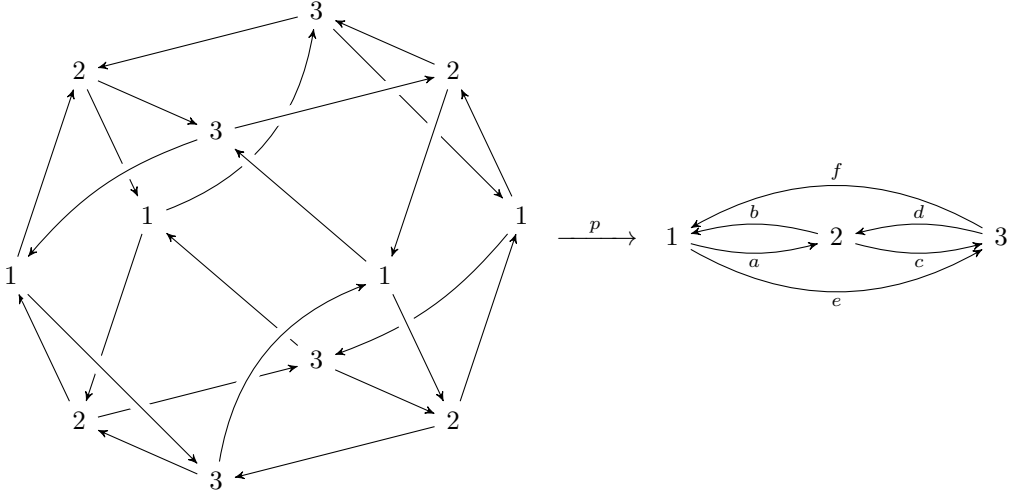
\begin{figure}[ht]
    \centering
    \begin{equation*}
    \begin{tikzcd}[sep = small]
    	&&&&& 3 \\
    	& 2 &&&&&& 2 \\
    	&&& 3 \\
    	\\
    	&& 1 &&&&&& 1 \\
    	1 &&&&&& 1 \\
    	\\
    	&&&&& 3 \\
    	& 2 &&&&&& 2 \\
    	&&& 3
    	\arrow[from=2-2, to=5-3]
    	\arrow[from=5-9, to=2-8]
    	\arrow[from=2-2, to=3-4]
    	\arrow[from=1-6, to=2-2]
    	\arrow[from=2-8, to=1-6]
    	\arrow[bend right = 30, from=5-3, to=1-6]
    	\arrow[from=9-2, to=8-6]
    	\arrow[from=8-6, to=9-8]
    	\arrow[from=9-8, to=10-4]
    	\arrow[from=10-4, to=9-2]
    	\arrow[from=5-3, to=9-2]
    	\arrow[from=9-2, to=6-1]
    	\arrow[from=6-1, to=2-2]
    	\arrow[from=9-8, to=5-9]
    	\arrow[from=8-6, to=5-3]
    	\arrow[bend left = 15, from=5-9, to=8-6]
    	\arrow[from=1-6, to=5-9]
        \arrow[crossing over, from=3-4, to=2-8]
        \arrow[crossing over, from=2-8, to=6-7]
    	\arrow[crossing over, from=6-7, to=9-8]
        \arrow[bend right = 15, crossing over, from=3-4, to=6-1]
    	\arrow[crossing over, from=6-1, to=10-4]
    	\arrow[bend left = 30, crossing over, from=10-4, to=6-7]
    	\arrow[crossing over, from=6-7, to=3-4]
    \end{tikzcd}
    \ 
    \xrightarrow{\quad p\quad }
    \ 
    \begin{tikzcd}[column sep = huge]
        1 \ar[r, bend right=15, swap, "a"] \ar[rr, bend right, swap, "e"] & 2 \ar[r, bend right=15, swap, "c"] \ar[l, bend right=15, swap, "b"] & 3 \ar[l, bend right=15, swap, "d"] \ar[ll, bend right, swap, "f"]
    \end{tikzcd}
\end{equation*}
    \caption{A globally weakly admissible covering.}
    \label{fig: glob weak admin covering}
\end{figure}

Another application of \Cref{prop: orbit mutation compatible} is the following sufficient condition for globally weakly admissible coverings. Let $\pi(Q)^2$ denote the normal subgroupoid of $\pi(Q)$ generated by squares of cyclic walks.

\begin{proposition}\label{prop: global admissible pi square}
    Any weakly admissible covering $p\colon \widetilde Q \rightarrow Q$ such that $\pi(Q)^2\subseteq p_*(\pi(\widetilde Q))$ is globally weakly admissible. Therefore for any reduced homotopy $H$ on $Q$ such that $\pi(Q)^2\subseteq H$, the covering \eqref{eq: galois correspond cover} $p_H\colon Q_H \rightarrow Q$ is globally weakly admissible.
\end{proposition}

\begin{proof}
    For $k\in Q_0$, let
    \[
        p^\ddag \colon \widetilde Q^{\ddag} \rightarrow Q^{\ddag}
    \]
    denote the orbit mutation of $p$ at $[k]$. By \Cref{lemma: sufficient k-mutable}, $p$ is $[k]$-mutable and thus $p^\ddag$ is again weakly admissible. We only have to show $\pi(Q^\ddag)^2\subseteq p^\ddag_*(\pi(\widetilde Q^\ddag))$; then by induction $p$ is globally weakly admissible.

    Choose a base point $v\in Q_0$ and a lift $\tilde v\in p^{-1}(v)$. By \Cref{prop: orbit mutation compatible}, we have
    \[
        Q^\ddag = Q^\dag \quad \text{and} \quad p^\ddag_*(\pi(\widetilde Q^\ddag)) = H^\dag,
    \]
    where $(Q^\dag, H^\dag) = \mu_k(Q, p_*(\pi(\widetilde Q)))$. Therefore the quotient groups are isomorphic by \Cref{rmk: quotient groupoid invariant}:
    \[
        \pi_1(Q^\ddag, v)/p^\ddag_*(\pi_1(\widetilde Q^\ddag, \tilde v)) \cong \pi_1(Q^\dag, v)/H^\dag(v,v) \cong \pi_1(Q,v)/p_*(\pi_1(\widetilde Q, \tilde v)).
    \]
    Any non-trivial element in $\pi_1(Q,v)/p_*(\pi_1(\widetilde Q, \tilde v))$ is of order two because of the assumption $\pi(Q)^2\subseteq p_*(\pi(\widetilde Q))$. Hence so is $\pi_1(Q^\ddag, v)/p^\ddag_*(\pi_1(\widetilde Q^\ddag, \tilde v))$, that is, $\pi(Q^\ddag)^2\subseteq p^\ddag_*(\pi(\widetilde Q^\ddag))$.
\end{proof}

Now we turn to understand equivalence classes of homotopies on an acyclic quiver as promised in \Cref{subsec: patterns of QH}. We use \Cref{prop: orbit mutation compatible}, \Cref{cor: global pattern} and a theorem of Huang and Li \cite{HL} to prove

\begin{theorem}\label{thm: any homotopy induces fz mutation acyclic}
    Let $Q$ be an acyclic quiver. Let $H$ be as in \Cref{prop: maximal H induces FZ mutation}, that is $H(i, i) = \pi_1(Q, i)$ for any $i\in Q_0$. Then $H$ is equivalent to the trivial homotopy $\trht$. In fact, any homotopy is equivalent to $\trht$.
\end{theorem}

\begin{proof}
    Notice that the regular covering $p_\trht$ corresponding to the trivial homotopy is the universal covering $\hat p \colon \widehat Q \rightarrow Q$. By \cite[Theorem 2.6]{HL}, the universal covering $\hat p$ is globally admissible (\Cref{def: weakly admissible}). Thus the quotient quivers along orbit mutations are related by FZ-mutations. Meanwhile by \Cref{cor: global pattern}, the quotient quivers coincide with the ones obtained from mutations with homotopies. Thus $\trht\sim_Q H$, since $H$ induces FZ-mutations by \Cref{prop: maximal H induces FZ mutation}. 
    
    Any other homotopy $H'$ can be regarded as in between $H$ and $\trht$. Let $p\colon \widetilde Q \rightarrow Q$ denote the regular covering corresponding to $H'$. Then the universal covering $\hat p$ factors through $p$. Write $\widetilde Q=\widehat Q/N$, where $N$ is a normal subgroup of $\Gamma=\operatorname{Deck}(\hat p)$. We show by induction on the number of mutations that $p$ remains admissible. Initially, this holds since $Q$ is acyclic. At each step, orbit pre-mutation commutes with taking the quotient by $N$, and the $\Gamma$-equivariant deletions of $2$-cycles descend to this quotient. The resulting quiver covers the mutated base quiver, which is $2$-acyclic by global admissibility of $\hat p$. It is therefore also $2$-acyclic, so the descended deletions are maximal and give the orbit mutation of the intermediate covering. This completes the induction. Applying \Cref{cor: global pattern} to $p$, we see that $H'$ is equivalent to $\trht$.
\end{proof}

The above theorem leads to a ``mutation-invariant'' for quivers obtained by FZ-mutations from a given acyclic quiver $Q$. It asserts that $Q$ is among the simplest ones in its FZ-mutation class.

\begin{corollary}\label{cor: lower bound fund group}
    Let $Q'$ be a quiver obtained by FZ-mutations from an acyclic quiver $Q$. Then the rank of $\pi_1(Q')$ is no smaller than the rank of $\pi_1(Q)$.
\end{corollary}

\begin{proof}
    Let $H'$ be the homotopy on $Q'$ such that $(Q', H')$ is obtained from $(Q, \trht)$ by mutations with homotopies. Notice that the quotient group $\pi_1(Q', i)/H'(i, i)$ remains unchanged and is thus isomorphic to the fundamental group $\pi_1(Q, i)$ where $i\in Q_0$. This implies that $\pi_1(Q', i)$ (always isomorphic to a free group $F_n$) has rank $n$ no less than the rank of $\pi_1(Q, i)$.
\end{proof}

\subsection{A comparison between homotopies and potentials}\label{subsec: compare homotopy potential}

The notion of homotopies parallels, but differs from the notion of \emph{potentials} on quivers developed by Derksen--Weyman--Zelevinsky \cite{DWZ08}. In their framework, mutations of quivers depend on the choice of a potential, defined as a formal linear combination of cyclic paths on the quiver. A potential is called \emph{degenerate} if some mutation sequence leads to the creation of a $2$-cycle that cannot be eliminated; otherwise it is \emph{non-degenerate}. We note that the DWZ mutation of quivers with potentials has not been fully adapted to quivers containing $2$-cycles. Indeed, if a mutation occurs at a vertex incident to $2$-cycles, one would need to introduce loop edges in order to write down the mutated potential $\widetilde{S}$ in \cite[(5.8)]{DWZ08}.

By contrast, this difficulty is completely avoided in the pre-mutation step with homotopies (see \Cref{def: pre-mutation}). From our perspective, all homotopies are equally natural because there is no obstruction to subsequent mutations even after the appearance of $2$-cycles. Some homotopies reproduce ordinary FZ-mutations (as in \Cref{prop: maximal H induces FZ mutation} and \Cref{thm: any homotopy induces fz mutation acyclic}), while others (mentioned in \Cref{subsec: qh from surface}), as we will see in \Cref{section: marked surfaces}, yield quiver mutations with $2$-cycles governed by surface combinatorics.

In view of \Cref{subsec: patterns of QH}, starting with the trivial homotopy on a loop-free quiver $Q$ and forming the resulting pattern $(Q_v, H_v)$, $v\in \mathbb{T}_n$, then it is in general unclear (already for $2$-acyclic quivers containing cyclic paths) how to realize these quivers $Q_v$ through mutations with potentials. Indeed, the quivers $Q_v$ typically contain $2$-cycles, and mutations with potentials are therefore easily obstructed.

\begin{example}
    Consider \Cref{example of mutation 1}. If we regard the cyclic word $dcba$ as the cyclically equivalent potential $W = cbad$, then the usual DWZ mutation at vertex $1$ is not directly applicable. Indeed, the subword $ad$ passes through the mutation vertex $1$, but the corresponding composite arrow $[ad]$ would be a loop. Not creating this loop would leave no arrow by which the subword $ad$ in $W$ can be replaced. 
    
    By contrast, in the framework of homotopies, one can use formal inverses: after reversing $a$ and $d$, the subword $ad$ is replaced by $(a^\star)^{-1}(d^\star)^{-1}$. Thus the homotopy relation can still be transported through mutation without creating a loop.
\end{example}

\section{Quivers with homotopies from surface triangulations}\label{section: marked surfaces}

Fomin, Shapiro and Thurston \cite{FST08} studied a class of cluster algebras associated to marked bordered surfaces, building upon work of Gekhtman, Shapiro and Vainshtein \cite{GSV} and of Fock and Goncharov \cite{FG1,FG2}. In particular they introduced tagged triangulations of those surfaces and attached a 2-acyclic quiver to any tagged triangulation such that flips of triangulations induce FZ-mutations of the attached quivers.

In this section we generalize the Fomin--Shapiro--Thurston model to allow a new kind of additional interior marked points, referred to as $\coii$-punctures, whereas the original interior marked points are called $\coi$-punctures. We associate a quiver with homotopy $(Q(T), H(T))$ to each tagged ideal triangulation $T$. The $\coii$-punctures lead to the presence of $2$-cycles in the associated quivers of certain triangulations.

We will show in \Cref{thm: mutation and flip} that flips of tagged triangulations induce mutations of quivers with homotopies. The proof uses a construction of a $2$-complex $X(T)$ such that 
\[
    \pi_1(X(T), X(T)_0) \cong \pi(Q(T))/H(T).
\]
By \Cref{prop: fund grp surface}, $X(T)$ is homotopy equivalent to the CW-complex $\widehat{\S}_c$, which is independent of the triangulation.

\subsection{Bordered surfaces with marked points and colored punctures}

The following definition follows the terminology of \cite{FST08}.

\begin{definition}[Bordered surfaces with marked points]\label{def: marked bordered surface}
Let $\S$ be a compact connected oriented surface with or without boundary. Fix a non-empty finite set $\M$ of \emph{marked points} on $\S$ such that there is at least one marked point on each boundary component of $\S$. Marked points in the interior of $\S$ are called \emph{punctures}, the collection of which is denoted by $\punc\subseteq \M$.
\end{definition}

As a generalization of the Fomin--Shapiro--Thurston model, we assign two different colors $\{\coi, \coii\}$ to punctures.

\begin{definition}\label{def: coloring}
    A \emph{coloring} on $(\S, \M)$ is a function $c\colon \punc \rightarrow \{\coi, \coii\}$. Write $\punc_\coi = c^{-1}(\coi)$ and $\punc_{\coii} = c^{-1}({\coii})$. A triple $(\S, \M, c)$ is called a \emph{marked bordered surface with colored punctures} if it is none of the following:
    \begin{itemize}
        \item a sphere with one or two punctures;
        \item an unpunctured or once-punctured monogon;
        \item an unpunctured digon or an unpunctured triangle;
        \item a thrice-punctured sphere where $\punc_{\coii}\neq \punc$;
        \item a sphere with four punctures where $|\punc_{\coi}| \geq 3$;
        \item a sphere with $|\punc_{\coii}|=1$.
    \end{itemize}
\end{definition}

Punctures colored by $\coi$ are treated the same as in the setting of \cite{FST08}, whereas $\coii$-punctures represent new data.

\begin{remark}
    A sphere with four $\coi$-punctures is allowed in \cite{FST08}. We exclude this case because some triangulations do not admit a puzzle-piece decomposition which we use later to define a quiver with homotopy.
\end{remark}

We define \emph{arcs} in $(\S, \M)$ in the same way as in \cite[Definition 2.2]{FST08}. Briefly, arcs are simple curves in $\S$ avoiding $\M$ in their interior but having endpoints in $\M$. Arcs are not contractible into $\M$ or onto the boundary of $\S$, and are considered up to isotopy. 

Two arcs are called \emph{compatible} if they do not intersect in the interior of $\S$. A maximal collection of mutually compatible arcs is called an \emph{ideal triangulation}, or simply a triangulation, of $(\S, \M)$. By arcs and triangulations of $(\S, \M, c)$, we mean those of $(\S, \M)$.

\subsection{Quivers with homotopies associated to triangulations}

We first construct a quiver $Q(T)$ for any triangulation $T$ of $(\S, \M, c)$ and then construct a homotopy $H(T)$ on $Q(T)$. Typically the quiver $Q(T)$ may contain 2-cycles when $\punc_{\coii}$ is non-empty.

Following \cite[Remark 4.2]{FST08}, we decompose any triangulation $T$ into a number of ``puzzle pieces''. There are four kinds of puzzle pieces and to each kind we assign a quiver whose vertices are arcs or boundary segments. The four puzzle pieces are displayed in \Cref{fig: puzzle pieces and their quivers} and \Cref{fig: self-folded II puncture} where a $\coi$-puncture is depicted as \begin{tikzpicture}
    \node[coii] at (0,0) {};
    \node[fill, circle, inner sep=0.5pt] at (0,0) {};
\end{tikzpicture}, a $\coii$-puncture as \begin{tikzpicture}
    \node[coii] at (0,0) {};
\end{tikzpicture}, and a puncture without a dedicated color or a boundary marked point as \begin{tikzpicture}
    \node[generic] at (0,0) {};
\end{tikzpicture}.

The three puzzle pieces in \Cref{fig: puzzle pieces and their quivers} already show up in the setting of \cite{FST08}. There is a new puzzle piece though: a self-folded triangle enclosing a $\coii$-puncture (\Cref{fig: self-folded II puncture}). Such a puzzle piece can arise when the surface (regardless of the coloring) is first decomposed into the three puzzle pieces in \Cref{fig: puzzle pieces and their quivers} and then a second or third puzzle piece may contain a self-folded triangle with a $\coii$-puncture, and hence further decomposes.

\begin{remark}
    There are cases for which the existence of puzzle-piece decomposition does not follow from that of \cite{FST08}. They are a thrice-punctured sphere with $\punc_{\coii} = \punc$ and a sphere with four punctures with $|\punc_{\coii}|=2,3,4$. Nonetheless it can be directly verified in these cases that any triangulation decomposes into the four kinds of puzzle pieces.
\end{remark}

\begin{figure}[ht]
    \centering
    \begin{tikzpicture}
        \coordinate (1A) at (0, {sqrt(3)});
        \coordinate (1B) at (-1, 0);
        \coordinate (1C) at (1, 0);

        \node[] at (1.7, 0) {$P_1$};

        \filldraw[black] (1A) circle (2pt);
        \filldraw[black] (1B) circle (2pt);
        \filldraw[black] (1C) circle (2pt);

        \draw[thick] (1A) -- (1B) node[midway, left] {$1$};
        \draw[thick] (1B) -- (1C) node[midway, below] {$3$};
        \draw[thick] (1C) -- (1A) node[midway, right] {$2$};

        \node[] (T1) at (-1, -1) {$1$};
        \node[] (T2) at (1, -1) {$2$};
        \node[] (T3) at (0, {-1-sqrt(3)}) {$3$};

        \path[->, >=stealth']
        (T1) edge (T3)
        (T3) edge (T2)
        (T2) edge (T1);

      \coordinate (2A) at (5,3); 
      \coordinate (2B) at (5,0); 
      \coordinate (2C) at (5,1.5); 
      \coordinate (2D) at (4,1.5);
      \coordinate (2E) at (6,1.5);
      \coordinate (2F) at (5,2.2);

      \node[] at (6, 0) {$P_2$};
    
      \draw[thick] (2A) to[bend left=60] (2B);
      \draw[thick] (2A) to[bend right=60] (2B);
    
      \draw[thick, black] (2C) -- (2B) node[fill=white, midway] {$4$};
      \draw[thick] (2B) .. controls  (4, 2.5) and (6, 2.5) .. (2B) node[pos=0.5, above] {$3$};
    
      \filldraw[black] (2A) circle (2pt);
      \filldraw[black] (2B) circle (2pt);
      \node[coi] at (2C) {};
    
      \draw (2D) node[fill=white] {$1$};
      \draw (2E) node[fill=white] {$2$};

      \node[] (S1) at (4,-2) {$1$};
      \node[] (S2) at (6,-2) {$2$};
      \node[] (S3) at (5,-1) {$3$};
      \node[] (S4) at (5,-3) {$4$};

      \path[->,>=stealth']
      (S1) edge (S3)
      (S1) edge (S4)
      (S2) edge (S1)
      (S3) edge (S2)
      (S4) edge (S2);

      \coordinate (3A) at (10.5, 0);
      \coordinate (3B) at ({10.5-0.6}, 1.2);
      \coordinate (3C) at ({10.5+0.6}, 1.2);

      \node[] at (12, 0) {$P_3$};

      \filldraw[black] (3A) circle (2pt);
      \node[coi] (3B) at (3B) {};
      \node[coi] (3C) at (3C) {};

      \draw[black, thick] (10.5, 1.5) circle (1.5);

      \draw[black, thick] (3A) -- (3B) node[fill=white, pos=0.6] {$5$};
      \draw[black, thick] (3A) -- (3C) node[fill=white, pos=0.6] {$3$};

      \draw[thick] (3A) .. controls ({10.5-2}, 2) and ({10.5-0.4}, 2.6) .. (3A) node[midway, above] {$4$};
      \draw[thick] (3A) .. controls ({10.5+2}, 2) and ({10.5+0.4}, 2.6) .. (3A) node[midway, above] {$2$};

      \node[below] at (10.5, 3) {$1$};

      \node[] (R1) at (10.5, -1) {$1$};
      \node[] (R2) at (11.5, -2) {$2$};
      \node[] (R3) at (11.5, -3) {$3$};
      \node[] (R4) at (9.5, -2) {$4$};
      \node[] (R5) at (9.5, -3) {$5$};

      \path[->, >=stealth']
      (R4) edge (R2)
      (R1) edge (R4)
      (R2) edge (R1)
      (R5) edge (R2)
      (R1) edge (R5)
      (R4) edge (R3)
      (R3) edge (R1)
      (R5) edge (R3);      
    \end{tikzpicture}
    \caption{From left to right: three puzzle pieces $P_1, P_2, P_3$ and their associated quivers.}
    \label{fig: puzzle pieces and their quivers}
\end{figure}
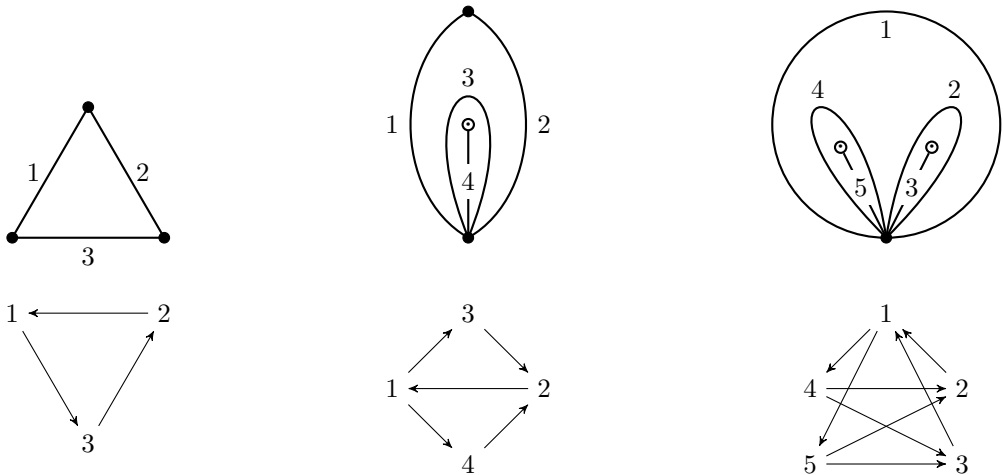

\begin{figure}[ht]
    \centering
    \begin{tikzpicture}[scale=1.1]
        \coordinate (A) at (0,0);
        \coordinate (B) at (0,1.5);

        \draw[thick] (A) -- (B) node[fill=white, pos=0.5] {$1$};
        \draw[thick] (A) .. controls (-1,2.5) and (1,2.5) .. (A) node[above, pos=0.5] {$2$};

        \filldraw[fill=white, thick] (B) circle (2pt);
        \filldraw[black] (A) circle (2pt);

        \node[] (X) at (1.5, 1) {$1$};
        \node[] (Y) at (3, 1) {$2$};

        \path[->, >=stealth']
        (X) edge[bend right=30] (Y)
        (Y) edge[bend right=30] (X);
    \end{tikzpicture}
    \caption{The new puzzle piece $P_4$ and its associated quiver.}
    \label{fig: self-folded II puncture}
\end{figure}
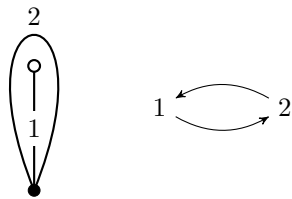

\begin{definition}\label{def: quiver with 2-cycle of T}
    The quiver $\widetilde Q(T)$ is constructed by gluing the associated quivers of all puzzle pieces along their common arcs. The quiver $Q(T)$ is obtained from $\widetilde Q(T)$ by one additional rule of deleting certain 2-cycles: whenever two triangles are glued as in \Cref{fig: I-puncture digon delete 2-cycle} where the middle marked point is $\coi$-colored, the 2-cycle (dashed) between $1$ and $2$ is deleted.
\end{definition}

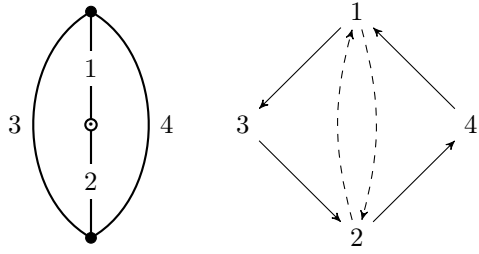
\begin{figure}[ht]
\centering
\begin{tikzpicture}[scale=0.8]
  \coordinate (A) at (0,1.5); 
  \coordinate (B) at (0,-1.5); 
  \coordinate (C) at (0,0); 
  \coordinate (D) at (-1.2,0);
  \coordinate (E) at (1.2,0);

  \draw[thick] (A) to[bend left=60] (B);
  \draw[thick] (A) to[bend right=60] (B);

  \draw[thick, black] (C) -- (A) node[fill=white, midway] {$1$};
  \draw[thick, black] (C) -- (B) node[fill=white, midway] {$2$};

  \filldraw[black] (A) circle (2pt);
  \filldraw[black] (B) circle (2pt);
  \node[coi] (C) at (C) {};

  \draw (D) node[fill=white] {$3$};
  \draw (E) node[fill=white] {$4$};

  \node[] (3) at (3, 0) {3};
  \node[] (1) at (4.5, 1.5) {1};
  \node[] (2) at (4.5, -1.5) {2};
  \node[] (4) at (6, 0) {4};

  \path[->, >=stealth']
  (1) edge (3)
  (3) edge (2)
  (2) edge (4)
  (4) edge (1)
  (1) edge[dashed, bend left=15] (2)
  (2) edge[dashed, bend left=15] (1);
\end{tikzpicture}
\caption{The dashed 2-cycle is deleted in $Q(T)$.}
\label{fig: I-puncture digon delete 2-cycle}
\end{figure}

\begin{remark}
    Without $\coii$-punctures, the quiver $Q(T)$ contains no 2-cycles, and its full subquiver on the vertices corresponding to arcs recovers the construction in \cite[Definition 4.1]{FST08}, given equivalently in the form of a skew-symmetric integer matrix $B(T)$.
\end{remark}

Next we define a homotopy $\widetilde H(T)$ on $\widetilde Q(T)$ and a reduced homotopy $H(T)$ on $Q(T)$. The following notion is needed.

\begin{definition}\label{def: puncture cycle}
    Given a triangulation $T$, define the cyclic walk $C_p = C_p(T)$ in $\widetilde{Q}(T)$ for any $p\in \punc_{\coi}$ of valency at least $2$ or any $p\in \punc_{\coii}$.
    \begin{enumerate}
        \item If $p\in \punc$ is of valency at least $2$, draw a small simple circle $\gamma$ around $p$, clockwise oriented. Choose an initial arc $i\in T$ that intersects with $\gamma$ and that is not enclosed in a self-folded triangle. Then read off the cyclic walk $C_p$ in $\widetilde Q(T)$ starting and ending at $i$ following $\gamma$ with one extra care: whenever $\gamma$ crosses a self-folded triangle enclosing a $\coi$-puncture, the cyclic walk $C_p$ skips over (the quiver vertex representing) the enclosed arc.
        \item If $p\in \punc_{\coii}$ is of valency $1$, i.e., $p$ lies inside a $P_4$-type puzzle piece with the radius $i$ and the enclosing arc $j$, then $C_p \coloneqq ab$ where $i\xrightarrow{b}j$ and $j\xrightarrow{a}i$. 
    \end{enumerate}
\end{definition}

We regard $C_p$ as an element in $\pi(\widetilde Q(T))(i, i) = \pi_1(\widetilde Q(T), i)$. Note that $C_p$ does not contain any reversed arrow of $\widetilde{Q}(T)$, so it is an oriented cycle.

\begin{example}
    Consider the surface triangulation and its associated quiver in \Cref{fig: oriented cycle around I-puncture}. Choose $i$ to be the vertex $6$. Then $C_p$ starts at $6$ and passes through the vertex $3$ but not $5$. Let $q$ be the $\coii$-puncture enclosed by the arc $2$. Then $C_q$ is the $2$-cycle from $3$ to $2$ and back. A $\coi$-puncture of valency $1$ (for example the one enclosed by arc $5$) does not contribute any such cycle by definition.
\end{example}

\begin{figure}
    \centering
    \begin{tikzpicture}[scale=0.9]
      \coordinate (3A) at (10.5, 0);
      \coordinate (3B) at ({10.5-0.6}, 1.2);
      \coordinate (3C) at ({10.5+0.6}, 1.2);
      \coordinate (3D) at (10.5, -1);
      \coordinate (3E) at (10.5, 4);

      \draw[black, thick] (10.5, 1.5) circle (1.5);
      \draw[black, thick] (10.5, 1.5) circle (2.5);
      \draw[black, thick] (10.5, 2) circle (2);

      \draw[black, thick] (3A) -- (3B) node[fill=white, pos=0.6] {$5$};
      \draw[black, thick] (3A) -- (3C) node[fill=white, pos=0.6] {$3$};
      \draw[black, thick] (3A) -- (3D) node[midway, left] {$6$};

      \draw[thick] (3A) .. controls ({10.5-2}, 2) and ({10.5-0.4}, 2.6) .. (3A) node[midway, above] {$4$};
      \draw[thick] (3A) .. controls ({10.5+2}, 2) and ({10.5+0.4}, 2.6) .. (3A) node[midway, above] {$2$};

      \filldraw[black] (3A) circle (2pt);
      \node[coi] (3B) at (3B) {};
      \node[coii] (3C) at (3C) {};
      \filldraw[black] (3E) circle (2pt);
      \filldraw[black] (3D) circle (2pt);

      \node[below] at (10.5, 3) {$1$};

      \node[] at ({10.5-1.8}, 0) {$7$};
      \node[] at ({10.5+1.8}, 0) {$8$};
      \node[] at ({10.5-1.9}, 1) {$9$};
      \node[] at ({10.5+1.97}, 1) {$10$};
      \node[below right] at (3A) {$p$};
      \node[above right] at (3C) {$q$};

      \node[] (2) at (17, 1.5) {$2$};
      \node[] (1) at (17, 3.5) {$1$};
      \node[] (3) at (17, 0.5) {$3$};
      \node[] (4) at (16, 2.5) {$4$};
      \node[] (5) at (18, 2.5) {$5$};
      \node[] (6) at (17, -0.5) {$6$};
      \node[] (7) at (15, 1.5) {$7$};
      \node[] (8) at (19, 1.5) {$8$};
      \node[] (9) at (15, 3.5) {$9$};
      \node[] (10) at (19, 3.5) {$10$};

      \path[->, >=stealth']
      (1) edge (4)
      (4) edge (2)
      (2) edge (1)
      (1) edge (5)
      (5) edge (2)
      (2) edge[bend right] (3)
      (3) edge[bend right] (2)
      (9) edge (1)
      (1) edge (10)
      (10) edge[bend right] (9)
      (9) edge (7)
      (7) edge[bend right] (6)
      (6) edge[bend left=10] (9)
      (10) edge[bend left=10] (6)
      (6) edge[bend right] (8)
      (8) edge (10);

      \node[] at (17, -1.1) {$C_p=6\leftarrow 10 \leftarrow 1 \leftarrow 2 \leftarrow 3 \leftarrow 2 \leftarrow 4 \leftarrow 1 \leftarrow 9 \leftarrow 6$};
    \end{tikzpicture}
    \caption{The oriented cycle $C_p$ around a puncture $p$.}
    \label{fig: oriented cycle around I-puncture}
\end{figure}
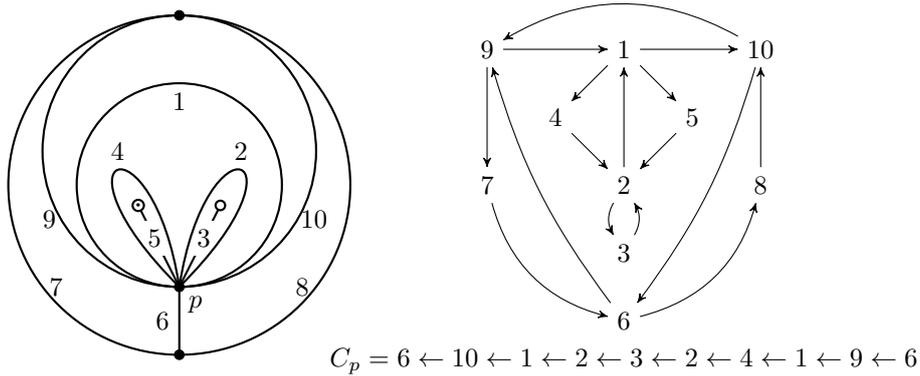

\begin{definition}\label{def: homotopy on Q(T)}
    The homotopy $\widetilde H(T)$ on $\widetilde Q(T)$ is defined to be generated by a finite collection $\widetilde Y(T)$ of cyclic walks consisting of 
    \begin{enumerate}
        \item the oriented 3-cycles in any of the three puzzle pieces of \Cref{fig: puzzle pieces and their quivers},
        \item the oriented cycle $C_p$ around any $\coi$-puncture $p$ of valency at least $2$, and
        \item $C_q^2$ for any $\coii$-puncture $q$.
    \end{enumerate} 
    The homotopy $H(T)$ is defined as the kernel of the functor $\pi(Q(T))\to \pi(\widetilde Q(T))/\widetilde H(T)$ induced by the quiver embedding $Q(T)\hookrightarrow \widetilde Q(T)$. 
\end{definition}

The lemma below shows that going from $(\widetilde{Q}(T), \widetilde{H}(T))$ to $(Q(T), H(T))$ is analogous to the reduction step used in \Cref{subsec: mutations with homotopy}.

\begin{lemma}
    The functor $\pi(Q(T))\to \pi(\widetilde Q(T))/\widetilde H(T)$ in \Cref{def: homotopy on Q(T)} is full. Hence we have
    \begin{equation}\label{eq: def reduced qh triang}
        \pi(Q(T))/H(T) \cong \pi(\widetilde Q(T))/\widetilde H(T).
    \end{equation}
\end{lemma}

\begin{proof}
    Use the notation of \Cref{fig: II-puncture digon} with the puncture being $\coi$-colored. Because of the relations $fe=fba=1$, the deleted arrows $e$ and $f$ are equivalent respectively to $ba$ and $(ba)^{-1}$ in the quotient $\pi(\widetilde Q(T))/\widetilde H(T)$, which are in the image of $\pi(Q(T))$.   
\end{proof}

\begin{remark}\label{rmk: def of Y(T)}
    We specify a set $Y(T)$ of cyclic walks in $Q(T)$ that generates $H(T)$. Use the notation of \Cref{fig: II-puncture digon} with a $\coi$-puncture. Then $Y(T)$ consists of cyclic walks in $\widetilde{Y}(T)$ with any occurrence of $e$ replaced with $ba$ and $f$ with $dc$. Denote the resulting cyclic walks around punctures again by $C_p(T)$. Clearly $Y(T)$ generates $H(T)$.
\end{remark}

\begin{example}\label{ex: quiver with homotopy digon}
    Consider the surface with a triangulation given in \Cref{fig: II-puncture digon}. (It can be a subsurface of a larger one.) If the middle puncture is $\coii$-colored, according to the construction in \Cref{def: quiver with 2-cycle of T} and \Cref{def: homotopy on Q(T)}, the $2$-cycle $fe$ is not deleted from $\widetilde Q(T)$ and thus remains in $Q(T)$, and the $3$-cycles $fba$, $edc$ and the $4$-cycle $(fe)^2$ are in $Y(T)$. If instead the puncture is $\coi$-colored, then $\widetilde H(T)$ contains $fba, dce, fe$. In $Q(T)$, the $2$-cycle $fe$ is deleted. The generators $fba$, $dce$, $fe$ in $\widetilde{Y}(T)$ all give rise to the $4$-cycle $dcba\in Y(T)$.
\end{example}

\begin{figure}[ht]
\centering
\begin{tikzpicture}[scale=0.8]
  \coordinate (A) at (0,1.5); 
  \coordinate (B) at (0,-1.5); 
  \coordinate (C) at (0,0); 
  \coordinate (D) at (-1,0);
  \coordinate (E) at (1,0);

  \draw[thick] (A) to[bend left=60] node[midway, right]{$4$} (B);
  \draw[thick] (A) to[bend right=60] node[midway, left]{$3$} (B);

  \draw[thick, black] (C) -- (A) node[fill=white, midway] {$1$};
  \draw[thick, black] (C) -- (B) node[fill=white, midway] {$2$};

  \filldraw[black] (A) circle (2pt);
  \filldraw[black] (B) circle (2pt);
  \node[generic] at (C) {};


  \node[] (3) at (2, 0) {3};
  \node[] (1) at (3.5, 1.5) {1};
  \node[] (2) at (3.5, -1.5) {2};
  \node[] (4) at (5, 0) {4};

  \path[->, >=stealth']
  (1) edge node[pos=0.5, left] {$a$} (3)
  (3) edge node[pos=0.5, left] {$b$} (2)
  (2) edge node[pos=0.5, right] {$c$} (4)
  (4) edge node[pos=0.5, right] {$d$} (1)
  (1) edge[bend left=15] node[midway, right] {$e$} (2) 
  (2) edge[bend left=15] node[midway, left] {$f$} (1);
\end{tikzpicture}
\caption{A triangulated digon with a puncture and the associated $\widetilde{Q}(T)$.}
\label{fig: II-puncture digon}
\end{figure}

\begin{example}\label{ex: 5 arc surface}
    Consider the surface with a triangulation in \Cref{fig: 5 arc surface}. The arrows $f, g, h$ in $\widetilde{Q}(T)$ arise from the triangle $125$. The remaining quiver comes from a puzzle piece of type $P_2$. The homotopy $\widetilde{H}(T)$ is by definition generated by $\widetilde{Y}(T)=\{cf$, $ghf$, $cba$, $ced\}$. Thus $ba\sim_{\widetilde{H}} ed \sim_{\widetilde{H}} f$ and $gh\sim_{\widetilde{H}} c$. According to \Cref{rmk: def of Y(T)}, $H(T)$ is generated by $Y(T)= \{ghba, ghed\}$.
\end{example}

\begin{figure}[ht]
    \centering
    \begin{tikzpicture}
    \begin{scope}[shift={(0,0)}]
        \node[generic] (B) at (0,0) {};
        \node[coi] (M) at (0,1) {};
        \node[coi] (T) at (0,2) {};
        \node[] at (0,-0.8) {$T$};

        \draw[thick] (B) to node[fill=white, midway]{\small{4}} (M);
        \draw[thick] (B) .. controls (-1,2) and (1,2) .. node[below, midway]{\small{3}} (B);
        \draw[thick, bend left=60] (B) to node[midway, left]{\small{1}} (T);
        \draw[thick, bend right=60] (B) to node[midway, right]{\small{2}} (T);
        \draw[thick] (B) arc[start angle=-90, end angle=270, radius=1.4] node[midway, above]{\small{5}};
    \end{scope}

    \begin{scope}[shift={(3,0)}]
        \node (A1) at (0,1) {1};
                \node (A2) at (2,1) {2};
                \node (A3) at (1,2) {3};
                \node[] (A4) at (1,0) {4};
                \node[] at (1,-0.8) {$\widetilde{Q}(T)$};
                \node (A5) at (1,3) {5};
                \draw[->] (A1) to node[above left]{$a$} (A3);
                \draw[->] (A3) to node[above right]{$b$} (A2);
                \draw[->] (A1) to node[below left]{$d$} (A4);
                \draw[->] (A4) to node[below right]{$e$} (A2);
                \draw[->, bend right] (A2) to node[right]{$h$} (A5);
                \draw[->, bend right] (A5) to node[left]{$g$} (A1);
                \draw[->, bend left=10] (A2) to node[below]{$c$} (A1);
                \draw[->, bend left=10] (A1) to node[above]{$f$} (A2);
    \end{scope}

    \begin{scope}[shift={(7,0)}]
        \node (A1) at (0,1) {1};
                \node (A2) at (2,1) {2};
                \node (A3) at (1,2) {3};
                \node[] (A4) at (1,0) {4};
                \node[] at (1,-0.8) {${Q}(T)$};
                \node (A5) at (1,3) {5};
                \draw[->] (A1) to node[above left]{$a$} (A3);
                \draw[->] (A3) to node[above right]{$b$} (A2);
                \draw[->] (A1) to node[below left]{$d$} (A4);
                \draw[->] (A4) to node[below right]{$e$} (A2);
                \draw[->, bend right] (A2) to node[right]{$h$} (A5);
                \draw[->, bend right] (A5) to node[left]{$g$} (A1);
    \end{scope}
    \end{tikzpicture}
    \caption{A triangulated surface and the associated quivers.}
    \label{fig: 5 arc surface}
\end{figure}

We would like to eventually perform mutations on $(Q(T), H(T))$ so that mutations are compatible with certain \emph{flips} of triangulations, which traditionally replace a diagonal of a quadrilateral with the other diagonal. The main issue is that there is no natural ``flip'' at the arc enclosed in a self-folded triangle. This issue is resolved in \cite{FST08} by introducing \emph{tagged triangulations}.

In our case, the new data of coloring on punctures leads to a modified version of tagged triangulations and of their mutations. We discuss this in \Cref{subsec: tagged triangulation}.

\subsection{\texorpdfstring{Topological realization of $(Q(T), H(T))$}{Topological realization of (Q(T), H(T))}}\label{subsec: topo realization}

To each triangulation $T$ of $(\S, \M, c)$, we associated a $2$-complex $X = X(T)$ such that $\pi_1(X, X_0)$, the fundamental groupoid of $X$ with base points $X_0$ is isomorphic to $\pi(Q(T))/H(T)$. 

We note that in the uncolored case, the complex $X(T)$ (and its close variants) have already appeared in the work of Amiot--Grimeland \cite{AG} and of Amiot--Labardini-Fragoso--Plamondon \cite{ALFP}.

Recall that in \Cref{def: homotopy on Q(T)} we have constructed $\widetilde Q(T)$ with a collection $\widetilde Y(T)$ of oriented cycles to generate $\widetilde H(T)$.

\begin{definition}\label{def: 2-complex X(T)}
    The $2$-complex $X(T)$ is obtained from the quiver $\widetilde Q(T)$ by attaching a $2$-cell along each cyclic walk in $\widetilde{Y}(T)$, as described in \Cref{lemma: 2-complex quiver homotopy}.
\end{definition}

\begin{example}\label{ex: X(T) for third puzzle piece}
    We illustrate the $2$-complex $X(T)$ in \Cref{fig: cell complex third puzzle piece} (where the 2-cells are shaded) for the third puzzle piece in \Cref{fig: puzzle pieces and their quivers}. Since all four $3$-cycles are in $\widetilde Y(T)$, we glue four triangles to $\widetilde Q(T)$.
\end{example}

\begin{figure}
    \centering
    \begin{tikzpicture}    
        \node[] (1) at (0, 0) {$1$};
        \node[] (4) at (-1, -1) {$4$};
        \node[] (2) at (1, -1) {$2$};
        \node[] (3) at (-1, 1) {$3$};
        \node[] (5) at (1, 1) {$5$};

        \path[->, >=stealth']
        (1) edge (4)
        (4) edge (2)
        (2) edge (1)
        (1) edge (5)
        (5) edge (3)
        (3) edge (1)
        (5) edge (2)
        (4) edge (3);

        \filldraw[fill=gray!50, draw=black]
        (3, 0) -- (4, -1) -- (2, -1) -- cycle;
        \filldraw[fill=gray!50, draw=black]
        (3, 0) -- (4, -1) -- (4, 1) -- cycle;
        \filldraw[fill=gray!50, draw=black]
        (3, 0) -- (4, 1) -- (2, 1) -- cycle;
        \filldraw[fill=gray!50, draw=black]
        (3, 0) -- (2, -1) -- (2, 1) -- cycle;
        \filldraw[]
        (3,0) circle (1pt);
        \filldraw[]
        (4,-1) circle (1pt);
        \filldraw[]
        (2,-1) circle (1pt);
        \filldraw[]
        (4,1) circle (1pt);
        \filldraw[]
        (2,1) circle (1pt);
    \end{tikzpicture}
    \caption{The quiver $Q(T)$ and the $2$-complex $X(T)$ for the puzzle piece $P_3$ of \Cref{fig: puzzle pieces and their quivers}.}
    \label{fig: cell complex third puzzle piece}
\end{figure}
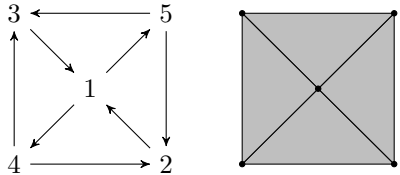

The $1$-skeleton $X(T)_1$ of $X(T)$ is the underlying $1$-complex of $\widetilde Q(T)$. Notice that $\widetilde Q(T)$ only differs with $Q(T)$ by having the dashed 2-cycles as in \Cref{fig: I-puncture digon delete 2-cycle} around $\coi$-punctures.

For each $q\in \punc_{\coii}$, choose a small simple loop $\gamma_q\subset \S\setminus\punc_{\coii}$ around $q$ oriented clockwise. Let $\mathbb{D}_q$ be a closed disk with the attaching map $\varphi_q\colon \partial\mathbb{D}_q\to \gamma_q$ being a degree-two covering. Define
\begin{equation}\label{eq: def cw complex}
    \widehat{\S}_c\coloneqq \left(\S\setminus\punc_{\coii} \right) \cup_{\{\varphi_q\}_q} \bigcup_{q\in \punc_{\coii}} \mathbb{D}_q.
\end{equation}
Note that the image of each $\mathbb{D}_q$ in $\widehat{\S}_c$ is homeomorphic to $\mathbb{RP}^2$.

\begin{proposition}\label{prop: fund grp surface}
    As groupoids we have isomorphisms
    \[
        \pi_1(X(T), X(T)_0) \cong \pi(\widetilde Q(T))/\widetilde H(T) \cong \pi(Q(T))/H(T).
    \]
    The $2$-complex $X(T)$ is homotopy equivalent to $\widehat{\S}_c$. Consequently we have the isomorphism of fundamental groups
    \[
        \pi_1(X(T)) \cong \pi_1(\widehat{\S}_c)\cong \frac{\pi_1(\S\setminus\punc_{\coii})}{\ncl{\gamma_q^2\mid q\in \punc_{\coii}}}.
    \]
\end{proposition}

\begin{proof}
    The first groupoid isomorphism follows from \Cref{lemma: 2-complex quiver homotopy}. The second follows from \eqref{eq: def reduced qh triang}.

    Let $X_{\mathrm{reg}}(T)\subseteqq X(T)$ be the subcomplex obtained by omitting the $2$-cells attached along $C_q^2$ for $q\in \punc_{\coii}$. For the second puzzle piece in \Cref{fig: puzzle pieces and their quivers}, its part in $X_{\mathrm{reg}}(T)$ deformation retracts to the triangle $123$. For the third puzzle piece in \Cref{fig: puzzle pieces and their quivers}, its associated $2$-complex deformation retracts to the triangle $124$. Here a deformation retraction means a homotopy from the identity map to a retraction onto the indicated subcomplex, fixing the subcomplex pointwise throughout.

    After these deformation retractions (which do not affect the parts further glued), we obtain a complex $\widehat{X}_{\mathrm{reg}}(T)$ homotopy equivalent (denoted as $\simeq$) to $X_{\mathrm{reg}}(T)$. Notice that $\widehat{X}_{\mathrm{reg}}(T)$ can then be easily embedded into $\S\setminus \punc_{\coii}$, and $\S\setminus \punc_{\coii}$ deformation retracts to the image of the embedding. Under this embedding, $C_q$ is homotopic to $\gamma_q$.
    
    Now we attach the omitted $2$-cells to $\widehat{X}_{\mathrm{reg}}(T)$ along $C_q^2$, $q\in \punc_{\coii}$, and denote the resulting complex by $\widehat{X}(T)$, which is clearly homotopy equivalent to $X(T)$. Since $\widehat{\S}_c$ is obtained by attaching disks along $\gamma_q^2$ to $\S\setminus\punc_{\coii}$, we have that $\widehat{\S}_c \simeq \widehat{X}(T) \simeq X(T)$. Consequently $\pi_1(X(T)) \cong \pi_1(\widehat{\S}_c)$. The final isomorphism follows from van Kampen's theorem.
\end{proof}

Let $\mathcal O_c$ be the compact oriented $2$-orbifold with underlying surface $\S$, a cone point of order two at each $q\in\punc_{\coii}$, and no other orbifold singularities. The $\coi$-punctures are ordinary points. Its orbifold fundamental group is obtained from $\pi_1(\S\setminus\punc_{\coii})$ by imposing the relations $\gamma_q^2=1$ for all $q\in\punc_{\coii}$; see \cite[Section~13.2]{Thurston22}. Thus \Cref{prop: fund grp surface} identifies
\[
    \pi_1(\widehat{\S}_c)\cong\pi_1^{\mathrm{orb}}(\mathcal O_c).
\]

\begin{corollary}\label{cor: G is hopfian}
    The fundamental group $G = \pi_1(\widehat{\S}_c)$ is Hopfian, meaning that any surjective endomorphism is an automorphism. 
\end{corollary}

\begin{proof}
    If $\S$ has nonempty boundary, $G$ is a free product of a finitely generated free group and copies of $\mathbb Z/2$.

    If $\S$ is closed, $\mathcal O_c$ is a compact oriented $2$-orbifold without boundary. It is in particular \emph{good} in Thurston's terminology, with universal cover $S^2$, $\mathbb{E}^2$ or $\mathbb{H}^2$ \cite[Proposition~13.3.2 and Theorem~13.3.6]{Thurston22}. The orbifold fundamental group $\pi_1^{\mathrm{orb}}(\mathcal O_c)$ is thus finite, Euclidean crystallographic, or Fuchsian, respectively. These groups, as well as the free products in the previous case, are finitely generated and residually finite, hence Hopfian \cite{Magnus69}.
\end{proof}

\begin{example}
    Recall the quiver $\widetilde Q(T) = Q(T)$ in \Cref{fig: oriented cycle around I-puncture}. We draw $\widehat{X}(T)$ in \Cref{fig: cell complexes}. The $2$-cell associated with $C_p$ is shown in a darker shade. The region filled with lines represents the $2$-cell attached along $C_q^2$, where $C_q = 3 \leftarrow 2 \leftarrow 3$. After the gluing, this part is a polygon model of $\mathbb{R}\mathbb{P}^2$. We note that the cell $125$ in $X(T)$ has been retracted to the edge $12$ in $\widehat{X}(T)$.
\end{example}

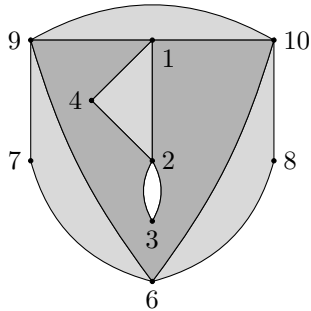
\begin{figure}[ht]
    \centering
    \begin{tikzpicture}[scale=0.8]
      \coordinate (2) at (0, 0);
      \coordinate (1) at (0, 2);
      \coordinate (3) at (0, -1);
      \coordinate (4) at (-1, 1);
      \coordinate (5) at (1, 1);
      \coordinate (6) at (0, -2);
      \coordinate (7) at (-2, 0);
      \coordinate (8) at (2, 0);
      \coordinate (9) at (-2, 2);
      \coordinate (10) at (2, 2);

      \path[fill=gray!30, draw=black]
      (6) to[bend left=10] (9) --
      (7) to[bend right=30] (6)
      -- cycle

      (9) -- (1) -- (10) to[bend right=30] cycle
      
      (10) to[bend left=10] (6) --
      (6) to[bend right=30] (8)
      -- cycle
      
      (1) -- (4) -- (2) -- cycle;
      
      \filldraw[pattern=north east lines, pattern color=black, draw=black]
      (2) to[bend left=45] (3) to[bend left=45] (2);

      \path[fill=gray!60, draw=black]
      (6) to[bend left=10] (9) -- (1) -- (4) --
      (2) to[bend right=45] (3) to[bend right=45] (2) --
      (1) -- (10) to[bend left=10] cycle;

      \filldraw[] (2) circle (1pt);
      \filldraw[] (1) circle (1pt);
      \filldraw[] (3) circle (1pt);
      \filldraw[] (4) circle (1pt);
      \filldraw[] (6) circle (1pt);
      \filldraw[] (7) circle (1pt);
      \filldraw[] (8) circle (1pt);
      \filldraw[] (9) circle (1pt);
      \filldraw[] (10) circle (1pt);

      \node[below right] at (1) {$1$};
      \node[right] at (2) {$2$};
      \node[below] at (3) {$3$};
      \node[left] at (4) {$4$};
      \node[below] at (6) {$6$};
      \node[left] at (7) {$7$};
      \node[right] at (8) {$8$};
      \node[left] at (9) {$9$};
      \node[right] at (10) {$10$};
    \end{tikzpicture}
    \caption{The associated $\widehat{X}(T)$ of \Cref{fig: oriented cycle around I-puncture}.}
    \label{fig: cell complexes}
\end{figure}

\begin{remark}\label{rmk: H(T) is reduced}
    \Cref{prop: fund grp surface} also verifies that $H(T)$ is reduced. Indeed, every oriented $2$-cycle retained in $Q(T)$ is of the form $C_q$ (up to rotations) for some $q\in \punc_{\coii}$. By \Cref{prop: fund grp surface}, its class represents $\gamma_q\in \pi_1(\widehat{\S}_c)$. It is thus nontrivial\footnote{The only exception would be a sphere with $|\punc_{\coii}|=1$, which is excluded in \Cref{def: coloring}.} and has order two. Hence $C_q\notin H(T)$, so $H(T)$ is reduced.
\end{remark}

\subsection{Tagged triangulations}\label{subsec: tagged triangulation}

As already mentioned earlier, the flip operation on a triangulation is not yet defined at the radius of a self-folded triangle (for example the arc $4$ of a $P_2$-piece). This issue affects both $\coi$- and $\coii$-punctures. To resolve it, we introduce \emph{tagged arcs} and \emph{tagged triangulations}, following \cite{FST08}.

\begin{definition}\label{def: tagged triangulation}
    Let $(\S,\M, c)$ be a marked bordered surface with colored punctures.
    \begin{enumerate}
        \item A \emph{tagged arc} is an arc in which each of its two ends is tagged \emph{plain} or \emph{notched} (represented by the $\bowtie$ symbol), so that
        \begin{itemize}
            \item the arc does not cut out a once-punctured monogon enclosing a $\coi$-puncture;
            \item an end connected to the boundary is tagged plain;
            \item both ends of a loop arc are tagged in the same way.
        \end{itemize}
        \item Two tagged arcs $\alpha$ and $\beta$ are called \emph{compatible} if the following conditions are satisfied:
        \begin{itemize}
            \item the untagged versions of $\alpha$ and $\beta$ are compatible;
            \item if the untagged versions of $\alpha$ and $\beta$ are different, and $\alpha$ and $\beta$ share an endpoint $a$, then the ends of $\alpha$ and $\beta$ connected to $a$ must be tagged in the same way;
            \item if the untagged versions of $\alpha$ and $\beta$ coincide, then at least one end of $\alpha$ must be tagged in the same way as the corresponding end of $\beta$;
            \item if the untagged versions of $\alpha$ and $\beta$ coincide, then any end of $\alpha$ connected to a $\coii$-puncture must be tagged in the same way as the corresponding end of $\beta$.
        \end{itemize}
        \item A \emph{tagged triangulation} of $(\S, \M, c)$ is a maximal collection of compatible tagged arcs.
    \end{enumerate}
\end{definition}

\begin{remark}
    Unlike at a $\coi$-puncture, a (tagged) loop cutting out a once-$\coii$-punctured monogon is allowed as a tagged arc. Another difference, compared with the compatibility of tagged arcs in \cite[Definition 7.4]{FST08}, is the last condition in (2); see \Cref{fig: triangulated once-punctured digon} for all tagged triangulations of a once-$\coii$-punctured digon.
\end{remark}

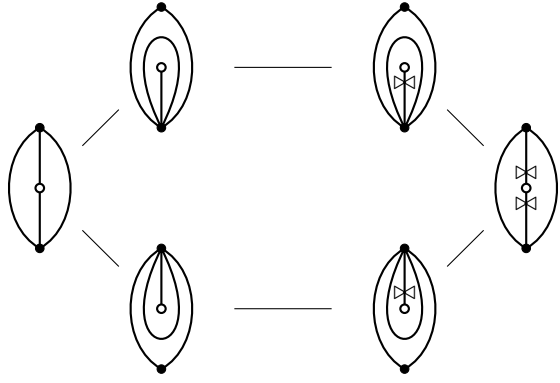
\begin{figure}[ht]
    \centering
    \begin{tikzpicture}[scale=0.8]
        \coordinate (C1) at (-4, 0);
        \coordinate (B1) at (-4, -1);
        \coordinate (A1) at (-4, 1);
        
        \coordinate (C2) at (-2, 2);
        \coordinate (B2) at (-2, 1);
        \coordinate (A2) at (-2, 3);
        
        \coordinate (C3) at (2, 2);
        \coordinate (B3) at (2, 1);
        \coordinate (A3) at (2, 3);
        
        \coordinate (C4) at (4, 0);
        \coordinate (B4) at (4, -1);
        \coordinate (A4) at (4, 1);
        
        \coordinate (C5) at (2, -2);
        \coordinate (B5) at (2, -3);
        \coordinate (A5) at (2, -1);
        
        \coordinate (C6) at (-2, -2);
        \coordinate (B6) at (-2, -3);
        \coordinate (A6) at (-2, -1);

        \filldraw[] (B1) circle (2pt);
        \filldraw[] (B2) circle (2pt);
        \filldraw[] (B3) circle (2pt);
        \filldraw[] (B4) circle (2pt);
        \filldraw[] (B5) circle (2pt);
        \filldraw[] (B6) circle (2pt);

        \filldraw[] (A1) circle (2pt);
        \filldraw[] (A2) circle (2pt);
        \filldraw[] (A3) circle (2pt);
        \filldraw[] (A4) circle (2pt);
        \filldraw[] (A5) circle (2pt);
        \filldraw[] (A6) circle (2pt);

        \path[thick, draw=black]
        (B1) to[bend left=60] (A1)
        (B2) to[bend left=60] (A2)
        (B3) to[bend left=60] (A3)
        (B4) to[bend left=60] (A4)
        (B5) to[bend left=60] (A5)
        (B6) to[bend left=60] (A6)
        (B1) to[bend right=60] (A1)
        (B2) to[bend right=60] (A2)
        (B3) to[bend right=60] (A3)
        (B4) to[bend right=60] (A4)
        (B5) to[bend right=60] (A5)
        (B6) to[bend right=60] (A6)
        (B1) -- (C1) -- (A1)
        (B2) -- (C2)
        (B3) -- (C3)
        (B4) -- (C4) -- (A4)
        (C5) -- (A5)
        (C6) -- (A6);

        \filldraw[fill=white, thick] (C1) circle (2pt);
        \filldraw[fill=white, thick] (C2) circle (2pt);
        \filldraw[fill=white, thick] (C3) circle (2pt);
        \filldraw[fill=white, thick] (C4) circle (2pt);
        \filldraw[fill=white, thick] (C5) circle (2pt);
        \filldraw[fill=white, thick] (C6) circle (2pt);

        \path[thick, draw=black]
        (B2) .. controls (-3,3) and (-1,3) .. (B2)
        (B3) .. controls (1,3) and (3,3) .. (B3)
        (A6) .. controls (-3,-3) and (-1,-3) .. (A6)
        (A5) .. controls (1,-3) and (3,-3) .. (A5);

        \node[below] at (C3) {$\bowtie$};
        \node[above] at (C4) {$\bowtie$};
        \node[below] at (C4) {$\bowtie$};
        \node[above] at (C5) {$\bowtie$};

        \draw[]
        (-3.3, 0.7) -- (-2.7, 1.3)
        (-3.3, -0.7) -- (-2.7, -1.3)
        (3.3, 0.7) -- (2.7, 1.3)
        (3.3, -0.7) -- (2.7, -1.3)
        (-0.8, 2) -- (0.8,2)
        (-0.8, -2) -- (0.8,-2);
    \end{tikzpicture}
    \caption{Tagged triangulations of a once-$\coii$-punctured digon related by flips.}
    \label{fig: triangulated once-punctured digon}
\end{figure}

Let $\tarc(\S, \M, c)$ be the set of all tagged arcs. The \emph{tagged arc complex} $\tagcom(\S, \M, c)$ is the clique complex for the compatibility relation on the set $\tarc(\S, \M, c)$. Namely the simplices of $\tagcom(\S, \M, c)$ are collections of mutually compatible tagged arcs. This notion generalizes $\tagcom(\S, \M)$ of \cite{FST08} to the current situation with colored punctures.

If $C$ is a simplex of $\tagcom(\S, \M, c)$, we extend the construction $C\mapsto C^{\circ}$ in \cite[Section 9.1]{FST08} to our setting, where $C^{\circ}$ is a collection of compatible (untagged) arcs with the same cardinality as $C$. Namely $C^{\circ}$ is obtained from $C$ by
\begin{itemize}
    \item replacing all notched tags at $\coii$-punctures by plain ones;
    \item replacing all notched tags at any $\coi$-puncture by plain ones if there are no plain tags at that puncture;
    \item replacing any tagged arc $\beta$ notched at a $\coi$-puncture $b$ by a loop $\gamma$ (based at the other endpoint of $\beta$) enclosing $b$ and $\beta$ if there is another tagged arc $\alpha$ that is plain tagged at $b$ and has the same underlying ordinary arc as $\beta$.
\end{itemize}

The following theorem generalizes \cite[Theorem 7.9]{FST08} straightforwardly.

\begin{theorem}\label{thm: tagcom pseudomfld}
    Each maximal simplex in $\tagcom(\S, \M, c)$ (i.e., a tagged triangulation) has the same cardinality $n=6g+3b+3p+s-6$ tagged arcs, where $g$ is the genus of $\S$, $b$ is the number of boundary components, $p = |\punc|$ is the number of punctures and $s = |\M\setminus\punc|$ is the number of marked points on the boundary. Each simplex of codimension $1$ is contained in precisely two maximal simplices. In other words, the tagged arc complex $\tagcom(\S, \M, c)$ is a pseudomanifold.
\end{theorem}

\begin{proof}
    Suppose first that $(\S, \M)$ is not a thrice-punctured sphere. Let $C$ be a maximal simplex and we want to show that $C^{\circ}$ is maximal. As in \cite[Proof of Thm 7.9]{FST08}, an ordinary arc in addition to and compatible with $C^{\circ}$ can be tagged so as to agree with the existing tags on its endpoints; $\coii$-punctures do not obstruct this operation. So $C$ is maximal if and only if $C^\circ$ is maximal, in which case the count $n = |C^\circ|$ is exactly the same as in \cite{FST08}.

    If $|C| = n-1$, we want to show there are precisely two ways to complete $C$. There are two cases according to whether $C^{\circ}$ is on the boundary of the ordinary arc complex $\Delta^{\circ}(\S, \M)$ (see \cite[Section 3]{FST08}). If it is not, then its two ordinary completions have unique compatible taggings, the same as in \cite{FST08}. If $C^{\circ}$ is on the boundary, then it contains a loop enclosing a single $\coii$-puncture but not the radius inside. Adding the radius gives the unique ordinary completion, and the two possible tags at the $\coii$-puncture give precisely two tagged completions.

    Finally if $(\S, \M)$ is a thrice-punctured sphere, then $\punc_{\coii} = \punc = \M$. An ordinary triangulation is either the three arcs joining the three pairs of punctures, or a loop based at one puncture together with two radii. Each has three arcs. A direct check shows that every compatible pair of tagged arcs has exactly two tagged completions.
\end{proof}

Edges of the dual graph $\mathbf{E}^{\bowtie}(\S, \M, c)$ (i.e., the \emph{flip graph}) of $\tagcom(\S, \M, c)$ are called \emph{flips} (of tagged triangulations). In other words, a flip replaces a tagged arc in a tagged triangulation by a unique new tagged arc such that the new collection forms a different tagged triangulation. \Cref{thm: tagcom pseudomfld} guarantees that a tagged triangulation can be flipped at any tagged arc, that is, $\mathbf{E}^{\bowtie}(\S, \M, c)$ is $n$-regular. See \Cref{fig: triangulated once-punctured digon} for the flip graph of a once-$\coii$-punctured digon. As a comparison, the flip graph of a once-$\coi$-punctured digon is a square \cite[Figure 20]{FST08}.

For any tagged triangulation $T$, we define $(Q(T), H(T)) \coloneqq (Q(T^\circ), H(T^\circ))$ where the latter is defined in \Cref{def: quiver with 2-cycle of T,def: homotopy on Q(T)}, with vertex labels inherited from $T$. According to \Cref{rmk: H(T) is reduced}, the pair $(Q(T), H(T))$ is reduced, so we are able to apply mutations to it. The following theorem shows that mutations of $(Q(T), H(T))$ are compatible with flips of tagged triangulations.

\begin{theorem}\label{thm: mutation and flip}
    Let $T = \{1, \dots, n\}$ be a tagged triangulation of $(\S, \M, c)$ and $T'$ be the flip of $T$ at a tagged arc $k\in T$, with the new arc retaining the label $k$. Then $(Q(T'),H(T'))=\mu_k(Q(T),H(T))$.
\end{theorem}

\begin{proof}
    Write $(Q,H)=(Q(T),H(T))$ and $(Q',H')=\widetilde\mu_k(Q,H)$. In each case in \textbf{Step II} below, let $Q^\dag$ be the quiver obtained by the specified deletions (labeled as reduction in each diagram) of $2$-cycles in $H'$, and put $H^\dag=\ker(\varphi|_{\pi(Q^\dag)})$, with $\varphi$ as in \Cref{def: pre-mutation}. 
    
    The proof of \Cref{lemma: surjective grp map} applies before knowing these reductions are maximal. Indeed, each deleted arrow is equivalent to a walk through $k$ using reversed arrows, which survive. Thus $\varphi|_{\pi(Q^\dag)}$ is full and we have the isomorphism
    \begin{equation}\label{eq: iso after reduction}
        \pi(Q^\dag)/H^\dag\cong\pi(Q)/H.
    \end{equation}

    \textbf{Step I:} reduction to an inclusion. We construct in \textbf{Step II} a quiver isomorphism $f:Q(T')\to Q^\dag$ that fixes every vertex and satisfies the inclusion
    \begin{equation}\label{eq: the inclusion}
        f_*(Y(T'))\subseteq H^\dag.
    \end{equation}
    Since $Y(T')$ generates $H(T')$, this inclusion induces a full functor $\overline f:\pi(Q(T'))/H(T')\to \pi(Q^\dag)/H^\dag$.
    Post-composing with \eqref{eq: iso after reduction} yields a full functor
    \begin{equation}\label{eq: a full functor}
        \pi(Q(T'))/H(T') \longrightarrow \pi(Q)/H.
    \end{equation}
    By \Cref{prop: fund grp surface}, the vertex groups on both sides of \eqref{eq: a full functor} are isomorphic to the fundamental group $\pi_1(\widehat{\S}_c)$.

    The restrictions of \eqref{eq: a full functor} to vertex groups are surjective maps between isomorphic Hopfian groups by \Cref{cor: G is hopfian}, hence injective. Thus $f_*(H(T'))=H^\dag$. Since $H(T')$ is reduced, $H^\dag$ contains no oriented $2$-cycle in $Q^\dag$. Hence the specified deletions are maximal and $(Q^\dag,H^\dag)=\mu_k(Q,H)$. Therefore $f$ identifies the two quivers with homotopies in the statement.
    
    \textbf{Step II:} construction of $f$ and checking the inclusion \eqref{eq: the inclusion}. Let
    \[
        F = \varphi\vert_{\pi(Q^\dag)}:\pi(Q^\dag)\longrightarrow\pi(Q)/H.
    \]
    Since $\ker F=H^\dag$, it suffices to check that each generator in $Y(T')$ has trivial image under $F$, after the identification $Q(T')\overset{f}{=}Q^\dag$.

    All identities below are in $\pi(Q)/H$. We repeatedly use
    \[
        F(\alpha^\star)=\alpha^{-1},\qquad
        F([\beta\alpha])=\beta\alpha,\qquad
        F(\alpha^\star\beta^\star[\beta\alpha])=1.
    \]
    We call the last $3$-cycle a \emph{mutation triangle}.

    First suppose $k$ is the radius of a $P_4$-piece (\Cref{fig: self-folded II puncture}) enclosing the $\coii$-puncture $q$, with arrows $a:j\to k$ and $b:k\to j$. The flip changes only the tag at $q$, so $Q(T)$ and $Q(T')$ are the exact same quiver. Set $f(a)=b^\star$, $f(b)=a^\star$, fixing all other arrows. Then $Ff_*$ sends $a$ to $b^{-1}$ and $b$ to $a^{-1}$, so the generator $C_q(T')^2=(ab)^2$ maps to $(ab)^{-2}=1$. In any other puncture cycle (\Cref{rmk: def of Y(T)}), a visit to $k$ uses $ba$ up to cyclic rotation. Its image under $Ff_*$ is $(ba)^{-1}=ba$ in $\pi(Q)/H$. Thus the remaining puncture generators also have trivial image. The generators in $Y(T')$ arising from triangle cycles avoid $k$ and are unchanged.

    For the remaining kinds of flips, we first justify passing to flips of untagged triangulations. Simultaneously reversing all tags at a $\coi$-puncture in $T$ and $T'$ preserves compatibility and flips. If only one tag occurs, the associated untagged triangulation is unchanged.

    If both tags occur, the reversal interchanges the radius and enclosing loop in $T^\circ$. Their quiver vertices have identical incidence in $P_2$ or $P_3$, and the triangle relations identify the two routes through them. Thus this interchange preserves the homotopy as well, including the puncture generators.

    At each $\coi$-puncture where the common collection $T\cap T'$ has only one tag, reverse tags if necessary to make that tag plain. The construction of $T^\circ$ then turns the tagged flip into an ordinary flip at a non-enclosed arc, as in \cite[Section 9.2]{FST08}. We can therefore compare $T^\circ$ and $(T')^\circ$ by a flip either inside one puzzle piece (\emph{Case 1}) or at a shared arc of two puzzle pieces (\emph{Case 2}).

    The diagrams below identify $Q(T')$ with $Q^\dag$, thus giving $f$. Arrows outside the displayed pieces are fixed, except for the arrows restored at punctures of valency two as described below. Generators disjoint from the changed region are therefore unchanged. For the puncture generators meeting it, we give the local path substitutions below; the rest of each puncture generator is unchanged.

    Each pair $\gamma[\beta\alpha]$ deleted in the diagrams satisfies $\varphi(\gamma[\beta\alpha])=\gamma\beta\alpha=1$ by an old triangle relation. Whenever $C_p(T)$ and $C_p(T')$ are defined, showing that $F(C_p(T'))$ is conjugate to $C_p(T)^{\pm1}$ checks the puncture generator for either coloring.

    \emph{Case 1.} The flip occurs inside a puzzle piece. There are two subcases \emph{Case 1.1} and \emph{Case 1.2}.

    \emph{Case 1.1.} The flip is at arc $3$ in the puzzle piece $P_2$ of \Cref{fig: puzzle pieces and their quivers}. There are three situations depicted in \Cref{fig: case 1.1}. In \emph{Case 1.1.1}, the marked point $p$ lies on the boundary or is a puncture of valency at least three.

    \begin{figure}[ht]
        \centering
        \begin{tikzpicture}
            \begin{scope}[shift={(0,0)}]
                \node[generic, label=below:$q$] (B) at (0,0) {};
                \node[] at (0,-0.7) {\emph{Case 1.1.1}};
                \node[coi] (M) at (0,1) {};
                \node[generic, label=above:\small{$\,\cdots$}, label=left:$p$] (T) at (0,2) {};

                \draw[thick] (B) to node[fill=white, midway]{\small{4}} (M);
                \draw[thick] (B) .. controls (-1,2) and (1,2) .. node[below, midway]{\small{3}} (B);
                \draw[thick, bend left=60] (B) to node[midway, left]{\small{1}} (T);
                \draw[thick, bend right=60] (B) to node[midway, right]{\small{2}} (T);
                \draw[thick] (T) to (-0.5,2.5);
                \draw[thick] (T) to (0.5,2.5);
            \end{scope}

            \begin{scope}[shift={(4,0)}]
                \node[generic, label=below:$q$] (B) at (0,0) {};
                \node[] at (0,-0.7) {\emph{Case 1.1.2}};
                \node[coi] (M) at (0,1) {};
                \node[coi, label=above right:$p$] (T) at (0,2) {};

                \draw[thick] (B) to node[fill=white, midway]{\small{4}} (M);
                \draw[thick] (B) .. controls (-1,2) and (1,2) .. node[below, midway]{\small{3}} (B);
                \draw[thick, bend left=60] (B) to node[midway, left]{\small{1}} (T);
                \draw[thick, bend right=60] (B) to node[midway, right]{\small{2}} (T);
                \draw[thick] (B) arc[start angle=-90, end angle=270, radius=1.4] node[midway, above]{\small{5}};
            \end{scope}

            \begin{scope}[shift={(8,0)}]
                \node[generic, label=below:$q$] (B) at (0,0) {};
                \node[] at (0,-0.7) {\emph{Case 1.1.3}};
                \node[coi] (M) at (0,1) {};
                \node[coii, label=above right:$p$] (T) at (0,2) {};

                \draw[thick] (B) to node[fill=white, midway]{\small{4}} (M);
                \draw[thick] (B) .. controls (-1,2) and (1,2) .. node[below, midway]{\small{3}} (B);
                \draw[thick, bend left=60] (B) to node[midway, left]{\small{1}} (T);
                \draw[thick, bend right=60] (B) to node[midway, right]{\small{2}} (T);
                \draw[thick] (B) arc[start angle=-90, end angle=270, radius=1.4] node[midway, above]{\small{5}};
            \end{scope}
        \end{tikzpicture}
        \caption{Three triangulations of \emph{Case 1.1}.}
        \label{fig: case 1.1}
    \end{figure}
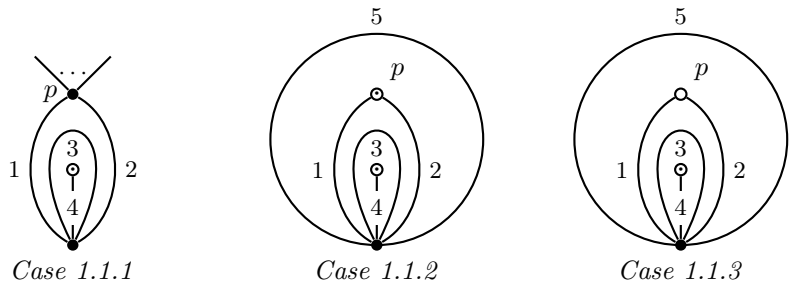

    \emph{Case 1.1.1.}
    \begin{equation*}\label{eq: 1.1.1}
        \begin{tikzpicture}
            \begin{scope}[shift={(0,0)}]
                \node (A1) at (0,1) {1};
                \node (A2) at (2,1) {2};
                \node (A3) at (1,2) {3};
                \node[label=below:{$Q(T)$}] (A4) at (1,0) {4};
                \draw[->] (A1) to node[above left]{$a$} (A3);
                \draw[->] (A3) to node[above right]{$b$} (A2);
                \draw[->] (A2) to node[above]{$c$} (A1);
                \draw[->] (A1) to node[below left]{$d$} (A4);
                \draw[->] (A4) to node[below right]{$e$} (A2);
            \end{scope}

            \draw[->] (3,1) -- node[above]{$\widetilde{\mu}_3$} (4,1);
            
            \begin{scope}[shift={(5,0)}]
                \node (A1) at (0,1) {1};
                \node (A2) at (2,1) {2};
                \node (A3) at (1,2) {3};
                \node[label=below:{$Q'$}] (A4) at (1,0) {4};
                \draw[->] (A3) to node[above left]{$a^\star$} (A1);
                \draw[->] (A2) to node[above right]{$b^\star$} (A3);
                \draw[->, bend left=10] (A2) to node[below]{$c$}(A1);
                \draw[->, bend left=10] (A1) to node[above]{\small{${[ba]}$}} (A2);
                \draw[->] (A1) to node[below left]{$d$} (A4);
                \draw[->] (A4) to node[below right]{$e$} (A2);
            \end{scope}

            \draw[->] (8,1) -- node[above]{\small{reduction}} (9,1);

            \begin{scope}[shift={(10,0)}]
                \node (A1) at (0,1) {1};
                \node (A2) at (2,1) {2};
                \node (A3) at (1,2) {3};
                \node[label=below:{$Q^\dag = Q(T')$}] (A4) at (1,0) {4};
                \draw[->] (A3) to node[above left]{$a^\star$} (A1);
                \draw[->] (A2) to node[above right]{$b^\star$} (A3);
                \draw[->] (A1) to node[below left]{$d$} (A4);
                \draw[->] (A4) to node[below right]{$e$} (A2);
            \end{scope}
        \end{tikzpicture}
    \end{equation*}
    The triangle relations give $cba=ced=1$ in $\pi(Q)/H$.
Hence the reduction deletes $c[ba]$, and
\[
 F(eda^\star b^\star)=ed(ba)^{-1}=1.
\]
If the upper marked point $p$ is a puncture, its old segment $c$ in $C_p(T)$ is
replaced by $a^\star b^\star$; if the lower marked point $q$ is a puncture, its
old segment $ba$ is replaced by $ed$. The identities
\[
 F(a^\star b^\star)=(ba)^{-1}=c,
 \qquad F(ed)=ed=ba
\]
therefore carry each affected puncture cycle to its old class.

    \emph{Case 1.1.2.}
    \begin{equation*}\label{eq: 1.1.2}
        \begin{tikzpicture}
            \begin{scope}[shift={(0,0)}]
                \node (A1) at (0,1) {1};
                \node (A2) at (2,1) {2};
                \node (A3) at (1,2) {3};
                \node[label=below:$Q(T)$] (A4) at (1,0) {4};
                \node (A5) at (1,3) {5};
                \draw[->] (A1) to node[below right]{$a$} (A3);
                \draw[->] (A3) to node[below left]{$b$} (A2);
                \draw[->] (A1) to node[below left]{$d$} (A4);
                \draw[->] (A4) to node[below right]{$e$} (A2);
                \draw[->] (A2) to node[right]{$h$} (A5);
                \draw[->] (A5) to node[left]{$g$} (A1);
            \end{scope}

            \draw[->] (3,1) -- node[above]{$\widetilde{\mu}_3=\mu_3$} (4,1);
            
            \begin{scope}[shift={(5,0)}]
                \node (A1) at (0,1) {1};
                \node (A2) at (2,1) {2};
                \node (A3) at (1,2) {3};
                \node[label=below:{$Q'=Q^\dag=Q(T')$}] (A4) at (1,0) {4};
                \node (A5) at (1,3) {5};
                \draw[->] (A3) to node[right]{$a^\star$} (A1);
                \draw[->] (A2) to node[left, pos=0.3]{$b^\star$} (A3);
                \draw[->] (A1) to node[below]{$[ba]$} (A2);
                \draw[->] (A1) to node[below left]{$d$} (A4);
                \draw[->] (A4) to node[below right]{$e$} (A2);
                \draw[->] (A2) to node[right]{$h$} (A5);
                \draw[->] (A5) to node[left]{$g$} (A1);
            \end{scope}
        \end{tikzpicture}
    \end{equation*}
    By \Cref{ex: 5 arc surface}, $ba=ed=(gh)^{-1}$ in $\pi(Q)/H$.
We obtain
\[
 F(a^\star b^\star ed)=1,\qquad F([ba]gh)=1,
 \qquad F(a^\star b^\star [ba])=1.
\]
The first two are the new $4$-cycle and triangle generators;
the third is the cycle around the $\coi$-puncture $p$.
If the lower marked point $q$ is a puncture, the old portion $hbag$ of $C_q(T)$
is replaced by $hedg$, and $F(hedg)=hbag$.
No other puncture cycles change.

    \emph{Case 1.1.3.}
    \begin{equation*}\label{eq: 1.1.3}
        \begin{tikzpicture}
            \begin{scope}[shift={(0,0)}]
                \node (A1) at (0,1) {1};
                \node (A2) at (2,1) {2};
                \node (A3) at (1,2) {3};
                \node[label=below:$Q(T)$] (A4) at (1,0) {4};
                \node (A5) at (1,3) {5};
                \draw[->] (A1) to node[right]{$a$} (A3);
                \draw[->] (A3) to node[left]{$b$} (A2);
                \draw[->] (A1) to node[below left]{$d$} (A4);
                \draw[->] (A4) to node[below right]{$e$} (A2);
                \draw[->] (A2) to node[right]{$h$} (A5);
                \draw[->] (A5) to node[left]{$g$} (A1);
                \draw[->, bend left=10] (A1) to node[above]{$c_1$} (A2);
                \draw[->, bend left=10] (A2) to node[below]{$c_2$} (A1);
            \end{scope}

            \draw[->] (3,1) -- node[above]{$\widetilde{\mu}_3$} (4,1);
            
            \begin{scope}[shift={(5,0)}]
                \node (A1) at (0,1) {1};
                \node (A2) at (2,1) {2};
                \node (A3) at (1,2) {3};
                \node (A4)[label=below:$Q'$] at (1,0) {4};
                \node (A5) at (1,3) {5};
                \draw[->] (A3) to node[right]{$a^\star$} (A1);
                \draw[->] (A2) to node[left, pos=0.3]{$b^\star$} (A3);
                \draw[->] (A1) to (A2);
                \draw[->, bend left=15] (A1) to (A2);
                \draw[->, bend left=15] (A2) to (A1);
                \draw[->] (A1) to node[below left]{$d$} (A4);
                \draw[->] (A4) to node[below right]{$e$} (A2);
                \draw[->] (A2) to node[right]{$h$} (A5);
                \draw[->] (A5) to node[left]{$g$} (A1);
            \end{scope}

            \draw[->] (8,1) -- node[above]{\small{reduction}} (9,1);
            
            \begin{scope}[shift={(10,0)}]
                \node (A1) at (0,1) {1};
                \node (A2) at (2,1) {2};
                \node (A3) at (1,2) {3};
                \node[label=below:{$Q^\dag=Q(T')$}] (A4) at (1,0) {4};
                \node (A5) at (1,3) {5};
                \draw[->] (A3) to node[right]{$a^\star$} (A1);
                \draw[->] (A2) to node[left, pos=0.3]{$b^\star$} (A3);
                \draw[->] (A1) to node[below]{$c_1$} (A2);
                \draw[->] (A1) to node[below left]{$d$} (A4);
                \draw[->] (A4) to node[below right]{$e$} (A2);
                \draw[->] (A2) to node[right]{$h$} (A5);
                \draw[->] (A5) to node[left]{$g$} (A1);
            \end{scope}
        \end{tikzpicture}
    \end{equation*}
    The relations $c_2ba=c_2ed=1$ allow the deletion of $c_2[ba]$
and give
\[
 F(eda^\star b^\star)=ed(ba)^{-1}=1.
\]
For the $\coii$-puncture $p$, choose $C_p(T)=c_2c_1$ and
$C_p(T')=a^\star b^\star c_1$. Then
\[
 F\bigl((C_p(T'))^2\bigr)
   =\bigl((ba)^{-1}c_1\bigr)^2=(c_2c_1)^2=1.
\]
The triangle generator $c_1gh$ is unchanged. At the lower
marked point $q$, the portion $hbag$ is replaced by $hedg$,
and $F(hedg)=hbag$ since $ed=ba$. All other puncture cycles are unchanged.

    \emph{Case 1.2.} The flip is at arc $2$ or arc $4$ in the puzzle piece $P_3$ of \Cref{fig: puzzle pieces and their quivers}. We demonstrate the flip at $4$; the other case is similar. The mutation is as follows.
    \[
        \begin{tikzpicture}
            \begin{scope}[shift={(0,0)}]
            \node (A1) at (0,0) {1};
            \node (A2) at (1,-1) {2};
            \node (A3) at (-1,1) {3};
            \node (A4) at (-1,-1) {4};
            \node (A5) at (1,1) {5};
            \draw[->] (A1) -- node[above left]{$a$} (A4);
            \draw[->] (A1) -- node[above left]{$b$} (A5);
            \draw[->] (A2) -- node[above right]{$c$} (A1);
            \draw[->] (A3) -- node[above right]{$d$} (A1);
            \draw[->] (A4) -- node[left]{$e$} (A3);
            \draw[->] (A4) to node[above]{$f$} node[below]{$Q(T)$} (A2);
            \draw[->] (A5) -- node[above]{$g$} (A3);
            \draw[->] (A5) -- node[right]{$h$} (A2);
            \end{scope}

            \draw[->] (2,0) -- node[above]{${\widetilde\mu}_4$} (3,0);

            \begin{scope}[shift={(5,0)}]
            \node (A1) at (0,0) {1};
            \node (A2) at (1,-1) {2};
            \node (A3) at (-1,1) {3};
            \node (A4) at (-1,-1) {4};
            \node (A5) at (1,1) {5};
            \draw[->] (A4) -- (A1);
            \draw[->] (A1) -- (A5);
            \draw[->, bend left] (A2) to (A1);
            \draw[->, bend left] (A3) to (A1);
            \draw[->] (A3) -- (A4);
            \draw[->] (A2) to node[below]{$Q'$} (A4);
            \draw[->] (A5) -- (A3);
            \draw[->] (A5) -- (A2);
            \draw[->, bend left=10] (A1) to (A3);
            \draw[->, bend left=10] (A1) to (A2);
            \end{scope}

            \draw[->] (7,0) -- node[above]{\small{reduction}} (8,0);

            \begin{scope}[shift={(10,0)}]
            \node (A1) at (0,0) {1};
            \node (A2) at (1,-1) {2};
            \node (A3) at (-1,1) {3};
            \node (A4) at (-1,-1) {4};
            \node (A5) at (1,1) {5};
            \draw[->] (A4) -- node[above]{$a^\star$} (A1);
            \draw[->] (A1) -- node[above left]{$b$} (A5);
            \draw[->] (A3) -- node[left]{$e^\star$} (A4);
            \draw[->] (A2) to node[above]{$f^\star$} node[below]{$Q^\dag = Q(T')$} (A4);
            \draw[->] (A5) -- node[above]{$g$} (A3);
            \draw[->] (A5) -- node[right]{$h$} (A2);
            \end{scope}
        \end{tikzpicture}
    \]
    The triangle relations $dea=cfa=dgb=chb=1$ give
\[
 F(a^\star e^\star gb)=dgb=1,
 \qquad F(a^\star f^\star hb)=chb=1.
\]
Thus both new $4$-cycle generators belong to $H^\dag$.
If the common base point $q$ of the loops in the $P_3$-piece is
a puncture, the old portion of $C_q(T)$ inside the piece is $cfa$; the new
portion is $a^\star f^\star hb$. The old portion is trivial modulo $H$,
and the new one has trivial image under $F$. Therefore $(C_q(T'))^2$
maps to the class of $C_q(T)^2$ if $q$ is $\coii$-colored;
the same calculation checks $C_q(T')$ if it is $\coi$-colored.
The $\coi$-puncture formerly enclosed by arc $4$ now has valency two; its generator is $a^\star f^\star hb$, already checked above. The other enclosed puncture remains enclosed, and all remaining puncture cycles are unchanged.

    \emph{Case 2.} The flip is at a shared arc of two puzzle pieces. We consider the ten unordered pairs $(P_i,P_j)$ with $1\leq i\leq j\leq4$.

    If gluing deletes an arrow in forming $Q(T)$ or $Q(T')$, use its replacement from \Cref{rmk: def of Y(T)}. When marked points of a patch are identified, apply all the corresponding local substitutions in the same puncture cycle.

    A puncture of valency two opposite the flipped arc requires the modification in \Cref{fig: valency two opposite flip}, using the labels of \emph{Case 2.1}.

    \begin{figure}[ht]
        \centering
        \begin{tikzpicture}[scale=0.9]
            \node at (0,1.5) {$Q(T)$};
            \node at (6,1.5) {$Q(T')$};
            \node at (-2.1,0) {$\coi$};
            \node at (-2.1,-3.1) {$\coii$};

            \begin{scope}
                \node (A1) at (0,0.9) {$1$};
                \node (A3) at (-1.25,-0.75) {$3$};
                \node (A4) at (1.25,-0.75) {$4$};
                \draw[->] (A3) -- node[above left] {$d$} (A1);
                \draw[->] (A1) -- node[above right] {$a$} (A4);
            \end{scope}
            \draw[->] (1.8,0) -- node[above] {$\mu_1$}
                node[below] {\small retain $[ad]$} (4.2,0);
            \begin{scope}[shift={(6,0)}]
                \node (B1) at (0,0.9) {$1$};
                \node (B3) at (-1.25,-0.75) {$3$};
                \node (B4) at (1.25,-0.75) {$4$};
                \draw[->] (B1) -- node[above left] {$d^\star$} (B3);
                \draw[->] (B4) -- node[above right] {$a^\star$} (B1);
                \draw[->] (B3) -- node[below] {$[ad]$} (B4);
            \end{scope}

            \begin{scope}[shift={(0,-3.1)}]
                \node (C1) at (0,0.9) {$1$};
                \node (C3) at (-1.25,-0.75) {$3$};
                \node (C4) at (1.25,-0.75) {$4$};
                \draw[->] (C3) -- node[above left] {$d$} (C1);
                \draw[->] (C1) -- node[above right] {$a$} (C4);
                \draw[->] (C4) to[bend left=18] node[below] {$e$} (C3);
                \draw[->] (C3) to[bend left=18] (C4);
            \end{scope}
            \draw[->] (1.8,-3.1) -- node[above] {$\mu_1$}
                node[below] {\small delete $e[ad]$} (4.2,-3.1);
            \begin{scope}[shift={(6,-3.1)}]
                \node (D1) at (0,0.9) {$1$};
                \node (D3) at (-1.25,-0.75) {$3$};
                \node (D4) at (1.25,-0.75) {$4$};
                \draw[->] (D1) -- node[above left] {$d^\star$} (D3);
                \draw[->] (D4) -- node[above right] {$a^\star$} (D1);
                \draw[->] (D3) -- (D4);
            \end{scope}
        \end{tikzpicture}
        \caption{Mutation at $1$ when the opposite puncture has valency two. Only the relevant arrows are shown. The unlabeled arrow belongs to the adjacent piece.}
        \label{fig: valency two opposite flip}
    \end{figure}

    If the puncture is $\coi$-colored, $e$ and the arrow $3\to4$ from the adjacent piece are absent from $Q(T)$. The retained arrow $[ad]$ represents the restored arrow $3\to4$ in $Q(T')$. Since $F([ad])=ad$ is its prescribed replacement in $Q(T)$, generators involving this arrow in the adjacent piece map to their old representatives. The new puncture generator satisfies
    \[
        F(d^\star a^\star[ad])=(ad)^{-1}ad=1.
    \]
    If the puncture is $\coii$-colored, delete $e[ad]$ using $ead=1$ and keep the arrow from the adjacent piece. Since $F(d^\star a^\star)=e$, the new puncture cycle maps to the old one, so its square has trivial image. The same argument applies at $4\to3$ and $5\to2$ in \emph{Case 2.1}, $5\to6$ in \emph{Case 2.2}, and $7\to6$ in \emph{Case 2.3}.

    \emph{Case 2.1.} $(P_1, P_1)$. We analyze by how many arcs are shared by the two copies of $P_1$. Suppose that they share only one arc. The mutation depends on the valencies of the four marked points and their colors. We demonstrate below the case where each marked point lies on the boundary or is a puncture of valency at least three. At a puncture of valency two opposite the flipped arc, apply \Cref{fig: valency two opposite flip} and the accompanying argument. The mutation is as follows.
    \[
        \begin{tikzpicture}
            \begin{scope}[shift={(0,0)}]
            \node (A1) at (0,0) {1};
            \node (A2) at (1,-1) {2};
            \node (A3) at (-1,1) {3};
            \node (A4) at (-1,-1) {4};
            \node (A5) at (1,1) {5};
            \draw[->] (A1) -- node[above left]{$a$} (A4);
            \draw[->] (A1) -- node[above left]{$b$} (A5);
            \draw[->] (A2) -- node[above right]{$c$} (A1);
            \draw[->] (A3) -- node[above right]{$d$} (A1);
            \draw[->] (A4) -- node[left]{$e$} (A3);
            \draw[->] (A5) -- node[right]{$f$} (A2);

            \node at (0,-1.7) {$Q(T)$};
            \end{scope}

            \draw[->] (1.8,0) -- node[above]{${\widetilde\mu}_1$} (2.8,0);

            \begin{scope}[shift={(5,0)}]
            \node (A1) at (0,0) {1};
            \node (A2) at (1,-1) {2};
            \node (A3) at (-1,1) {3};
            \node (A4) at (-1,-1) {4};
            \node (A5) at (1,1) {5};
            \draw[->] (A4) -- (A1);
            \draw[->] (A5) -- (A1);
            \draw[->, bend right=15] (A3) to node[midway,left,font=\small,inner sep=1.2pt]{$[ad]$} (A4);
            \draw[->, bend right=15] (A4) to (A3);
            \draw[->] (A2) -- (A4);
            \draw[->] (A3) -- (A5);
            \draw[->, bend right=15] (A5) to (A2);
            \draw[->, bend right=15] (A2) to node[midway,right,font=\small,inner sep=1.2pt]{$[bc]$} (A5);
            \draw[->] (A1) -- (A2);
            \draw[->] (A1) -- (A3);

            \node at (0,-1.7) {$Q'$};
            \end{scope}

            \draw[->] (7.2,0) -- node[above]{\small{reduction}} (8.2,0);

            \begin{scope}[shift={(10,0)}]
            \node (A1) at (0,0) {1};
            \node (A2) at (1,-1) {2};
            \node (A3) at (-1,1) {3};
            \node (A4) at (-1,-1) {4};
            \node (A5) at (1,1) {5};
            \draw[->] (A4) -- node[midway,above left,inner sep=1.2pt]{$a^\star$} (A1);
            \draw[->] (A5) -- node[midway,below right,inner sep=1.2pt]{$b^\star$} (A1);
            \draw[->] (A2) -- node[midway,below,inner sep=2pt]{$[ac]$} (A4);
            \draw[->] (A3) -- node[midway,above,inner sep=2pt]{$[bd]$} (A5);
            \draw[->] (A1) -- node[midway,above right,inner sep=1.2pt]{$c^\star$} (A2);
            \draw[->] (A1) -- node[midway,below left,inner sep=1.2pt]{$d^\star$} (A3);

            \node at (0,-1.7) {$Q^\dag=Q(T')$};
            \end{scope}
        \end{tikzpicture}
    \]
    The old triangle relations $dea=cfb=1$ allow us to delete $e[ad]$ and $f[bc]$. The two new triangle generators are $c^\star a^\star[ac]$ and $d^\star b^\star[bd]$, whose images under $F$ are trivial.

    At the four marked points of the quadrilateral, the changed portions of the puncture cycles satisfy
    \begin{equation}\label{eq: four identities in case 2.1}
    \begin{aligned}
        F(d^\star a^\star)&=(ad)^{-1}=e,
        & F(c^\star b^\star)&=(bc)^{-1}=f,\\
        F([ac])&=ac,
        & F([bd])&=bd.
    \end{aligned}
    \end{equation}
    Adjoining the unchanged portions shows that $F(C_p(T'))$ is conjugate to $C_p(T)$ for every affected puncture $p$.

    A puncture of valency two at an endpoint of the flipped arc occurs in the configuration with two shared arcs considered below.

    If exactly two arcs are shared, the union can be an annulus. Its quiver is obtained from the preceding diagram by identifying $2$ with $3$, or, symmetrically, $4$ with $5$. Keep arrows distinguished by their originating triangles and perform the same deletions. For $2=3$, the two local portions satisfy
    \[
        F([ac]d^\star a^\star)=ace,
        \qquad F([bd]c^\star b^\star)=bdf.
    \]
    These follow from \eqref{eq: four identities in case 2.1}. Adjoining the unchanged portions checks the required puncture generators. The new triangle generators again have trivial image under $F$. The case $4=5$ is analogous.

    If the union is instead a punctured digon and the middle puncture is $\coi$-colored, then the mutation is inverse to that of \emph{Case 1.1}. So this situation is resolved because mutation is an involution (\Cref{prop: mutation involutive}). If the puncture is $\coii$-colored, the flip is as in \Cref{fig: flip in II-colored digon}.
    
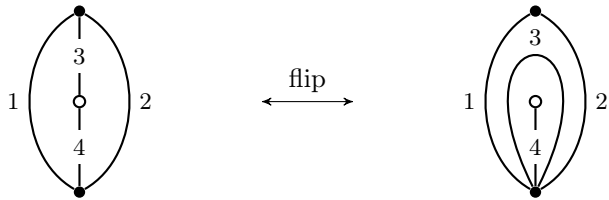
    \begin{figure}[ht]
    \centering
        \begin{tikzpicture}[scale=1.1]
            \begin{scope}[shift={(0,0)}]
                \node[generic] (B) at (0,0) {};
                \node[coii] (M) at (0,1) {};
                \node[right] at (M) {$q$};
                \node[generic] (T) at (0,2) {};

                \draw[thick] (B) to node[fill=white, midway]{\small{4}} (M);
                \draw[thick] (M) to node[fill=white, midway]{\small{3}}(T);
                \draw[thick, bend left=60] (B) to node[midway, left]{\small{1}} (T);
                \draw[thick, bend right=60] (B) to node[midway, right]{\small{2}} (T);
            \end{scope}

            \draw[<->] (2,1) to node[midway, above]{flip} (3,1);
            
            \begin{scope}[shift={(5,0)}]
                \node[generic] (B) at (0,0) {};
                \node[coii] (M) at (0,1) {};
                \node[above] at (M) {$q$};
                \node[generic] (T) at (0,2) {};

                \draw[thick] (B) to node[fill=white, midway]{\small{4}} (M);
                \draw[thick] (B) .. controls (-1,2) and (1,2) .. node[above, midway]{\small{3}} (B);
                \draw[thick, bend left=60] (B) to node[midway, left]{\small{1}} (T);
                \draw[thick, bend right=60] (B) to node[midway, right]{\small{2}} (T);
            \end{scope}    
        \end{tikzpicture}
        \caption{Flip inside a $\coii$-punctured digon.}
        \label{fig: flip in II-colored digon}
    \end{figure}
    
    Then the mutation is as follows.
    
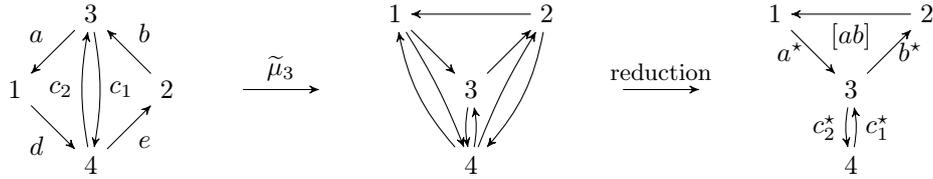
\[
        \begin{tikzpicture}
            \begin{scope}[shift={(0,0)}]
                \node (A1) at (0,1) {1};
                \node (A2) at (2,1) {2};
                \node (A3) at (1,2) {3};
                \node[label=below:{$Q(T)$}] (A4) at (1,0) {4};
                \draw[<-] (A1) to node[above left]{$a$} (A3);
                \draw[<-] (A3) to node[above right]{$b$} (A2);
                \draw[->] (A1) to node[below left]{$d$} (A4);
                \draw[->] (A4) to node[below right]{$e$} (A2);
                \draw[->, bend left=10] (A3) to node[right]{$c_1$} (A4);
                \draw[->, bend left=10] (A4) to node[left]{$c_2$} (A3);
            \end{scope}

            \draw[->] (3,1) -- node[above]{$\widetilde{\mu}_3$} (4,1);
            
            \begin{scope}[shift={(5,0)}]
                \node (A1) at (0,2) {1};
                \node (A2) at (2,2) {2};
                \node (A3) at (1,1) {3};
                \node[label=below:{$Q'$}] (A4) at (1,0) {4};
                \draw[<-] (A3) to (A1);
                \draw[<-] (A2) to (A3);
                \draw[->] (A2) to (A1);
                \draw[<-, bend left=10] (A3) to (A4);
                \draw[<-, bend left=10] (A4) to (A3);
                \draw[->, bend left=5] (A1) to (A4);
                \draw[->, bend left=15] (A4) to (A1);
                \draw[->, bend left=15] (A2) to (A4);
                \draw[->, bend left=5] (A4) to (A2);
            \end{scope}

            \draw[->] (8,1) -- node[above]{\small{reduction}} (9,1);
            
            \begin{scope}[shift={(10,0)}]
                \node (A1) at (0,2) {1};
                \node (A2) at (2,2) {2};
                \node (A3) at (1,1) {3};
                \node[label=below:{$Q^\dag=Q(T')$}] (A4) at (1,0) {4};
                \draw[<-] (A3) to node[left]{$a^\star$} (A1);
                \draw[<-] (A2) to node[right]{$b^\star$}(A3);
                \draw[->] (A2) to node[above]{$[ab]$}(A1);
                \draw[<-, bend left=10] (A3) to node[right]{$c_1^\star$} (A4);
                \draw[<-, bend left=10] (A4) to node[left]{$c_2^\star$} (A3);
            \end{scope}    
        \end{tikzpicture}
    \]
    The $3$-cycles $dac_2$ and $ec_1b$ are in $H$. So the $2$-cycles $d[ac_2]$ and $e[c_1b]$ are deleted.

    The three changed portions are
    \[
    \begin{aligned}
     F([ab])&=ab,\\
     F(b^\star c_1^\star c_2^\star a^\star)
       &=b^{-1}c_1^{-1}c_2^{-1}a^{-1}=ed,\\
     F(c_2^\star c_1^\star)&=(c_1c_2)^{-1}.
    \end{aligned}
    \]
Here the second equality uses $dac_2=ec_1b=1$. The first two lines
give the cycles at the two outer marked points, while the last line
gives $C_q(T')\mapsto C_q(T)^{-1}$. Thus $C_q(T')^2$ has trivial image; at either outer marked point that is a puncture, the corresponding generator $C_p(T')$ or $C_p(T')^2$ also has trivial image. The triangle in $Q(T')$ is the usual
mutation triangle, so its image is also trivial.

    If three arcs are shared, then the surface is either a sphere with three punctures (all of which are $\coii$-colored) or a torus with one puncture. The former situation is inverse to \emph{Case 2.10} while the latter is shown in \Cref{subsec: torus one mark}.

    \emph{Case 2.2.} $(P_1, P_2)$. See \Cref{fig: flip shared arc P1 P2} for the flip. Here and below, $\omega(i_0,\ldots,i_m)$ and $\omega'(i_0,\ldots,i_m)$ denote paths through the listed vertices in $Q(T)$ and $Q(T')$, respectively. When vertices are identified, distinguish arrows by their pieces and replace deleted arrows as in \Cref{rmk: def of Y(T)}.
    \[
        \begin{tikzpicture}
            \begin{scope}[shift={(0,0)}]
                \node[] (A1) at (-2,0) {1};
                \node[] (A2) at (0,0) {2};
                \node[] (A3) at (-1,1) {3};
                \node[] (A4) at (-1,-1) {4};
                \node[] (A5) at (1,-1) {5};
                \node[] (A6) at (1,1) {6};

                \draw[->] (A1) to node[above] {$a$} (A2);
                \draw[->] (A2) to node[above right] {$b$} (A3);
                \draw[->] (A2) to (A4);
                \draw[->] (A2) to node[below left] {$\ell$} (A5);
                \draw[->] (A3) to node[above left] {$c$} (A1);
                \draw[->] (A4) to (A1);
                \draw[->] (A6) to node[above left] {$g$} (A2);
                \draw[->] (A5) to node[right] {$h$} (A6);

                \node[label=below:{$Q(T)$}] at (0,-1) {};
            \end{scope}

            \draw[->] (1.5,0) to node[above]{$\tilde\mu_2$}(2.5,0);

            \begin{scope}[shift={(5,0)}]
                \node[] (A1) at (-2,0) {1};
                \node[] (A2) at (0,0) {2};
                \node[] (A3) at (-1,1) {3};
                \node[] (A4) at (-1,-1) {4};
                \node[] (A5) at (1,-1) {5};
                \node[] (A6) at (1,1) {6};

                \draw[<-] (A1) to (A2);
                \draw[<-] (A2) to (A3);
                \draw[<-] (A2) to (A4);
                \draw[<-] (A2) to (A5);
                \draw[->, bend left=10] (A3) to (A1);
                \draw[->, bend left=10] (A1) to (A3);
                \draw[->, bend left=10] (A4) to (A1);
                \draw[->, bend left=10] (A1) to (A4);
                \draw[<-] (A6) to (A2);
                \draw[->,bend right=10] (A6) to (A5);
                \draw[->,bend right=10] (A5) to (A6);
                \draw[->] (A1) to (A5);
                \draw[->] (A6) to (A3);
                \draw[->, bend left=15] (A6) to (A4);

                \node[label=below:{$Q'$}] at (0,-1) {};
            \end{scope}

            \draw[->] (6.5,0) to node[above]{\small{reduction}}(7.5,0);

            \begin{scope}[shift={(10,0)}]
                \node[] (A1) at (-2,0) {1};
                \node[] (A2) at (0,0) {2};
                \node[] (A3) at (-1,1) {3};
                \node[] (A4) at (-1,-1) {4};
                \node[] (A5) at (1,-1) {5};
                \node[] (A6) at (1,1) {6};

                \draw[<-] (A1) to (A2);
                \draw[<-] (A2) to (A3);
                \draw[<-] (A2) to (A4);
                \draw[<-] (A2) to (A5);
                \draw[<-] (A6) to (A2);
                \draw[->] (A1) to (A5);
                \draw[->] (A6) to (A3);
                \draw[->, bend left=15] (A6) to (A4);

                \node[label=below:{$Q^\dag=Q(T')$}] at (0,-1) {};
            \end{scope}  
        \end{tikzpicture}
    \]
    The three new triangle generators lie in $H^\dag$. Taking $3$ as the enclosing arc of the self-folded triangle, the old relations $cba=gh\ell=1$ give
    \[
    \begin{aligned}
        F\bigl(\omega'(1,5)\bigr)&=\ell a,\\
        F\bigl(\omega'(5,2,6)\bigr)&=h,\\
        F\bigl(\omega'(6,3,2,1)\bigr)&=cbg.
    \end{aligned}
    \]
    These are the old portions of the affected puncture cycles.

    For $6=1$, concatenate the first two substitutions at the common marked point; the third applies at the base of the self-folded triangle.

    For $5=1$, the puncture $p$ between arcs $1$ and $2$ becomes enclosed. If $p$ is $\coii$-colored, then
    \[
    \begin{aligned}
        F\bigl(\omega'(1,2,1)\bigr)&=(\ell a)^{-1},\\
        F\bigl(\omega'(6,3,2,1,2,6)\bigr)&=hcbg.
    \end{aligned}
    \]
    The first identity checks the square of the cycle at $p$; the second identifies the portion at the other marked point.

    If $p$ is $\coi$-colored, $a,\ell$ are deleted, with replacements $gh,cb$, respectively. The old relation $cbgh=1$ gives
    \[
    \begin{aligned}
        F\bigl(\omega'(6,3,2,6)\bigr)&=1,\\
        F\bigl(\omega'(6,3,1,6)\bigr)&=hcbg=1.
    \end{aligned}
    \]
    These check the two triangles through $3$, and the first path also replaces the old portion $hcbg$. Replacing $3$ by $4$ checks the other two triangles.

    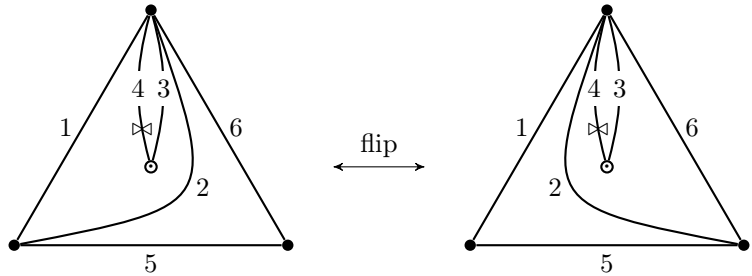
\begin{figure}[ht]
        \centering
        \begin{tikzpicture}[scale=0.6]
            \begin{scope}[shift={(0,0)}]
                \node[coi] (C) at (0,0) {};
                \node[generic] (L) at (-3,{-sqrt(3)}) {};
                \node[generic] (R) at (3,{-sqrt(3)}) {};
                \node[generic] (T) at (0,{sqrt(12)}) {};

                \draw[thick] (L) to node[left]{1} (T);
                \draw[thick] (T) to node[right]{6} (R);
                \draw[thick] (R) to node[below]{5} (L);

                \draw[thick, bend left=15] (T) to node[fill=white]{$3$}(C);
                \draw[thick, bend right=15] (T) to node[fill=white]{$4$} node[pos=0.8]{\small{$\bowtie$}}(C);
                \draw[thick] (T) .. controls ({3/2},{-sqrt(3)/2}) .. node[right, midway]{2} (L);
            \end{scope}

            \draw[<->] (4,0) -- node[above]{flip} (6,0);

            \begin{scope}[shift={(10,0)}]
                \node[coi] (C) at (0,0) {};
                \node[generic] (L) at (-3,{-sqrt(3)}) {};
                \node[generic] (R) at (3,{-sqrt(3)}) {};
                \node[generic] (T) at (0,{sqrt(12)}) {};

                \draw[thick] (L) to node[left]{1} (T);
                \draw[thick] (T) to node[right]{6} (R);
                \draw[thick] (R) to node[below]{5} (L);

                \draw[thick, bend left=15] (T) to node[fill=white]{$3$}(C);
                \draw[thick, bend right=15] (T) to node[fill=white]{$4$} node[pos=0.8]{\small{$\bowtie$}}(C);
                \draw[thick] (T) .. controls ({-3/2},{-sqrt(3)/2}) .. node[left, midway]{2} (R);
            \end{scope}
        \end{tikzpicture}
        \caption{Flip between $T$ and $T'$ at arc $2$.}
        \label{fig: flip shared arc P1 P2}
    \end{figure}

    \emph{Case 2.3.} $(P_1, P_3)$. The flip at the shared arc (labeled by $1$) results in two copies of $P_2$ glued along a common arc. The mutation of $(Q, H)$ is as follows.
    \[
        \begin{tikzpicture}
            \begin{scope}[shift={(0,0)}]
                \node (A1) at (0,0) {1};
                \node (A2) at (1,-1) {2};
                \node (A3) at (-1,1) {3};
                \node (A4) at (-1,-1) {4};
                \node (A5) at (1,1) {5};
                \node (A6) at (-1,2) {6};
                \node (A7) at (1,2) {7};
                
                \draw[->] (A1) -- (A4);
                \draw[->] (A1) -- (A5);
                \draw[->] (A2) -- (A1);
                \draw[->] (A3) -- (A1);
                \draw[->] (A4) -- (A3);
                \draw[->] (A4) -- (A2);
                \draw[->] (A5) -- (A3);
                \draw[->] (A5) -- (A2);
                \draw[->] (A6) -- (A1);
                \draw[->] (A1) -- (A7);
                \draw[->] (A7) -- (A6);

                \node[label=below:{$Q(T)$}] at (0,-1) {}; 
            \end{scope}

            \draw[->] (2,0) -- node[above]{${\widetilde\mu}_1$} (3,0);

            \begin{scope}[shift={(5,0)}]
                \node (A1) at (0,0) {1};
                \node (A2) at (1,-1) {2};
                \node (A3) at (-1,1) {3};
                \node (A4) at (-1,-1) {4};
                \node (A5) at (1,1) {5};
                \node (A6) at (-1,2) {6};
                \node (A7) at (1,2) {7};
                \draw[->] (A4) -- (A1);
                \draw[->] (A5) -- (A1);
                \draw[->] (A1) -- (A2);
                \draw[->] (A1) -- (A3);
                \draw[->, bend left=10] (A4) to (A3);
                \draw[->, bend left=10] (A3) to (A4);
                \draw[->, bend left=10] (A4) to (A2);
                \draw[->, bend left=10] (A2) to (A4);
                \draw[->, bend left=10] (A5) to (A3);
                \draw[->, bend left=10] (A3) to (A5);
                \draw[->, bend left=10] (A5) to (A2);
                \draw[->, bend left=10] (A2) to (A5);
                \draw[->] (A1) -- (A6);
                \draw[->] (A7) -- (A1);
                \draw[->, bend left=10] (A7) to (A6);
                \draw[->, bend left=10] (A6) to (A7);
                \draw[->, bend right] (A6) to (A4);
                \draw[->, bend right] (A2) to (A7);
                \draw[->] (A6) to (A5);
                \draw[->] (A3) to (A7);

                \node[label=below:{$Q'$}] at (0,-1) {};
            \end{scope}

            \draw[->] (7,0) -- node[above]{\small{reduction}} (8,0);

            \begin{scope}[shift={(10,0)}]
                \node (A1) at (0,0) {1};
                \node (A2) at (1,-1) {2};
                \node (A3) at (-1,1) {3};
                \node (A4) at (-1,-1) {4};
                \node (A5) at (1,1) {5};
                \node (A6) at (-1,2) {6};
                \node (A7) at (1,2) {7};
                \draw[->] (A4) -- (A1);
                \draw[->] (A5) -- (A1);
                \draw[->] (A1) -- (A2);
                \draw[->] (A1) -- (A3);
                \draw[->] (A1) -- (A6);
                \draw[->] (A7) -- (A1);
                \draw[->, bend right] (A6) to (A4);
                \draw[->, bend right] (A2) to (A7);
                \draw[->] (A6) to (A5);
                \draw[->] (A3) to (A7);

                \node[label=below:{$Q^\dag = Q(T')$}] at (0,-1) {};
            \end{scope}
        \end{tikzpicture}
    \]
    The four new triangle generators have the form $\alpha^\star\beta^\star[\beta\alpha]$, so they belong to $H^\dag$.

    There are two outer marked points. Their changed portions satisfy
\[
\begin{aligned}
 F\bigl(\omega'(7,1,6)\bigr)&=\omega(7,6),\\
 F\bigl(\omega'(6,4,1,2,7)\bigr)
   &=\omega(6,1,4,2,1,7).
\end{aligned}
\]
The first equality uses the triangle $6\to1\to7\to6$.
For the second, both sides equal $\omega(6,1,7)$, since
$\omega(1,4,2,1)=1$. Thus the affected puncture cycles map to their old classes, up to conjugacy. The two enclosed $\coi$-punctures contribute no puncture generator.

    \emph{Case 2.4.} $(P_1, P_4)$. This case is inverse to a situation in the case $(P_1, P_1)$ where two arcs are shared and their common puncture is $\coii$-colored.
    
    \emph{Case 2.5.} $(P_2, P_2)$. This case is inverse to $(P_1, P_3)$ unless the two copies of $P_2$ glue to get a sphere with four punctures among which two are $\coi$-colored and two are $\coii$-colored. The remaining case is then inverse to a later resolved case $(P_3, P_4)$.

    \emph{Case 2.6.} $(P_2, P_3)$. The flip at the shared arc again gives a gluing of $P_2$ and $P_3$.
    \[
        \begin{tikzpicture}
            \begin{scope}[shift={(0,0)}]
                \node (A1) at (0,0) {1};
                \node (A2) at (1,-1) {2};
                \node (A3) at (-1,1) {3};
                \node (A4) at (-1,-1) {4};
                \node (A5) at (1,1) {5};
                \node (A6) at (-1,2) {6};
                \node (A7) at (1,2) {7};
                \node (A8) at (0,3) {8};
                \draw[->] (A1) -- (A4);
                \draw[->] (A1) -- (A5);
                \draw[->] (A2) -- (A1);
                \draw[->] (A3) -- (A1);
                \draw[->] (A4) -- (A3);
                \draw[->] (A4) -- (A2);
                \draw[->] (A5) -- (A3);
                \draw[->] (A5) -- (A2);
                \draw[->] (A6) -- (A1);
                \draw[->] (A1) -- (A7);
                \draw[->] (A7) -- (A6);
                \draw[->] (A7) -- (A8);
                \draw[->] (A8) -- (A1);
                \node at (0,-1.65) {$Q(T)$};
            \end{scope}

            \draw[->] (2,0) -- node[above]{${\widetilde\mu}_1$} (3,0);

            \begin{scope}[shift={(5,0)}]
                \node (A1) at (0,0) {1};
                \node (A2) at (1,-1) {2};
                \node (A3) at (-1,1) {3};
                \node (A4) at (-1,-1) {4};
                \node (A5) at (1,1) {5};
                \node (A6) at (-1,2) {6};
                \node (A7) at (1,2) {7};
                \node (A8) at (0,3) {8};
                \draw[->] (A4) -- (A1);
                \draw[->] (A5) -- (A1);
                \draw[->] (A1) -- (A2);
                \draw[->] (A1) -- (A3);
                \draw[->, bend left=10] (A4) to (A3);
                \draw[->, bend left=10] (A3) to (A4);
                \draw[->, bend left=10] (A4) to (A2);
                \draw[->, bend left=10] (A2) to (A4);
                \draw[->, bend left=10] (A5) to (A3);
                \draw[->, bend left=10] (A3) to (A5);
                \draw[->, bend left=10] (A5) to (A2);
                \draw[->, bend left=10] (A2) to (A5);
                \draw[->] (A1) -- (A6);
                \draw[->] (A7) -- (A1);
                \draw[->, bend left=10] (A7) to (A6);
                \draw[->, bend left=10] (A6) to (A7);
                \draw[->, bend right] (A6) to (A4);
                \draw[->, bend right] (A2) to (A7);
                \draw[->] (A6) to (A5);
                \draw[->] (A3) to (A7);
                \draw[->] (A1) -- (A8);
                \draw[->, bend left=10] (A7) to (A8);
                \draw[->, bend left=10] (A8) to (A7);
                \draw[->, bend right=60] (A8) to (A4);
                \draw[->] (A8) to (A5);
                \node at (0,-1.65) {$Q'$};
            \end{scope}

            \draw[->] (7,0) -- node[above]{\small{reduction}} (8,0);

            \begin{scope}[shift={(10,0)}]
                \node (A1) at (0,0) {1};
                \node (A2) at (1,-1) {2};
                \node (A3) at (-1,1) {3};
                \node (A4) at (-1,-1) {4};
                \node (A5) at (1,1) {5};
                \node (A6) at (-1,2) {6};
                \node (A7) at (1,2) {7};
                \node (A8) at (0,3) {8};
                \draw[->] (A4) -- (A1);
                \draw[->] (A5) -- (A1);
                \draw[->] (A1) -- (A2);
                \draw[->] (A1) -- (A3);
                \draw[->] (A1) -- (A6);
                \draw[->] (A7) -- (A1);
                \draw[->, bend right] (A6) to (A4);
                \draw[->, bend right] (A2) to (A7);
                \draw[->] (A6) to (A5);
                \draw[->] (A3) to (A7);
                \draw[->] (A1) -- (A8);
                \draw[->] (A8) to (A4);
                \draw[->] (A8) to (A5);
                \node at (0,-1.65) {$Q^\dag=Q(T')$};
            \end{scope}
        \end{tikzpicture}
    \]
    The six new triangle cycles belong to $H^\dag$. The only remaining
changed portion of a puncture cycle is at the common base point of
the three self-folded triangles. It satisfies
\[
 F\bigl(\omega'(7,1,6,4,1,2,7)\bigr)
   =1=\omega(7,6,1,4,2,1,7).
\]
On the left the reversed arrows and composite arrows cancel in pairs.
On the right use the old triangle cycles
$\omega(1,4,2,1)$ and $\omega(7,6,1,7)$.
If this base point is a puncture $q$, adjoining the unchanged portion
outside the patch gives $F(C_q(T'))=C_q(T)$ up to conjugacy.
    
    \emph{Case 2.7.} $(P_2, P_4)$. This is inverse to \emph{Case 2.2} with $5=1$ and $p$ $\coii$-colored.

    \emph{Case 2.8.} $(P_3,P_3)$. Two copies of $P_3$ glue to a sphere with five punctures, and the flip at their shared arc again gives two copies of $P_3$. The four enclosed punctures are $\coi$-colored, and the remaining puncture $p$ must also be $\coi$-colored by \Cref{def: coloring}. By \Cref{prop: fund grp surface}, every vertex group of $\pi(Q)/H$ is trivial. Hence $F$ sends every cyclic walk to the identity. In particular, the triangle generators and $C_p(T')$ belong to $H^\dag$; there are no generators $C_q(T')^2$ in this case. The following diagram gives the required quiver isomorphism.
    \[
        \begin{tikzpicture}
            \begin{scope}[shift={(0,0)}]
            \node (A1) at (0,0) {1};
            \node (A2) at (1,-1) {2};
            \node (A3) at (-1,1) {3};
            \node (A4) at (-1,-1) {4};
            \node (A5) at (1,1) {5};
            \node (A6) at (2,-1) {6};
            \node (A7) at (-2,1) {7};
            \node (A8) at (-2,-1) {8};
            \node (A9) at (2,1) {9};
            
            \draw[->] (A1) -- (A4);
            \draw[->] (A1) -- (A5);
            \draw[->] (A2) -- (A1);
            \draw[->] (A3) -- (A1);
            \draw[->] (A4) -- (A3);
            \draw[->] (A4) -- (A2);
            \draw[->] (A5) -- (A3);
            \draw[->] (A5) -- (A2);
            
            \draw[->] (A1) -- (A8);
            \draw[->] (A1) -- (A9);
            \draw[->] (A6) -- (A1);
            \draw[->] (A7) -- (A1);
            \draw[->] (A8) -- (A7);
            \draw[->, bend right=15] (A8) to (A6);
            \draw[->, bend right=15] (A9) to (A7);
            \draw[->] (A9) -- (A6);
            \node at (0,-2) {$Q(T)$};
            \end{scope}

            \draw[->] (3,0) -- node[above]{$\mu_1$} (4,0);

            \begin{scope}[shift={(7,0)}]
            \node (A1) at (0,0) {1};
            \node (A2) at (1,-1) {2};
            \node (A3) at (-1,1) {3};
            \node (A4) at (-1,-1) {4};
            \node (A5) at (1,1) {5};
            \node (A6) at (2,-1) {6};
            \node (A7) at (-2,1) {7};
            \node (A8) at (-2,-1) {8};
            \node (A9) at (2,1) {9};

            \draw[<-] (A1) -- (A4);
            \draw[<-] (A1) -- (A5);
            \draw[<-] (A2) -- (A1);
            \draw[<-] (A3) -- (A1);
            \draw[<-] (A1) -- (A8);
            \draw[<-] (A1) -- (A9);
            \draw[<-] (A6) -- (A1);
            \draw[<-] (A7) -- (A1);

            \draw[->] (A7) -- (A4);
            \draw[->, bend left=15] (A7) to (A5);
            \draw[->, bend left=15] (A6) to (A4);
            \draw[->] (A6) -- (A5);

            \draw[->] (A3) -- (A8);
            \draw[->, bend left=15] (A3) to (A9);
            \draw[->, bend left=15] (A2) to (A8);
            \draw[->] (A2) -- (A9);
            \node at (0,-2) {$Q^\dag=Q(T')$};
            \end{scope}
        \end{tikzpicture}
    \]

    \emph{Case 2.9.} $(P_3, P_4)$. The gluing is necessarily a sphere with two $\coi$-punctures and two $\coii$-punctures depicted in \Cref{fig: P3 and P4}. The leftmost and rightmost arcs are glued.
    \begin{figure}[ht]
        \centering
        \begin{tikzpicture}
            \begin{scope}[shift={(0,0)}]
                \draw[thick] (0,2) circle[radius=2];
                \draw[thick] (0,1.5) circle [radius=1.5];
                
                \node[coii] (B) at (0, 0) {};
                \node[coi] (M1) at (-0.75,1.5) {};
                \node[coi] (M2) at (0.75,1.5) {};
                \node[coii] (T) at (0,4) {};

                \draw[thick, bend left=20] (B) to node[pos=0.8]{$\bowtie$} (M1);
                \draw[thick, bend right=20] (B) to (M1);
                \draw[thick, bend left=20] (B) to (M2);
                \draw[thick, bend right=20] (B) to node[pos=0.8]{$\bowtie$} (M2);
            \end{scope}

            \draw[<->] (3.5,2) to node[above]{\small{flip}} (4.5,2);

            \begin{scope}[shift={(8,0)}]
                \draw[thick] (0,2) circle[radius=2];
                
                \node[coii] (B) at (0, 0) {};
                \node[coi] (M1) at (-0.75,1.5) {};
                \node[coi] (M2) at (0.75,1.5) {};
                \node[coii] (T) at (0,4) {};

                \draw[thick, bend left=20] (B) to node[pos=0.8]{$\bowtie$} (M1);
                \draw[thick, bend right=20] (B) to (M1);
                \draw[thick, bend left=20] (B) to (M2);
                \draw[thick, bend right=20] (B) to node[pos=0.8]{$\bowtie$} (M2);
                \draw[thick] (B) to (T);
            \end{scope}
        \end{tikzpicture}
        \caption{A gluing of $P_3$ and $P_4$ and its flip.}
        \label{fig: P3 and P4}
    \end{figure}
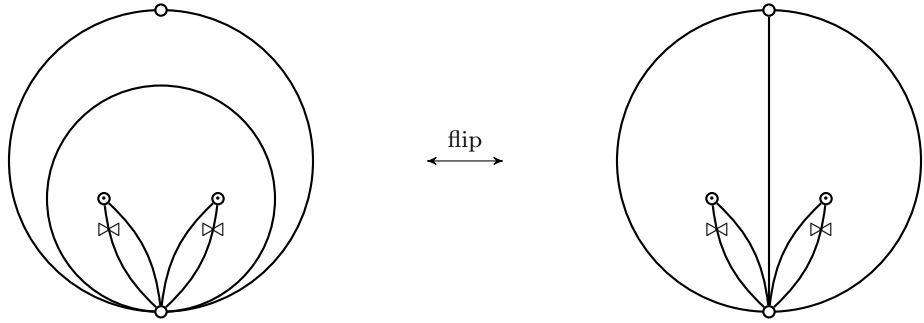
    The mutation is as follows.
    \[
        \begin{tikzpicture}
            \begin{scope}[shift={(0,0)}]
                \node (A1) at (0,0) {1};
                \node (A2) at (1,-1) {2};
                \node (A3) at (-1,1) {3};
                \node (A4) at (-1,-1) {4};
                \node (A5) at (1,1) {5};
                \node (A6) at (0,2) {6};
                
                \draw[->] (A1) -- (A4);
                \draw[->] (A1) -- (A5);
                \draw[->] (A2) -- (A1);
                \draw[->] (A3) -- (A1);
                \draw[->] (A4) -- (A3);
                \draw[->] (A4) -- (A2);
                \draw[->] (A5) -- (A3);
                \draw[->] (A5) -- (A2);
                \draw[->, bend left=10] (A6) to (A1);
                \draw[->, bend left=10] (A1) to (A6);
                \node at (0,-1.65) {$Q(T)$};
            \end{scope}

            \draw[->] (2,0) -- node[above]{${\widetilde\mu}_1$} (3,0);

            \begin{scope}[shift={(5,0)}]
                \node (A1) at (0,0) {1};
                \node (A2) at (1,-1) {2};
                \node (A3) at (-1,1) {3};
                \node (A4) at (-1,-1) {4};
                \node (A5) at (1,1) {5};
                \node (A6) at (0,2) {6};
                
                \draw[->] (A4) -- (A1);
                \draw[->] (A5) -- (A1);
                \draw[->] (A1) -- (A2);
                \draw[->] (A1) -- (A3);
                \draw[->, bend left=10] (A4) to (A3);
                \draw[->, bend left=10] (A3) to (A4);
                \draw[->, bend left=10] (A4) to (A2);
                \draw[->, bend left=10] (A2) to (A4);
                \draw[->, bend left=10] (A5) to (A3);
                \draw[->, bend left=10] (A3) to (A5);
                \draw[->, bend left=10] (A5) to (A2);
                \draw[->, bend left=10] (A2) to (A5);
                \draw[<-, bend left=10] (A1) to (A6);
                \draw[<-, bend left=10] (A6) to (A1);

                \draw[->, bend right=60] (A6) to (A4);
                \draw[->, bend right=60] (A2) to (A6);
                \draw[->] (A6) to (A5);
                \draw[->] (A3) to (A6);
                \node at (0,-1.65) {$Q'$};
            \end{scope}

            \draw[->] (7,0) -- node[above]{\small{reduction}} (8,0);

            \begin{scope}[shift={(10,0)}]
                \node (A1) at (0,0) {1};
                \node (A2) at (1,-1) {2};
                \node (A3) at (-1,1) {3};
                \node (A4) at (-1,-1) {4};
                \node (A5) at (1,1) {5};
                \node (A6) at (0,2) {6};
                
                \draw[->] (A4) -- (A1);
                \draw[->] (A5) -- (A1);
                \draw[->] (A1) -- (A2);
                \draw[->] (A1) -- (A3);
                \draw[<-, bend left=10] (A1) to (A6);
                \draw[<-, bend left=10] (A6) to (A1);
                \draw[->, bend right=60] (A6) to (A4);
                \draw[->, bend right=60] (A2) to (A6);
                \draw[->] (A6) to (A5);
                \draw[->] (A3) to (A6);
                \node at (0,-1.65) {$Q^\dag=Q(T')$};
            \end{scope}
        \end{tikzpicture}
    \]
    The diagram gives $Q^\dag=Q(T')$. Each of its four triangle generators has the form $\alpha^\star\beta^\star[\beta\alpha]$ and is sent to the identity by $F$, by \Cref{def: pre-mutation}. Let $q,r$ be the two $\coii$-punctures. By \Cref{prop: fund grp surface}, the vertex groups of $\pi(Q)/H$ are isomorphic to
    \[
        \pi_1(S^2\setminus\{q,r\})/\ncl{\gamma_q^2,\gamma_r^2}
        \cong\mathbb Z/2.
    \]
    Thus $F(C_s(T')^2)=F(C_s(T'))^2=1$ for $s=q,r$. The two $\coi$-punctures remain enclosed and contribute no puncture generators. This checks every generator in $Y(T')$.

    \emph{Case 2.10.} $(P_4, P_4)$. The two puzzle pieces glue as a sphere with three punctures. In this case all punctures must be $\coii$-colored. The quiver $Q(T)$ is necessarily as follows.
    \[
        \begin{tikzpicture}
            \begin{scope}[shift={(0,0)}]
                \node[] (A1) at (-1,0) {1};
                \node[] (A2) at (0,0) {2};
                \node[] (A3) at (1,0) {3};

                \draw[->, bend left] (A1) to node[above]{$a_1$} (A2);
                \draw[->, bend left] (A2) to node[below]{$a_2$} (A1);
                \draw[->, bend left] (A2) to node[above]{$b_1$} (A3);
                \draw[->, bend left] (A3) to node[below]{$b_2$} (A2);
                \node at (0,-1.05) {$Q(T)$};
            \end{scope}

            \draw[->] (2,0) to node[above]{$\widetilde\mu_2 = \mu_2$} (3,0);

            \begin{scope}[shift={(5,0)}]
                \node[] (A1) at (-1,0) {1};
                \node[] (A2) at (0,0) {2};
                \node[] (A3) at (1,0) {3};

                \draw[<-, bend left=15] (A1) to (A2);
                \draw[<-, bend left=15] (A2) to (A1);
                \draw[<-, bend left=15] (A2) to (A3);
                \draw[<-, bend left=15] (A3) to (A2);
                \draw[->, bend left] (A1) to node[above]{$[b_1a_1]$} (A3);
                \draw[->, bend left] (A3) to node[below]{$[a_2b_2]$} (A1);
                \node at (0,-1.05) {$Q'=Q^\dag=Q(T')$};
            \end{scope}
        \end{tikzpicture}
    \]
    Put $x=a_1a_2$ and $y=b_2b_1$, based at vertex $2$. The three puncture cycles are represented, up to conjugation, by $x,y,xy$, so \Cref{rmk: def of Y(T)} gives
    \[
    \begin{gathered}
        H(T)=\ncl{x^2,y^2,(xy)^2},\\
        \pi(Q)(2,2)/H(T)(2,2)
        \cong\langle x,y\mid x^2=y^2=(xy)^2=1\rangle
        \cong(\mathbb Z/2)^2.
    \end{gathered}
    \]
    Since
    $\varphi([a_2b_2][b_1a_1])=a_2b_2b_1a_1=a_1^{-1}(xy)a_1$ and $xy\ne1$ in this quotient, the composite $2$-cycle is not deleted. The two triangle generators $a_1^\star b_1^\star[b_1a_1]$ and $b_2^\star a_2^\star[a_2b_2]$ are sent to the identity by \Cref{def: pre-mutation}. For each of the three punctures $p$, we also have $F(C_p(T')^2)=F(C_p(T'))^2=1$, because every element of the quotient vertex groups has square one. These are all the generators of $H(T')$, proving the required inclusion.
\end{proof}

\subsection{Example: once-punctured torus}\label{subsec: torus one mark}

Let $(\S, \M)$ be a torus with one puncture $p$. There is a triangulation $T$ (\Cref{fig: torus one mark}) with two triangles whose associated $Q(T)$ is a Markov quiver. It is actually the only triangulation up to diffeomorphisms of $\S$ fixing $p$. In a tagged triangulation, the ends of all arcs must be tagged in the same way and the way they are tagged is invariant under flips.

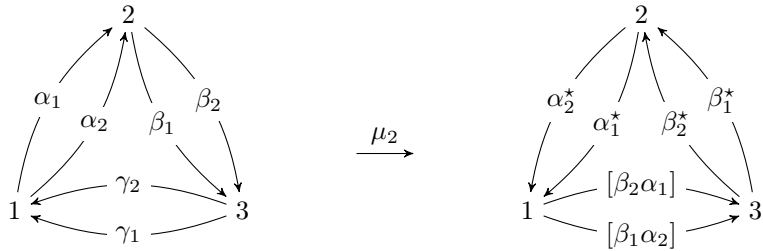
\begin{figure}
    \centering
    \begin{tikzpicture}[scale=2]
        \begin{scope}[shift={(0,0)}]
            \node[generic] (A) at (0,0) {};
            \node[generic] (B) at (1,0) {};
            \node[generic] (C) at (1,1) {};
            \node[generic] (D) at (0,1) {};

            \draw[thick] (0,0) -- node[below]{$1$} (1,0) -- node[right]{$3$} (1,1) -- node[above]{$1$} (0,1) -- node[left]{$3$} (0,0);
            \draw[thick] (0,1) -- node[above]{$2$}(1,0);
        \end{scope}

        \draw[<->] (2,0.5) -- node[above]{flip} (2.5,0.5);

        \begin{scope}[shift={(3.5,0)}]
            \node[generic] (A) at (0,0) {};
            \node[generic] (B) at (1,0) {};
            \node[generic] (C) at (1,1) {};
            \node[generic] (D) at (0,1) {};
            
            \draw[thick] (0,0) -- node[below]{$1$} (1,0) -- node[right]{$3$} (1,1) -- node[above]{$1$} (0,1) -- node[left]{$3$} (0,0);
            \draw[thick] (0,0) -- node[above]{$2$}(1,1);
        \end{scope}
    \end{tikzpicture}
    \caption{Once-punctured torus: triangulations $T$ and $T'$ related by a flip.}
    \label{fig: torus one mark}
\end{figure}
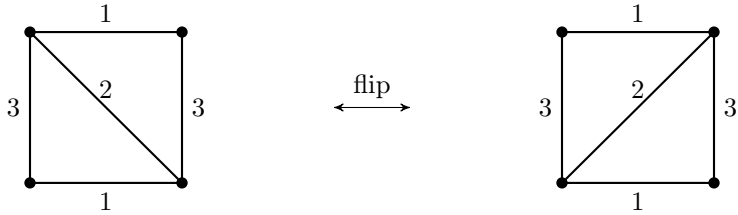

 \[
    \begin{tikzpicture}[->, >=stealth', scale=0.75]
        \begin{scope}[shift={(0,0)}]
        \node (1) at (-2,0) {1};
        \node (2) at (0,{sqrt(12)}) {2};
        \node (3) at (2,0) {3};
        \path
        (1) edge[bend right=20] node[fill=white] {$\alpha_2$} (2)
        (1) edge[bend left=20] node[fill=white] {$\alpha_1$} (2)
        (2) edge[bend right=20] node[fill=white] {$\beta_1$} (3)
        (2) edge[bend left=20] node[fill=white] {$\beta_2$} (3)
        (3) edge[bend left=20] node[fill=white] {$\gamma_1$} (1)
        (3) edge[bend right=20] node[fill=white] {$\gamma_2$} (1);
        \node at (0,-1.55) {$Q(T)$};
        \end{scope}

        \draw (4,1) to node[above]{$\mu_2$} (5,1);

        \begin{scope}[shift={(9,0)}]
        \node (1) at (-2,0) {1};
        \node (2) at (0,{sqrt(12)}) {2};
        \node (3) at (2,0) {3};
        \path
        (2) edge[bend right=20] node[fill=white] {$\alpha_2^\star$} (1)
        (2) edge[bend left=20] node[fill=white] {$\alpha_1^\star$} (1)
        (3) edge[bend right=20] node[fill=white] {$\beta_1^\star$} (2)
        (3) edge[bend left=20] node[fill=white] {$\beta_2^\star$} (2)
        (1) edge[bend left=20] node[fill=white] {$[\beta_2\alpha_1]$} (3)
        (1) edge[bend right=20] node[fill=white] {$[\beta_1\alpha_2]$} (3);
        \node at (0,-1.55) {$Q^\dag=Q(T')$};
        \end{scope}
    \end{tikzpicture}
\]
The homotopy $H(T)$ depends on $c(p)$. If $c(p) = \coi$, then $H(T)$ is generated by 
\[
    \{\gamma_1\beta_1\alpha_1, \gamma_2\beta_2\alpha_2, \gamma_1\beta_2\alpha_1\gamma_2\beta_1\alpha_2\}.
\]
In the pre-mutation $Q'$, delete $\gamma_1[\beta_1\alpha_1]$ and $\gamma_2[\beta_2\alpha_2]$, whose images under $\varphi$ are the old triangle relations. This gives the displayed quiver $Q^\dag=Q(T')$. The generators prescribed for $H(T')$ are
\[
    \{\alpha_1^\star\beta_2^\star[\beta_2\alpha_1], \alpha_2^\star\beta_1^\star[\beta_1\alpha_2], \alpha_1^\star\beta_1^\star[\beta_2\alpha_1]\alpha_2^\star\beta_2^\star[\beta_1\alpha_2]\}.
\]
If $c(p)=\coii$, replace the third generator in each of the two preceding displays by its square, as required by \Cref{def: homotopy on Q(T)} (3). The two triangle generators and the mutated quiver are unchanged. To check the puncture generator, apply the restriction $F:\pi(Q^\dag)\to\pi(Q(T))/H(T)$ of the functor in \Cref{def: pre-mutation}. The old triangle relations give $F(\alpha_i^\star\beta_i^\star)=\gamma_i$ for $i=1,2$, and $F([\beta_i\alpha_j])=\beta_i\alpha_j$. Hence the third word in the second display maps to the third word in the first. This also checks its square for a $\coii$-puncture. The new triangle generators are mutation triangles. Hence, in either coloring, all generators of $H(T')$ have trivial image. Step I of the proof of \Cref{thm: mutation and flip} now gives $\mu_2(Q(T),H(T))=(Q(T'),H(T'))$.

\section{2-cycle-allowed cluster algebras}\label{section: 2-cycle cluster algebras}
In this section, we define 2-cycle-allowed cluster algebras which extend the classical cluster algebras defined by Fomin and Zelevinsky \cite{FZI, FZIV}.

Let $(\mathbb P, \oplus, \cdot)$ be a semifield, $\mathbb{ZP}$ the group ring, $\mathbb{QP}$ the fraction field of $\mathbb {ZP}$, and $\mathcal F$ the field of rational functions in $n$ variables with coefficients in $\mathbb{QP}$. 
A  {\em  seed}  is a triple $\Sigma = (\mathbf{x}, \mathbf y, (Q, H))$, where
\begin{enumerate}
	\item[$\bullet$] $\mathbf{x} =  (x_1, \dots, x_n)$  is a free generating set of $\mathcal F$, and we call $\mathbf{x}$ a {\em cluster} and $x_1,\dots,x_n$ {\em cluster variables};
	\item[$\bullet$] $\mathbf y=(y_1,y_2,\dots,y_n)$ is an $n$-tuple of elements of $\mathbb P$, and we call $\mathbf y$ the {\em coefficient tuple}; 
	\item[$\bullet$] $(Q, H)$ is a quiver $Q$ with homotopy $H$  and $Q_0=\{1,2,\dots,n\}$. Note that $Q$ may have $2$-cycles.
\end{enumerate}

 For each quiver $Q$, the {\em adjacency matrix} $P=P(Q)$ of $Q$ is the matrix $(p_{ij})_{n\times n}$, where $p_{ij} = |i\rightarrow j|$ is the number of arrows from $i$ to $j$ in $Q$. For $k\in \{1,2,\dots, n\}$,  we define another seed $\Sigma'=(\mathbf x', \mathbf y', (Q', H'))$ which is called the {\em mutation of $\Sigma = (\mathbf x, \mathbf y, (Q, H))$
at $k$} and obtained from $\Sigma$ by the following rules:
\begin{enumerate}
    \item[$\bullet$] $\mathbf{x'} = (x'_1, x'_2,\dots, x'_n)$ is given by
                         \begin{equation} \label{mutation formula of clusters} x'_i =\begin{cases} x_i, &\text{$i\neq k$},\\  
	\dfrac{ y_k  \prod\limits_{j=1}^nx_j^{p_{jk}} +    \prod\limits_{j=1}^nx_j^{p_{kj}}         }{(y_k\oplus 1)x_k} ,&\text{$i=k$;}
\end{cases}  \end{equation} 
    \item[$\bullet$] The coefficient tuple $\mathbf y' = (y'_1, y'_2, \dots, y'_n)$ is given by 
     \begin{equation}   \label{Y-mutation fumula}               
        y'_j=\begin{cases} y_k^{-1}, &\text{$j= k$},\\  
  	y_j y_k^{p_{kj}}(y_k\oplus 1)^{p_{jk}-p_{kj}}        ,&\text{$j\neq k$.}
  \end{cases}
  \end{equation}
	\item[$\bullet$]   $(Q', H') = \mu_k(Q, H)$ is the mutation of $(Q, H)$ at $k$ following from Definition \ref{def: mutation}.  
\end{enumerate}

Denote $\Sigma' = \mu_k(\Sigma)$.  We also call $(\mathbf{y},(Q,H))$ a \emph{$Y$-seed}, and the formula (\ref{Y-mutation fumula}) is called the \emph{$Y$-mutation formula}.

\begin{remark}
    When $Q$ is $2$-acyclic, the skew-symmetric matrix $B=(b_{ij})_{n\times n}$ associated with $Q$ (as defined in \Cref{eq: exchange matrices of quivers}) is closely related to the adjacency matrix $P=(p_{ij})_{n\times n}$:
    \[  p_{ij} = \max\{b_{ij},0\} = \max\{-b_{ji},0\} , \quad b_{ij} = p_{ij}-p_{ji}, \quad \forall i,j\in Q_0, \quad \text{and}\;\; B= P-P^T.  \]
\end{remark}

Let $\mathbb T_n$ be the $n$-regular tree such that edges are labeled by the numbers $1,\dots, n,$ and the $n$ edges emanating
from each vertex receive different labels.
\begin{definition}
     A \emph{cluster pattern} is an assignment of a labeled seed $(\mathbf x_t, \mathbf y_t, (Q_t, H_t))$ to every vertex $t\in \mathbb T_n$ such that the seeds assigned to the endpoints of any edge $t\frac{k}{\quad}t'$ are obtained from each other by the mutation in direction $k$.  
\end{definition}

\begin{definition}
   The {\em $2$-cycle-allowed cluster algebra} $\mathcal A=\mathcal A(\Sigma)$ is the $\mathbb{ZP}$-subalgebra of $\mathcal F$ generated by all cluster variables obtained from $\Sigma$ by applying a finite sequence of mutations. 
\end{definition}

When $Q$ admits no $2$-cycles and $H(i,i)=\pi_1(Q,i)$ for all $i\in Q_0$, the seeds and mutations coincide with those defined by Fomin and Zelevinsky in \cite{FZI, FZIV}, and the $2$-cycle-allowed cluster algebras are the classical cluster algebras in the sense of Fomin and Zelevinsky.

We can similarly define the $Y$-pattern on an $n$-regular tree $\mathbb T_n$ as in \cite{FZIV}.
\begin{definition}
    A $Y$-pattern on $\mathbb T_n$ is an assignment of $Y$-seeds
    \[
        t\in \mathbb T_n \mapsto (\mathbf y_t, (Q_t, H_t))
    \]
    such that $t\mapsto (Q_t, H_t)$ is a pattern of quivers with homotopies and if $t'\frac{k}{\quad\quad} t$, then $\mathbf y_{t'}$ and $\mathbf y_t$ are related by the $Y$-mutation formula (\ref{Y-mutation fumula}).
\end{definition}

The following result shows that one can obtain a $Y$-pattern from a cluster pattern.

\begin{lemma}[\cf \cite{FZIV}, Proposition 3.9]\label{lemma: y hat}
	Let $(\mathbf{x}_t, \mathbf y_t, (Q_t, H_t))_{t\in \mathbb T_n}$ be a cluster pattern. For each $t\in \mathbb T_n$ and $j\in [1,n]$, define 
	\[    \widehat y_{j,t}  =  y_{j,t}\prod\limits_{i=1}^n  x_{i,t}^{p_{ij,t} - p_{ji,t}}.\]
	Then $(\widehat y_t, (Q_t, H_t))_{t\in \mathbb T_n}$ is a $Y$-pattern.
\end{lemma}

\begin{proof}
	Let $(\mathbf x', \mathbf y', (Q', H')) =  \mu_k(\mathbf x, \mathbf y, (Q, H))$. We prove 
	 \[\widehat y'_j=\begin{cases} \widehat y_k^{-1}, &\text{$j= k$},\\  
		\widehat y_j \widehat y_k^{p_{kj}}(\widehat y_k+1)^{p_{jk}-p_{kj}}        ,&\text{$j\neq k$.}
	\end{cases}\]
	
\begin{enumerate}
	\item[(I)] If $j=k$, then 
	                              \[
	                                 \widehat y'_k = y'_k \prod\limits_{i=1}^n {x'_i}^{p'_{ik}-p'_{ki}}  = y_k^{-1} \prod\limits_{i=1}^n ({x_i}^{p_{ik}-p_{ki}} )^{-1} = \widehat y_k^{-1}.
                                      \]
     \item[(II)]  Suppose that  $j\neq k$.     Note that 
                       \begin{equation}\label{xmut}  
                              	  x'_k = \frac{\widehat y_k+1}{y_k\oplus 1}  x_k^{-1} \prod\limits_{i=1}^n x_i^{p_{ki}},             
                       \end{equation}      
     we have that 
      \[   \begin{split}
         \widehat y'_j  &= y'_j \prod\limits_{i=1}^n (x'_i)^{p'_{ij} - p'_{ji}}  
      =  y'_j (x'_k)^{p_{jk}-p_{kj}} \prod\limits_{i\neq k}x_i^{p'_{ij} - p'_{ji}}       \\
                      &=  y_j y_k^{p_{kj}} (y_k\oplus 1)^{p_{jk} -p_{kj}} (\frac{\widehat y_k+1}{y_k\oplus 1}  x_k^{-1} \prod\limits_{i=1}^n x_i^{p_{ki}})^{p_{jk}-p_{kj}} \prod\limits_{i\neq k}x_i^{p'_{ij} - p'_{ji}}    \\
                      &= \widehat y_j(\widehat y_k+1)^{p_{jk} - p_{kj}}y_k^{p_{kj}} \prod\limits_{i\neq k}x_i^{p'_{ij} - p'_{ji}+p_{ji}-p_{ij}+p_{ki}(p_{jk}-p_{kj})}\\
                      &= \widehat y_j(\widehat y_k+1)^{p_{jk} - p_{kj}}y_k^{p_{kj}} \prod\limits_{i\neq k}x_i^{p_{kj}(p_{ik}-p_{ki})}\\
                      &= \widehat y_j \widehat y_k^{p_{kj}}(\widehat y_k+1)^{p_{jk}-p_{kj}} .
  \end{split}                      \]              
Here $p'_{ij}-p'_{ji} = p_{ij}-p_{ji}+p_{ik}p_{kj} -p_{jk}p_{ki}$ follows from the mutation rules.
\end{enumerate}	
 \end{proof}

If $\mathbb P = \mathbb Q_{sf}(y_1,\dots, y_n)$ is the {\em universal semifield}, we denote the coefficient by $Y_{i,t}$. Clearly, for the coefficient $y_{i,t}$ from an arbitrary semifield $\mathbb P$, we have that 
\[    y_{i,t}  =  Y_{i,t}|_{\mathbb P}(y_1,\dots,y_n),\;\; \text{and}\;\; \widehat y_{i,t} =  Y_{i,t}|_{\mathcal F}(\widehat y_1,\dots,\widehat y_n).   \]

Let $\mathrm{Trop}(u_j, j=1,2,\dots,m)$ be the multiplicative abelian group freely generated by $u_j (j=1,2,\dots,m)$. The addition $\oplus$ in $\mathrm{Trop}(u_j,j =1,2,\dots,m)$ given by \[\prod u_j^{a_j} \oplus \prod u_j^{b_j} := \prod u_j^{\min\{a_j,b_j\}}\]   
endows a semifield structure on $\mathrm{Trop}(u_j, j=1,2,\dots,m)$, and we call it a \emph{tropical semifield}. 
We say the corresponding $2$-cycle-allowed cluster algebra has \emph{principal coefficients at $t_0$} and we denote the corresponding cluster variable by $X_{i,t}$ if $\mathbb P = \mathrm{Trop}(y_1,\dots,y_n)$ is the tropical semifield generated by $n$ variables and $\mathbf y_{t_0}=(y_1,\dots,y_n)$. In this case, we define $F_{i,t}\in \mathbb Q_{sf}(y_1,\dots,y_n)$ to be the subtraction-free function obtained from $X_{i,t}$ by specializing the initial cluster variables $x_1,\dots,x_n$ to $1$.

\begin{lemma}[\cf \cite{FZIV}, Proposition 3.12]
	For each  $t\frac{k}{\quad}t'$, we have that 
	\begin{equation}\label{Xxmut}
		    X_{k, t'}  = \frac{(Y_{k,t}+1) (\widehat y_1,\dots,\widehat y_n)}{(Y_{k,t}+1)|_{\mathrm{Trop}(y_1,\dots,y_n)}(y_1,\dots,y_n)} X_{k,t}^{-1} \prod\limits_{i=1}^n X_{i,t}^{p_{ki,t}},  
	\end{equation}             
and 
\begin{equation} \label{Fmut}       F_{k,t'}= \frac{ Y_{k,t}+1}{(Y_{k,t}+1)|_{\mathrm{Trop}(y_1,\dots,y_n)}}  F_{k,t}^{-1} \prod\limits_{i=1}^n F_{i,t}^{p_{ki,t}} .      \end{equation}
\end{lemma}

\begin{proof}
    Applying \Cref{xmut}, we obtain \Cref{Xxmut}. Applying \Cref{xmut} to the $2$-cycle-allowed cluster algebra with principal coefficients under the specialization $x_1 =\dots=x_n=1$, we obtain \Cref{Fmut}.
\end{proof}

The following result gives the so-called \emph{separation formula}.

\begin{proposition}[\cf \cite{FZIV}, Theorem 3.7]\label{separation} 
	Let $(\mathbf{x}_t,\mathbf{y}_t, (Q_t,H_t))_{t\in \mathbb T_n}$ be a cluster pattern with coefficients from an arbitrary semifield $\mathbb P$.  Then 
	\[       x_{i,t}  =  \frac{X_{i,t}|_{\mathcal F}(x_1,\dots,x_n;y_1,\dots,y_n)}{F_{i,t}|_{\mathbb P}(y_1,\dots,y_n)}.                     \]
\end{proposition}

\begin{proof} We prove the formula by induction on the distance of $t$ and $t_0$.  Assume that $t\frac{k}{\quad}t'$ is an edge of $\mathbb T_n$ and the formulae hold at $t$. By  (\ref{xmut}),  (\ref{Xxmut}) and (\ref{Fmut}), we have that 
\[     \begin{split}
x_{k,t'}  &= \frac{(Y_{k,t}+1)|_{\mathcal F}(\widehat y_1,\dots,\widehat y_n)}{(Y_{k,t}+1)|_{\mathbb P}(y_1,\dots,y_n)} x_{k,t}^{-1} \prod\limits_{i=1}^n x_{i,t}^{p_{ki,t}} \\
& =  \frac{(Y_{k,t}+1)|_{\mathcal F}(\widehat y_1,\dots,\widehat y_n)}{(Y_{k,t}+1)|_{\mathbb P}(y_1,\dots,y_n)}  \frac{(X_{k,t}|_{\mathcal F})^{-1} \prod_{i=1}^n (X_{i,t}|_{\mathcal F})^{p_{ki,t}}}{(F_{k,t}|_{\mathbb P})^{-1} \prod_{i=1}^n (F_{i,t}|_{\mathbb P})^{p_{ki,t}}}\\
&=   \frac{(Y_{k,t}+1)|_{\mathcal F}(\widehat y_1,\dots,\widehat y_n)}{(Y_{k,t}+1)|_{\mathbb P}(y_1,\dots,y_n)}  \frac{(Y_{k,t}+1)|_{\mathrm{Trop(y_1,\dots,y_n)}}X_{k,t'}|_{\mathcal F}}{(Y_{k,t}+1)|_{\mathcal F}(\widehat y_1,\dots,\widehat y_n)} \frac{(Y_{k,t}+1)|_{\mathbb P}(y_1,\dots,y_n) }{ (Y_{k,t}+1)|_{\mathrm{Trop(y_1,\dots,y_n)}}F_{k,t'}|_{\mathbb P}}\\
&= \frac{X_{k,t'}|_{\mathcal F}(x_1,\dots,x_n;y_1,\dots,y_n)}{F_{k,t'}|_{\mathbb P}(y_1,\dots,y_n)} .
\end{split}     \]
This completes the proof.
	\end{proof}

A fundamental result in the theory of classical cluster algebras is the \emph{Laurent phenomenon}, which asserts that any cluster variable is a Laurent polynomial in the variables of any given cluster. It remains unclear to us whether this property extends to cluster algebras that allow $2$-cycles. The standard proofs of the Laurent phenomenon (for example those in the classical setting \cite{FZI} and in the Laurent phenomenon algebras introduced by Lam and Pylyavskyy \cite{LamPyly}) rely on the assumption that no cluster variable divides its corresponding exchange polynomial, which fails in the presence of 2-cycles.

However, when a quiver with homotopy $(Q, H)$ admits a global weakly admissible covering $p_H \colon Q_H \rightarrow Q$ (see \eqref{eq: galois correspond cover}), the Laurent phenomenon follows from that of the ordinary cluster algebra associated to $Q_H$. This includes, for example, the case where $H$ contains all squares of cyclic walks by \Cref{prop: global admissible pi square}.

\begin{theorem}\label{thm: laurent}
Let $\mathcal A = \mathcal A(\mathbf x, \mathbf y, (Q, H))$ be a $2$-cycle cluster algebra with coefficients from an arbitrary semifield such that $(Q, H)$ admits a globally weakly admissible covering. Then for any $i,t$,
\[   x_{i,t}\in \mathbb{Z}_{\geq 0}\mathbb{P}[x^{\pm 1}_1,x^{\pm 1}_2,\dots,x^{\pm 1}_n], \;\; \text{and}\;\; \mathcal A\subseteq \mathbb{ZP}[x^{\pm 1}_1,x^{\pm 1}_2,\dots,x^{\pm 1}_n].    \] 
Moreover, if $\mathcal A$ has principal coefficients at $t_0$,  then 
\[   X_{i,t}\in \mathbb{Z}_{\geq 0}[x^{\pm 1}_1,x^{\pm 1}_2,\dots,x^{\pm 1}_n; y_1,\dots,y_n], \;\; \text{and}\;\;   F_{i,t} \in \mathbb{Z}_{\geq 0}[y_1,\dots,y_n].    \] 
\end{theorem}

\begin{proof}
 Thanks to \Cref{separation}, we only need to prove the case for the principal coefficients. This can be obtained via Theorem 7 in \cite{Gra15} and \Cref{cor: global pattern}, as well as slightly adapting Lemma 4.4.8 and Corollary 4.4.11 in \cite{FWZ-II}.
\end{proof}

The following proposition allows us to define $g$-vectors. 
\begin{proposition}[\cf \cite{FZIV}, Proposition 6.1]
	Let $\mathcal A =  \mathcal A(\mathbf x,\mathbf y,(Q, H))$ be the cluster algebra with principal coefficients. 
	Then each cluster variable $X_{i,t}$ is homogeneous with respect to the $\mathbb Z^n$-grading in $\mathbb{Q}[x_1,\dots,x_n; y_1,\dots,y_n]$ given by 
	\begin{equation}
		\deg(x_i) = \mathbf{e}_i,\quad\text{and}\quad  \deg(y_j) = \begin{pmatrix}
			p_{j1}-p_{1j} \\
			\vdots \\
			p_{jn}-p_{nj}
		\end{pmatrix}
	\end{equation}
    ($\mathbf e_i$ being the $i$-th unit vector in $\mathbb Z^n$) in the sense that $X_{i,t}$ can be expressed as $U/V$ where $U$ and $V$ are homogeneous elements in $\mathbb{Q}[x_1,\dots,x_n; y_1,\dots,y_n]$.
\end{proposition}
\begin{proof}
    The induction proof in \cite{FZIV} applies here \emph{mutatis mutandis}.
\end{proof}

We define the degree of the cluster variable $X_{i,t}$ by $\mathbf g_{i,t}\coloneqq \deg U - \deg V$ and call it the {\em $g$-vector}. The $g$-vectors of cluster variables in the same cluster satisfy a remarkable \emph{sign coherence} \cite[Conjecture 6.13]{FZIV} in the ordinary case (proven in \cite{DWZ10, GHKK}). We are not able to determine whether this property extends to the current more general case with $2$-cycles.

\section{Final Remarks}\label{section: final remarks}

In this article, we have developed a mutation theory for loop-free quivers with oriented $2$-cycles, based on the notion of a \emph{homotopy}, i.e., a normal subgroupoid of the quiver's fundamental groupoid.  
This framework:
\begin{enumerate}
    \item generalizes the Fomin--Zelevinsky mutation of $2$-acyclic quivers (\Cref{def: mutation});
    \item recovers orbit mutations arising from weakly admissible coverings, and extends them to cases where further orbit mutations are obstructed (\Cref{prop: orbit mutation compatible});
    \item is involutive and preserves the quotient groupoid (\Cref{prop: mutation involutive,rmk: quotient groupoid invariant}).
\end{enumerate}

We constructed quivers with homotopies from triangulations of marked bordered surfaces with colored punctures and proved a flip--mutation correspondence (\Cref{thm: mutation and flip}) that extends the Fomin--Shapiro--Thurston model \cite{FST08} to the $2$-cycle setting and parallels Labardini-Fragoso’s theorem \cite[Theorem 30]{LF09} on quivers with potentials.

Based on the involutive property of mutations of quivers with homotopies, we also defined $2$-cycle-allowed cluster algebras $\mathcal{A}(Q,H)$ and established the Laurent phenomenon for those arising from weakly admissible coverings admitting infinite orbit mutations (\Cref{thm: laurent}).

The following problems remain open.

\begin{problem}\label{prob: laurent}
    Does the Laurent phenomenon hold for all $2$-cycle-allowed cluster algebras $\mathcal A(Q,H)$? \Cref{thm: laurent} settles the case arising from weakly admissible coverings with infinite orbit mutations; the general case may require new techniques beyond \cite{FZI} and \cite{LamPyly}. Computational evidence supports the conjecture that the answer to \Cref{prob: laurent} is affirmative.
\end{problem}

\begin{problem}
    Investigate the relationship between homotopy-based mutations and Derksen--Weyman--Zelevinsky (DWZ) mutations of quivers with potentials~\cite{DWZ08}, aiming to identify new classes of non-degenerate potentials realizing homotopy-based mutations in the presence of $2$-cycles.
\end{problem}

In the case of acyclic quivers, we have shown that any homotopy gives rise to the Fomin--Zelevinsky mutation (\Cref{thm: any homotopy induces fz mutation acyclic}). In comparison, one finds a related phenomenon for potentials: the only potential is the zero potential, which is rigid and non-degenerate \cite{DWZ08}.

\begin{problem}
    Find categorical realizations of $\mathcal A(Q,H)$, for instance in the representation theory of Jacobian algebras \cite{DWZ10} whose quivers contain $2$-cycles.
\end{problem}

For $2$-acyclic quivers, categorification relies on the Caldero--Chapoton formula \cite{CC} that expresses the Laurent expansion of a cluster variable using quiver representations. Establishing an analogue in the $2$-cycle setting would be a major step towards understanding $\mathcal A(Q, H)$ and resolving \Cref{prob: laurent}.

\section*{Acknowledgment} LM thanks Alastair King and Daniel Labardini-Fragoso for helpful conversations.
FL  is supported by the National Natural Science Foundation of China (No.\,12131015) for this project.
SL is supported by the National Natural Science Foundation of China (No.\,12301048 and No.\,12471023)  and the Zhejiang Provincial Natural Science Foundation of China (No. LQ24A010007).

\providecommand{\bysame}{\leavevmode\hbox to3em{\hrulefill}\thinspace}
\providecommand{\MR}{\relax\ifhmode\unskip\space\fi MR }
\providecommand{\MRhref}[2]{%
  \href{http://www.ams.org/mathscinet-getitem?mr=#1}{#2}
}
\providecommand{\href}[2]{#2}

\end{document}